\documentclass[11pt]{amsart}
\usepackage{geometry}
\usepackage{amsmath}
\usepackage{amssymb,amsfonts,latexsym}

\usepackage{verbatim}
\usepackage[all]{xy}
\usepackage{xcolor}
\usepackage{comment}
\usepackage{mathrsfs}
\usepackage{url}
\usepackage{enumitem}

\newcommand\AB[1] {\textcolor{red}{{#1}}}

\newtheorem{thm}{Theorem}[section]
\newtheorem{cor}[thm]{Corollary}
\newtheorem{lem}[thm]{Lemma}
\newtheorem{prop}[thm]{Proposition}
\theoremstyle{definition}

\theoremstyle{remark}
\newtheorem{rem}[thm]{Remark}

\newtheorem{exam}[thm]{Example}

\numberwithin{equation}{section}

\usepackage{chngcntr}
\counterwithin{table}{section}

\newcommand{\eps}{\varepsilon}

\newcommand{\mQ}{\mathbb{Q}}
\newcommand{\mC}{\mathbb{C}}
\newcommand{\Z}{\mathbb{Z}}

\newcommand{\mP}{\mathbb{P}}

\newcommand{\F}{\mathbb{F}}

\newcommand{\mW}{\mathcal{W}}
\newcommand{\mX}{\mathcal{X}}
\newcommand{\mY}{\mathcal{Y}}
\newcommand{\mZ}{\mathcal{Z}}

\DeclareMathOperator\soc{soc}
\DeclareMathOperator\Red{Red}
\DeclareMathOperator\Spec{Spec}

\newcommand\gp{p}

\newcommand{\ra}{\rightarrow}

\DeclareMathOperator\Gal{Gal}
\newcommand\divides{\mid}

\newcommand\oline[1] {{\overline{#1}}}

\newcommand\Aut{{\operatorname{Aut}}}

\newcommand\lcm{{\operatorname{lcm}}}

\newcommand\GL[1][n] {{\operatorname{GL}_{#1}}}

\DeclareMathOperator\Mon{Mon}

\DeclareMathOperator\PSL{PSL}
\DeclareMathOperator\PGL{PGL}
\DeclareMathOperator\AGL{AGL}

\DeclareMathOperator\rk{rk}

\begin{document}

\title{The Davenport--Lewis--Schinzel problem on the reducibility of $f(X)-g(Y)$}

\author{Angelot Behajaina}
\address{Univ. Lille, CNRS, UMR 8524, Laboratoire Paul Painlevé, F-59000 Lille, France}
\email{angelot.behajaina@univ-lille.fr}

\author{Joachim König}
\address{Department of Mathematics Education, Korea National University of Education, Cheongju, South Korea}
\email{jkoenig@knue.ac.kr}

\author{Danny Neftin}
\address{Department of Mathematics, Technion - Israel Institute of Technology, Haifa, Israel}
\email{dneftin@technion.ac.il}

\begin{abstract}
We solve the problem of Davenport--Lewis--Schinzel (DLS), originating in the 1950s, regarding the reducibility of $f(X)-g(Y)\in\mathbb C[X,Y]$. This yields an almost-complete solution to the Hilbert--Siegel problem: For a polynomial map $f$ whose composition factors avoid only very specific low-degree polynomials,  we  explicitly describe over which integers the  fibers of  $f$ are reducible.  
We further apply the solution to stability of iterates of $f$ in arithmetic dynamics, and to solving the functional equation $f(X)=g(Y)$ in  $X,Y\in\mC(z)$.
\end{abstract}
\maketitle

\section{Introduction}
\label{sec:intro}
Reducibility of polynomials is a central topic of interest in number theory, cf.\ \cite{Schinzel}. 
In a prominent paper  \cite{Schinzel1963}, Schinzel poses nine problems concerning reducibility of polynomials. 
As Zannier notes in \cite[Part E]{Sch}, the first three are ``substantial, involving several mathematical fields," owing to their relation to monodromy, combinatorial group theory, and the classification of finite simple groups (CFSG).
The first problem originates in the late 1950s \cite{Eren}, cf.\ \cite[Pg.\ 2]{Cas68}, and is first stated by   
Davenport--Lewis--Schinzel 
\footnote{The problem is sometimes attributed solely to Schinzel,  see e.g.\ 
\cite{FG, Fried12}. } 
in  \cite{DLS}. 
We henceforth refer to it as the DLS problem.  
The second problem concerning the irreducibility of $(f(X)-f(Y))/(X-Y)$ was solved by Fried  \cite[Thm.\ 1]{Fried-Schur},
while the third problem concerning the reducibility of separated polynomials $f(X_1,\ldots,X_m)-g(Y_1,\ldots,Y_n)\in\mQ[X_1,\ldots,Y_n]$ was reduced to the DLS problem \cite[Thm.\ 2]{DS}. 
As we shall see below, the DLS problem arises naturally in several topics including Hilbert's irreducibility, 
low degree points in fibers, 
stability in arithmetic dynamics, functional equations, Kronecker equivalence \cite{Mul}, intersections of lemniscates \cite{Pak3},  
expanding polynomials  \cite{Tao12,Tao2}, and  sum-product estimates \cite[Pf.\  Thm.\ 6]{BT}. 

The DLS problem is deceptively easy to state: \\
\centerline{\it  ``For which polynomials $f,g\in\mC[X]$ of degree $\geq 2$, is  $f(X)-g(Y)\in \mC[X,Y]$  reducible?"
}
This is equivalent to the reducibility of the curve $f(X)=g(Y)$, as well as to the reducibility of the fiber product of the  maps $f,g:\mP^1_\mC\to\mP^1_\mC$. 
Trivially, (a)  $f(X)-f(Y)$ has a diagonal factor $X-Y$, and (b) if $f_1(X)-g_1(Y)\in\mC[X,Y]$ is reducible, then so is the substitution $f_1(f_2(X))-g_1(g_2(Y))$ for  $f_2,g_2\in\mC[X]\setminus\mC$ with $(\deg(f_2),\deg(g_2))\neq (1,1)$. A pair $f,g$ for which $f(X)-g(Y)\in\mC[X,Y]$ is reducible but is not of the form (b)  is called {\it minimally reducible}. Among such pairs, those not arising from (a) via linear substitutions are called {\it nontrivial}. The problem is then to classify the nontrivial minimally reducible pairs $f,g$. 

The first nontrivial pair $(f,g)=(T_4,-T_4)$ was given by Davenport, Lewis, and Schinzel  \cite{DLS}, where $T_n$ is the degree-$n$ Chebyshev polynomial satisfying $T_n(X+1/X)=X^n+1/X^n$. 
The close relation of the DLS problem with  monodromy groups of polynomials was then revealed by Cassels--Guy \cite{Cas68}, proving the existence of nontrivial pairs where $\deg(f)=\deg(g)$ is  $7$ or $11$. 
The degree-$7$ polynomials, for example,  arise from the inequivalent actions of their monodromy group $\Mon(f)=\PSL_3(2)$ on points and lines in the Fano plane, or equivalently, on lines and hyperplanes in $\mathbb F_2^3$, arising from an outer automorphism of $\PSL_3(2)$. 
Explicit such polynomials were known by Klein's work  \cite{Klein1879,Klein1879b}, and an extended family was given  by Birch~\cite[Pg.~14]{Cas68}. 
After  extensive work on the case where  $f,g\in\mC[X]$ are indecomposable polynomials  (i.e., cannot be written as a composition of two polynomials of degree $>1$), e.g.\ \cite{Schinzel1967,tv68}, Fried showed that the reducibility of  $f(X)-g(Y)\in\mC[X,Y]$  for such $f,g$ implies that $\deg(f)=\deg(g)$ is $7,11,13,15,21$, or $31$. 
This was  announced in \cite{Fri73}  contingent on the CFSG, and the  proof is given in \cite{Fried86a}.
As above, the source of all such examples lies in outer automorphisms of $\Mon(f)$.
The indecomposable polynomials of these degrees were then given by Cassou-Nogu\`es--Couveignes \cite{CC}. 
Throughout many of the above stages, it was speculated and sometimes conjectured that no further minimally reducible pairs exist. 
The potentially-existing minimally reducible pair of polynomials with $\geq 4$ branch points is also referred to as  a  ``Cassels monster" \cite[Pg.\ 15]{Cas68}.

For decomposable polynomials, the problem has so far remained largely open. 
Illustrating the difficulty of the decomposable case, Fried posed the $(2,3)$-problem 
\cite{Fried86a}, cf.\ \cite[Prob.\ 7.31]{Fried12},\cite[Pg.\ 15]{Fried87}: given an  elliptic curve $Y^2=p(X)$, where $p\in\mC[X]$ is simply-branched of degree $3$, is there  a substitution $Y\mapsto g(Y)$, $X\mapsto f(X)$ for $f,g\in\mC[X]\setminus\mC$ such that $g(Y)^2=p(f(X))$ is reducible over $\mC$?  More generally, for integers $m,n\geq 2$, not both $2$, the $(m,n)$-problem is stated similarly, replacing $Y^2$ and $p(X)$ by simply-branched polynomials of degrees $m,n$, resp., with disjoint branch loci. 
The decomposable case has been further studied including in  determining when  $f(X)-g(Y)\in\mC[X,Y]$ admits a quadratic factor by Bilu \cite{Bilu}; and  in studying when  $f(X)-(\alpha\circ f)(Y)$ is reducible for $\alpha\in \mC[X]$ of degree $1$ by Avanzi--Zannier \cite{AZ2} and 
Fried--Gusi\'c \cite{Fried12,FG}. 
Further references and context can be found on MathOverflow  \cite{Rur12,Tao12}.  

So far, all known minimally reducible pairs arise from the pair $(T_4,-T_4)$ 
or one of the above indecomposable pairs. 
On the other hand,  little was known about the existence of decomposable minimally reducible pairs: 
the $(2,3)$-problem has remained widely open, and it was even unclear whether 
minimally reducible pairs could be  of unbounded degree.

We prove there exist no further examples, thereby solving the DLS problem: 
\begin{thm}\label{thm:DLS-C}
Let $f,g\in \mathbb C[X]$ be  of degree $>1$. 
Then $f(X)-g(Y)$ is reducible in $\mC[X,Y]$  
if and only if one of the following holds for some $f_1,g_1,\mu\in \mC[X]\setminus \mC$ with $\deg(\mu)=1$:
\begin{enumerate}[leftmargin=*,label={(\arabic*)},ref=(\arabic*)]
\item    $f$ and $g$ have a common left composition  factor  $h\in \mC[X]$ of degree at least $2$, that is, $f = h\circ f_1$ and $g = h\circ g_1$;
\label{case:DLS-C_1}
\item $f=(\mu\circ h_1)\circ f_1$ and $g=(\mu\circ h_2)\circ g_1$, 
where $(h_1,h_2)$ is 
one of the pairs of polynomials of degrees 7,11,13,15,21, or 31  
given in \cite[\S 5]{CC};
\label{case:DLS-C_2}
\item $f=(\mu\circ T_{4})\circ f_1$ and $g=(\mu\circ (-T_{4}))\circ g_1$. 
\label{case:DLS-C_3}
\end{enumerate}
\end{thm}
The problem is solved more generally over an arbitrary field $k$ of characteristic $0$ in Theorem \ref{thm:DLS}, which specializes to Theorem \ref{thm:DLS-C} for $k=\mC$. In particular, this solves the $(m,n)$-problem for all relevant $m,n$, see Corollary \ref{rem:m-n}.  Since all minimally reducible pairs in Theorem \ref{thm:DLS-C} have three branch points, this also rules out the existence of a Cassels monster. 
Theorem \ref{thm:DLS-C} also proves   \cite[Conj.\ 7.29]{Fried12}, cf.\ \cite[\S 1.4]{FG}, concerning the reducibility of $f(X)-(\alpha\circ f)(Y)$ for $\alpha\in \mC[X]$ of degree $1$. It also extends \cite[Thm.\ 1.2]{KN} which  restricts to nonsolvable decomposition factors, see \S\ref{sec:methods} for more on their relation.  

We next discuss  consequences to some of the topics mentioned in the first paragraph. \\
{\it Reducible fibers and the Hilbert--Siegel problem.} 
For $f\in \mQ[X]$ of degree $d\geq 2$, Hilbert's Irreducibility Theorem (HIT)  asserts the existence of infinitely many $a\in \Z$ such that the fiber $f^{-1}(a)\subseteq\mathbb C$ is irreducible\footnote{Equivalently, $f^{-1}(a)$ is {irreducible}  if it is irreducible as a scheme over $\Spec(\mQ(a))$,  or simply if $f(X)-a\in \mQ(a)[X]$ is irreducible. } over $\mQ$, that is, $[\mQ(\alpha):\mQ]=d$ for any $\alpha\in f^{-1}(a)$. The Hilbert--Siegel problem \cite[Pg.\ 2]{Fried86b}, cf.\ \cite[\S 7.1.3]{Fried12},  asks to determine, up to a finite set, the set of integral exceptions for Hilbert's theorem: $$\Red_f(\Z) := \{a\in \Z\,|\,f^{-1}(a)\text{ is reducible over } \mQ\}.$$ 
The problem is closely related to the DLS problem: Indeed, as in the proof of HIT, the problem reduces to determining the values sets $g(\mQ)$ for which $g(\mQ)\cap \Red_f(\Z)$ is infinite,  
or equivalently  to determining when is the curve $f(X)=g(Y)$  reducible for a {\it Siegel function} $g\in\mQ(X)$, that is, a  rational function whose value set $g(\mQ)$ contains infinitely many integers. 
For a polynomial $g\in\mQ[X]$, this is equivalent to the reducibility of $f(X)-g(Y)\in\mQ[X,Y]$, that is, to  the DLS problem over $\mQ$. 

Clearly, $\Red_f(\Z)$ contains every integer in $f(\mQ)$, and furthermore every integer in $f_1(\mQ)$ for a decomposition $f=f_1\circ f_2$ in $\mQ[X]$ with $\deg(f_1)>1$. 
The problem is then to determine whether $\Red_f(\Z)\setminus \bigcup f_1(\mQ)$ is finite, where the union runs over all left composition factors $f_1$ of $f$ with $\deg(f_1)>1$.  
For indecomposable $f\in\mQ[X]$ of degree $>5$, the finiteness of $\Red_f(\Z)\setminus f(\mQ)$ was  shown in  \cite[Thm.\ 1.2]{Fried86b} by Fried, extending \cite{Fried74}. 
Degree-$5$ examples   $f\in\mQ[X]$ for which this set is infinite were constructed by D\`ebes--Fried \cite{DF}. 
Several variants of the problem have been since considered, see \S\ref{sec:related}. 
However, the original Hilbert--Siegel problem  has so far remained largely open for decomposable polynomials.  

We give a uniform approach to the DLS and Hilbert--Siegel problems, yielding  an almost-complete solution of the latter: 
\begin{thm}\label{thm:Hilbert}
Let $f\in \mathbb{Q}[X]\setminus \mathbb{Q}$ be a nonlinear polynomial such that $f$ does not factor through an indecomposable of degree $2$ or $4$. 
Then one of the following holds:
\begin{enumerate}[leftmargin=*,label={(\arabic*)},ref=(\arabic*)]
\item $\Red_f(\mathbb{Z})$ is the union of $\bigcup_{f_1} \left(f_1(\mQ)\cap \Z\right)$ with a finite set, where $f_1\in \mQ[X]$ runs through all (nonlinear) indecomposable left factors $f=f_1\circ h$, $h\in \mQ[X]$, of $f$. 
\label{case:Hilbert_1}
\item $f=f_1\circ h$ for a polynomial $f_1\in \mathbb{Q}[X]$ of degree $5$ belonging to the family  \cite[(1.8)]{DF}.
\label{case:Hilbert_2}
\end{enumerate}
\end{thm}
This is proved in \S\ref{sec:siegel} with an approach that mostly applies over general number fields, see Remark \ref{rem:limitations}(1).
The theorem substantially extends \cite[Thm.\ 1.1]{KN}, which applies only when all the indecomposable factors of $f$ admit nonsolvable monodromy, cf.\ \S\ref{sec:methods}. 
\subsubsection*{Stability in arithmetic dynamics} 
Stability of polynomials under iterates is a central topic in arithmetic dynamics  \cite[\S 19]{AD-Survey}. 
A polynomial $f\in\mQ[X]$ is called {\it stable} over $a\in \mQ$, if the fibers over $a$ of the $n$-fold iterates $f^{\circ n}:=f\circ\cdots\circ f$ are irreducible for all $n\in\mathbb N$. 
In particular, when the fiber over $a$ of $f^{\circ (n-1)}$ is irreducible whereas the one of $f^{\circ n}$ is not, $f^{\circ n}$ is called newly reducible over $a$; cf.\ \cite{Jones2012} and \cite{Jones2021} for investigations of this phenomenon.  It is natural to ask whether $f^{\circ n}$ can be newly reducible over infinitely many $a\in\Z$. 
%
In other words, when is $\Red_{f^{\circ n}}(\Z)\setminus\Red_{f^{\circ n-1}}(\Z)$ infinite? 
As a direct consequence of Theorem \ref{thm:Hilbert},  we obtain: 
\begin{cor}
Let $f\in\mQ[X]$ be a polynomial of degree $>1$ that does not factor through a polynomial of degree $2$ or $4$, and let $n\ge 2$. Then $\Red_{f^{\circ n}}(\Z)\setminus\Red_f(\Z)$ is finite. 
\end{cor}
Allowing quadratic factors, counterexamples with $n=2$ exist, the simplest example being $f=X^2$, due to the factorization $X^4+4Y^4 = (X^2 - 2XY + 2Y^2)(X^2 + 2XY + 2Y^2)$. 

\subsubsection*{Functional equations} 
For which polynomials $f,g\in\mC[X]$ of degree $\geq 2$, 
does the functional equation $f(X)=g(Y)$ have a solution in rational functions $X=X(z),Y=Y(z)\in\mC(z)$? 
The question admits several  variants and draws its motivation from number theory,  functional equations, Nevanlinna theory, and dynamical systems,  see \S\ref{sec:related}. 

Attempts to answer this question and its variants naturally divide into two cases according to the reducibility of $f(X)-g(Y)\in\mC[X,Y]$. 
However, with the exception of the cases  $g=c f$ for $c\in\mC^\times$,  
and cases where $\deg(f)\gg \deg(g)$, see 
\cite[Thm.\ 1.3]{Pak4},\cite[Thm.\ 1.2]{Pak5}, and \cite{Fried23}, little  appears in the literature on the reducible case of the question. 


The  combination of Theorem \ref{thm:DLS-C} with  \cite[Thm.\ 1]{AZ2} gives the following simple resolution of the reducible case. 
First, we dispose of  solutions that are not genuinely new. 
Clearly, if $f_1(X)=g_1(Y)$, for $f_1,g_1\in\mC[X]$, has a solution  $X(z),Y(z)\in \mC(z)\setminus \mC$, 
then this is also a solution to 
$(w\circ f_1)(X)=(w\circ g_1)(Y)$ for  $w\in\mC[X]$.    
Such solutions satisfy $q(X(z),Y(z))=0$ for an irreducible factor $q(X,Y)$ of $f_1(X)-g_1(Y)\in\mC[X,Y]$  (and not only of $f(X)-g(Y)$). Say an irreducible factor $q(X,Y)$ of $f(X)-g(Y)\in\mC[X,Y]$ is {\it right-reduced} if it is not a factor of $f_1(X)-g_1(Y)\in\mC[X,Y]$ for any  decomposition $f=w\circ f_1$, $g=w\circ g_1$ with $w,f_1,g_1\in \mC[X]\setminus \mC$ and $\deg(w)>1$. A solution is {\it reduced} if it annihilates a right-reduced factor. 
The goal is then to determine when  reduced solutions exist. 
\begin{cor}\label{cor:genus0}
Let $f,g\in \mC[X]\setminus \mC$ be polynomials. Suppose that $f(X)-g(Y)\in\mC[X,Y]$ is reducible and admits a reduced solution $f(X(z))=g(Y(z))$ in  $X(z),Y(z)\in \mC(z)\setminus \mC$. 
Then one of the following holds for some $\mu,\lambda_1,\lambda_2 
\in\mC[X]$ with $\deg (\mu)=\deg(\lambda_1)=\deg(\lambda_2)=1$:
\begin{enumerate}[leftmargin=*,label={(\arabic*)},ref=(\arabic*)]
\item 
$f= \mu\circ T_n \circ \lambda_1$ and $g=\mu\circ (-T_m)\circ \lambda_2$ for 
$m,n\geq 3$ with $\gcd(m,n)>2$;
\label{case:genus0_1}
\item 
$f= \mu\circ P_i \circ \lambda_1$ and $g= \mu\circ  P_i \circ \lambda_2$, for $i\in\{1,2,3\}$, where 
$$P_1(X)=X^a(X-1)^b, \ \gcd(a,b)=1,\ a+b\ge 4,$$ 
$$P_2(X)=X^3(X^2+5X+40),\text{ and }$$
$$P_3(X)=X(X+1)^3(X+a+3)^3, \ \text{where } a^2+a+2=0;$$
\label{case:genus0_2}
\item $f=\mu\circ h_1 \circ \lambda_1$ and $g= \mu\circ h_2 \circ \lambda_2$, where $\{h_1,h_2\}$ is  one of the pairs of degree-$7$  or degree-$13$ polynomials appearing in  \S 5.1 or \S5.3, respectively,  of \cite{CC}. 
\label{case:genus0_3}
\end{enumerate}
\end{cor}
The degree-$7$ and $13$ polynomials in (3) are given in Remark \ref{rem:genus0}. For these polynomials, and those in (1) and (2), we verify that $f(X)=g(Y)$ indeed has a reduced solution, 
yielding the converse statement. 

Note that over number fields $k$, classifications of $f,g\in k[X]$ for which the curve $f(X)=g(Y)$ admits an irreducible component of genus $\leq 1$ over $\mC$ were announced, cf.\ \cite{Red12},  in the irreducible case  \cite{Zieve1} and the reducible case  \cite{Zieve2}, but have yet to appear in the literature.  We  suspect that Theorem \ref{thm:DLS}, which applies over general fields, 
will be useful in developing a simple approach to our question and its variants over number fields $k$, and  would 
therefore have broad applications to the above-mentioned subjects. 

\subsection{On the proof}\label{sec:methods}
A major reason why the DLS problem was previously deemed inaccessible for arbitrary decomposable polynomials is the complexity of  monodromy groups of decomposable polynomials in comparison to  indecomposable ones. This is especially true for polynomials $f$ with solvable monodromy $\Mon(f)$, where  the ``largeness" of monodromy from  \cite{KNR24} does not apply, and no analogous constraints are  known.

Instead, in the solvable case, we observe that the minimal reducibility of a pair $f,g\in\mC[X]\setminus \mC$ imposes heavy constraints on their monodromy. More specifically, writing $f=h\circ f_r$ and $h=h_2\circ f_{r-1}$ for indecomposable $f_{r-1}, f_r\in\mathbb C[X]$, the kernel $K$ of the restriction map $\Mon(f)\to \Mon(h)$ naturally embeds into $\Gamma^d$, where $\Gamma:=\Mon(f_r)$ and $d=\deg(h)$. If $\deg(f_r)
>2$, we show that minimal reducibility forces $K$ to be diagonal, that is, its $d$ projections to $\Gamma$ are injective.  When $\deg(f_r)=2$, we show that $K$ embeds in $C_2\times C_2$. This is carried out in~\S\ref{sec:diagkernel}. 


In contrast, we obtain lower bounds on the rank of $K$ via the following two main ideas: \\
Firstly, we consider the monodromy group 
$\Gamma_2:=\Mon(f_{r-1}\circ f_r)$ of a two-step right-factor of $f$. 
We use ``largeness properties" for the kernel $K_2=\ker(\Gamma_2\to \Mon(f_{r-1}))$, developed in the  companion paper \cite{BKN}. If for example $\Mon(f_r)\cong C_q$ for a prime $q$, 
these show that 
$K_2$ contains an abelian subgroup of rank at least $d_{r-1}-1$, where $d_{r-1}=\deg(f_{r-1})$. 
We further apply these properties to obtain bounds on the image of the $2$-step kernel $\ker(\Mon(f)\to\Mon(h_2))$ in $K_2$ by exploiting the  representations of $\Mon(f_{r-1})$ over $\mathbb F_q$ obtained from its action on corresponding subgroups of $K_2\leq C_q^{d_{r-1}}$. 
Since the resulting bounds do not suffice when $\deg(f_{r-1}), \deg(f_r)$ are very small,   we  develop a $3$-step analogue of such an argument.  This is carried out in \S\ref{sec:lowerbounds}.

Secondly, we observe that the above lower bounds on $2$-step (or $3$-step) kernels also lead to meaningful lower bounds on the kernels $K$ of arbitrarily long compositions, contradicting the aforementioned diagonality requirements except in a few very concrete low-degree cases. 
Length-$3$ right-factors of degrees $8$ and $16$ then require a separate argument (Proposition \ref{prop:special-func}),  which uses the database of transitive groups of these degrees,   our  only essential  use of Magma \cite{Magma} through the above argument. 
This is the heart of the proof of the solvable case of the DLS problem, carried out in \S\ref{sec:solvable}. We expect such largeness arguments would be useful far beyond the problems discussed here.
Based on \cite{KN}, we give a much simpler argument for polynomials of nonsolvable monodromy in \S\ref{sec:DLS}, improving   \cite[\S 4]{KN}. Its combination with the solvable case yields the main Theorem~\ref{thm:DLS}. 

Most considerations also apply to the Hilbert--Siegel problem and are given along with the above. This problem amounts to determining the minimally reducible pairs  $f\in\mQ[X]$ and $g\in\mQ(X)$, where $g$ is a Siegel function (and not merely a polynomial). Our strategy breaks when  $f$ has factors of degree $2$ or $4$ (of solvable monodromy), 
see Remark \ref{rem:limitations}. Further ideas are required  to classify such examples. Finally, we derive the consequence to functional equations  in \S\ref{sec:functional}. 

The CFSG is used in the proof of Theorem \ref{thm:DLS} only in the very last step to assert that indecomposable $f_1,g_1\in \mC[X]\setminus \mC$ of nonsolvable monodromy with the same Galois closure, must be one of the aforementioned pairs of degrees $7,11,13,15,21$ or $31$, see Remark \ref{rem:CFSG}. 

%

\subsection{Related work}\label{sec:related} 
Several variants of the original Hilbert--Siegel problem have been considered. 
Firstly, it is natural to consider  degree-$d$ maps $f:X\to\mP^1_\mQ$  from an arbitrary (smooth projective) curve $X$ of positive genus. 
In the general case, where  $\Mon_\mQ(f)=S_d$, the set of reducible fibers $\Red_f(\Z)$ is in fact finite, see  M\"uller \cite{Mul4}, and \cite{Mul5} for other cases. 
For indecomposable $f\in \mQ[X]$ with $\deg(f)>20$,  the analogous set $\Red_f(\mQ)$, of rational values with reducible fibers, is the union of $f(\mQ)$ and a finite set, by the combination of theorems of M\"uller \cite{Mul2} and Guralnick--Shareshian \cite{GS}, see \cite[Thm.\ 5.4]{KN}. Similar results were shown for indecomposable rational functions $f\in\mQ(X)$ of sufficiently large degree \cite[Thm.\ 1.2]{PHDTali}, using \cite{NZ}. 
The question also recently arose in the context of algebraic points of fixed degree $d$ in fibers of maps $f:X\to\mP^1_\mQ$ over rational points, as considered  by Derickx--Rawson \cite{DR25}.
Over fields $k$ of positive characteristic, except for the classification of $f,g\in k[X]$ for which  $f(X)-g(Y)$ has a quadratic factor \cite{KMS}, little appears in the literature. 
The above variants of DLS and Hilbert--Siegel, involving more general rational maps, remain wide-open. 

As for the DLS problem \cite{DLS}, much of the motivation for studying the functional equation $f(X(z))=g(Y(z))$ comes from the question: for which $f,g\in \mQ[X]$  is the set $f(\mQ)\cap g(\mQ)$ 
infinite? The existence of a solution $X=X(z),Y=Y(z)$ for $f(X)=g(Y)$ is equivalent to the existence of an irreducible component of genus $0$ in a  curve\footnote{Equivalently, there exists  a (parametrization) map $\mP^1_\mC\to \mathcal C$ given by $z\mapsto (X(z),Y(z))$ on affine charts.}  $\mathcal C$ birational to $f(X)=g(Y)$, whereas if $f(\mQ)\cap g(\mQ)$ is infinite, then  $\mathcal C$ admits an irreducible component  of genus $\leq 1$  by Faltings' theorem. The pairs $f,g$ for which $f(\mathbb Z)\cap g(\mathbb Z)$ is infinite were determined by Bilu--Tichy \cite{BiTi}; For these, Siegel's theorem implies  that $\mathcal C$ has an irreducible component of genus $0$, and   further constraints on the number of preimages of infinity. 
Moreover, the solvability of the functional equation $f(X)=g(Y)$ in meromorphic (resp.\ entire) functions $X=X(z),Y=Y(z)$ is equivalent to $\mathcal C$ admitting an irreducible component of genus $\leq 1$ by Picard's theorem (resp.\ reduces to our question \cite{Pak} with rational $X(z),Y(z)$). Variants of the equation also arise in studying intersections of orbits in arithmetic dynamics,   see \cite[\S 7]{AD-Survey}. 
Further motivation, coming from other subjects, is described in \cite[Pg.\ 2]{DZ}. 
See also
\cite{AZ,Pak2,HT} for further details on this functional equation and its variants.

The DLS problem is also closely related to Davenport's problem concerning polynomials $f,g\in\mQ[X]$ with the same image $f(\mathbb Z/p)=g(\mathbb Z/p)$ for all but finitely many primes $p$. More specifically, this condition implies,   $f(X)-g(Y)\in \mQ[X,Y]$ has to be reducible \cite[\S 7]{MV}. 

\subsubsection*{Acknowledgments}
The first and third authors were supported by the Israel Science Foundation, grant no.~353/21. The first author is also grateful for the support of a Technion fellowship, of an Open University of
Israel post-doctoral fellowship. He also acknowledges the support of the CDP C2EMPI,
as well as the French State under the France-2030 programme, the University of Lille, the Initiative of Excellence of the University of Lille, the European Metropolis of Lille for their
funding and support of the R-CDP-24-004-C2EMPI project.
The second author was supported by the National
Research Foundation of Korea (NRF Basic Research Grant RS-2023-00239917).

\section{Basic setup and Preliminaries}
\subsection{Basic setup} Throughout the paper, $k$ is a field of characteristic $0$ and $\oline k$ its algebraic closure. 
A \emph{map} $f: \mathcal{X} \rightarrow \mathcal{Y}$ {\it over $k$} is a finite (dominant and generically unramified) morphism of (smooth irreducible projective) varieties defined over $k$. This induces a field extension $k(\mathcal{X})/k(\mathcal{Y})$ via the pullback $f^*: k(\mathcal{Y}) \to  k(\mathcal{X})$, given by $h\mapsto h \circ f$. The \emph{degree} ${\deg(f)}$ of $f$ is then defined as $[k(\mathcal{X}):k(\mathcal{Y})]$. We say $f$ is {\it indecomposable} if it is of degree $>1$ and is not a composition of two maps of degree $>1$. 
Say that the function field $\oline k(\mathcal{X})$ has \emph{genus} $g$ if the curve\footnote{Recall that smooth projective curves with isomorphic function fields are isomorphic \cite[II.6.8]{Har}.} $\mathcal{X}$ has genus $g$.

A rational function $f\in k(X)\setminus k$, then induces an $k(x)/k(t)$ of rational functions fields such that $f(x)=t$. The degree of $f$ is then  
$\max\{\deg(f_1),\deg(f_2)\}$, 
where $f=f_1/f_2$ for coprime $f_1,f_2\in k[X]$. 
We say $f,g \in k(X) \setminus k$ are \emph{linearly related over $k$}
if there exist $\mu,\nu \in k(X)$ of degree $1$ (also called \emph{linear fractionals}) such that $f=\mu \circ g \circ \nu$. 
We shall denote $k$-rational places of $k(t)$ by $t\mapsto a$ for $a\in \mP^1(k)=k\cup\{\infty\}$. 
\subsubsection*{Permutation groups} All group actions are left actions. Recall that a permutation group $G\leq S_n$ is {\it primitive} if it preserves no nontrivial block system. It is \emph{affine} if it has an elementary-abelian regular normal subgroup; equivalently, $G$ embeds as $V\le G\le \AGL(V)=V\rtimes \GL[](V)$ for a finite vector space $V$. 
For $G\leq S_n$, $H\leq S_m$, the wreath product $H\wr G$ is defined as $H^n\rtimes G$, where $G$ acts on $H^n$ by permuting the $n$ copies. It is naturally an imprimitive permutation group of degree $mn$. It also has a natural primitive action of degree $m^n$, the {\it product type} action. This is the action on $\{1,\ldots,m\}^n$ in which $G$ acts by permuting the coordinates and $H^n$ acts coordinatewise. We denote by $G=A.H$ a group extension of $H$ with kernel $A$ (which is not necessarily split). If $A$ is abelian, we will regard it as an $H$-module via the action of $H$ by conjugation in $G$. Also recall that the socle $\soc(G)$ is the subgroup of $G$ generated by minimal normal subgroups of $G$.    

\subsection{Monodromy groups}\label{sec:setup}
Let $f: \mathcal{X} \rightarrow \mathcal{Y}$ be a map of {degree $d$} defined over $k$. 
The \emph{monodromy group} 
$\Mon_k(f)=\Mon(f)$ of $f$ is $\Gal(\Omega/k(\mathcal{Y}))$, where $\Omega$ is the Galois closure of the extension $k(\mathcal{X})/k(\mathcal{Y})$. It is a permutation group of degree $d$, via the action on the generic fiber of $f$, or equivalently via the action on the roots of a minimal polynomial for $k(\mathcal{X})$ over $k(\mathcal{Y})$. If $f=f_1/f_2$ is a  rational map for coprime $f_1,f_2\in k[X]$, then $\Mon_k(f)$ is just the Galois group of $f_1(X)-tf_2(X)\in k(t)[X]$. We refer to a root $x$ of the last polynomial, for brevity as a root of $f(X)-t$ in an extension of $k(t)$. Recall that $f$ is indecomposabe if and only if $\Mon_k(f)$ is primitive. 

When $\mathcal{X}$ and $\mathcal{Y}$ are geometrically irreducible, letting $f_{\overline{k}}:\mathcal{X}\otimes_k \overline{k} \rightarrow \mathcal{Y}\otimes_k \overline{k}$ be the map induced by $f$ over $\oline k$, the \emph{geometric monodromy group} ${\rm Mon}_{\overline{k}}(f_{\overline{k}})=\Gal(\Omega\oline k/\oline k(\mathcal{X}))$ is isomorphic to the image of the action of the \'{e}tale fundamental group $\pi_1^{ \textrm{\'et}}(\mathcal{Y}\setminus {\rm Br}(f))$ on the fiber $f^{-1}(y_0)$ of a base point $y_0 \in   \mathcal{Y}(\oline k)$ over which $f$ is unramified, that is, the classical definition of monodromy. Here, ${\rm Br}(f) \subset \mathcal{Y}(\overline{k})$ denotes the branch locus of $f$. 

\subsubsection*{Polynomial monodromy} A polynomial $f \in k[X] \setminus k$ with cyclic geometric monodromy group is well known to be linearly related to $X^n$ over $\oline k$. 
Similarly, every  polynomial $f\in k[X]\setminus k$ with dihedral geometric monodromy group is linearly related over $\overline{k}$ to the Chebyshev polynomial of degree $n=\deg(f)$ \cite[Lemma\ 3.3]{MZ}, that is, the unique degree-$n$ polynomial $T_n$ for which $T_n(X+1/X)=X^n+1/X^n$. 
We shall also let $D_{n,\alpha}$ denote the $n$-th degree Dickson polynomial with parameter $\alpha$, defined via $D_{n,\alpha}(x+\frac{\alpha}{X}) = X^n+\left(\frac{\alpha}{X}\right)^n$.\footnote{Note that $D_{n,0}=X^n$, whereas for $\alpha\ne 0$, the polynomial $D_{n,\alpha}$ is linearly related to $T_n$ over $k(\sqrt{\alpha})$.}
Note that in the above examples of $f$, the group $\Mon_k(f)$ contains a regular cyclic normal subgroup $C_n$ of order $n$, and hence $\Mon_k(f)$ is  isomorphic (as a permutation group) to a subgroup of $\AGL_1(n):=\mathbb Z/n\rtimes (\mathbb Z/n)^\times$, that is, of the holomorph of $C_n$.

Monodromy groups of indecomposable polynomials were completely classified by M\"uller  \cite{Mul2}. While this result uses the CFSG, the following important partial result does not, cf., e.g., \cite[Theorem 2.1]{KN}.

\begin{prop}
\label{prop:indec_poly}
Let $f\in k[X]$ be an indecomposable polynomial. If $\Mon_k(f)$ is solvable, then either $f$ is linearly related over $\oline k$ to $X^p$ or $T_p$ for a  prime $p$, or $\Mon_k(f)=S_4$. If $\Mon_k(f)$ is nonsolvable, then it is an almost-simple group with primitive socle.
\end{prop}

In the solvable cases of Proposition \ref{prop:indec_poly}, $\Mon_k(f)$ is affine, and moreover it embeds into $\AGL_1(p)$ or $\Mon_k(f)=S_4\cong \AGL_2(2)$, respectively. Another characterizing property of these groups is that they contain a regular elementary abelian  normal subgroup: 
\begin{lem}
\label{lem:reg_normal_sub}
Let $d\ge 2$, $p$ a prime with $(p,d)\ne (2,2)$, and let $G\le S_{p^d}$ be a transitive group containing a cyclic transitive subgroup. Then $G$ cannot contain a $p$-elementary-abelian regular normal subgroup.
\end{lem}
\begin{proof}
Assume on the contrary that $N\cong C_p^d$ is an elementary-abelian regular normal subgroup of $G$. Then $G\cong N\rtimes G_1$, where $G_1\le G$ is a point stabilizer, and moreover $G$ embeds as a permutation group into the affine linear group $\AGL_d(p)$. This group does not have a cyclic transitive subgroup unless $d=1$ or $(p,d)=(2,2)$ \cite[Lemma\ 3.6]{Mul3}.\end{proof}


The following group-theoretic version of the well-known Capelli's lemma, see e.g.\ \cite[\S 2.1]{Ostrov},  characterizes intransitive subgroups of $\AGL_1(n)$: 
\begin{lem}\label{lem:onAGL1} Let $n\geq 1$ be an integer and $U\leq \AGL_1(n)$ be an intransitive subgroup. Then: \\
(1) 
If $n=p$ is prime,   $U\leq {\rm AGL}_1(p)$ fixes a point.\\
(2) If $n$ is a composite number,  then there exists a divisor $d|n$ which is either prime or equal to $4$ such that $U$ projects to an intransitive subgroup of $\AGL_1(d)$.
\end{lem}


\subsubsection*{Siegel functions}
Let $k$ be a number field with ring of integers $O_k$ and $\varphi:\mathcal{X}\to \mathbb{P}^1_k$ a map defined over $k$. By a famous theorem of Siegel, if $\varphi(\mathcal{X}(k))\cap O_k$ is infinite, then firstly, $\mathcal{X}$ is birational to $\mathbb{P}^1_k$, i.e.\ $\varphi$ is given by a rational function $f\in k(X)$, and furthermore $|\varphi^{-1}(\infty)|\le 2$, see \cite[Prop.\ 3.2]{Mul5}. When $k=\mathbb{Q}$, it is furthermore necessary for the preimages of $\infty$ to be algebraically conjugate. Motivated by this, we call a rational function $f\in k(X)$ over an arbitrary field $k$ of characteristic $0$ a {\it Siegel function} if $|f^{-1}(\infty)|\le 2$, and for $k=\mathbb{Q}$, we call $f$ a {\it Siegel function over $\mathbb{Q}$} if additionally either $|f^{-1}(\infty)|=1$ or the two preimages of $\infty$ are algebraic conjugates\footnote{Note that slightly different notions exist in the literature. However, importantly, the defining property used here is implied by the diophantine property of Siegel functions used in \S\ref{sec:intro}.}.

As in the case of indecomposable polynomials, a full classification of monodromy groups of indecomposable Siegel functions was obtained by M\"uller \cite{Mul3}, invoking the classification of finite simple groups. As noted in \cite[Theorem 2.2]{KN}, the following useful partial result can be obtained relying on the classification only in a ``mild" way, namely via certain bounds on the outer automorphism groups of simple groups.
\begin{prop}
\label{prop:siegel_indec}
Let $f\in k(X)$ be an indecomposable Siegel function. Then $\Mon_k(f)$ is either affine, or almost-simple, or contained in $(\Aut(S)\times \Aut(S))\rtimes C_2$ for a simple group $S$. In all cases, $\Mon_k(f)$ contains a unique minimal normal subgroup.
\end{prop}
We shall also use the following elementary property of composite Siegel functions. 
\begin{lem}
\label{lem:siegel_by_cyclic}
Let $g,h\in k(X)\setminus k$ be such that $f=g\circ h$ is a Siegel function. Then either $|g^{-1}(\infty)|=1$ or $\Mon_{\overline{k}}(h)$ is cyclic.
\end{lem}
\begin{proof}
Assume that $g^{-1}(\infty) = \{a,b\}$ consists of two elements. Since $|f^{-1}(\infty)|\le 2$, this enforces $|h^{-1}(a)|=|h^{-1}(b)|=1$, i.e., the rational function $h$ is totally ramified over these two points. 
Thus, the Riemann--Hurwitz formula shows that $h$ has no further branch points. 
By composing $h$ with linear fractionals, we may obtain a rational map totally ramified over $0,\infty$ with preimages $0,\infty$, resp., forcing it to be $cX^n$, where $n=\deg(h)$ and $c\in \oline k[X]$. Thus, $h$ is linearly related over $\overline{k}$ to $X^n$ and $\Mon_{\oline k}(h)$ is cyclic.
\end{proof}

\subsection{Composing maps}\label{ss:polydecom}
Recall that the monodromy group $\Mon_k(f)$ of a composition $f=g\circ h$ of two maps $g:\mathcal{Y}\ra\mZ,h:\mathcal{X}\ra \mathcal{Y}$ is a subgroup of $A\wr B:=A^d\rtimes B$, where $A:=\Mon_k(h)$ and $B:=\Mon_k(g)$.  
In particular,  $B$ is a natural quotient of $\Mon_k(f)$. Letting $B$ act on the set of roots $\mathcal B$ {of a minimal polynomial of $k(\mathcal{Y})/k(\mZ)$},  the  {\it stabilizer} of a block $b\in\mathcal B$ is the subgroup of $\Mon_k(f)$ fixing $b$ under its action through $B$. The {\it block kernel} is the kernel of the action of $\Mon_k(f)$ on $\mathcal B$. Letting $\Omega'\subseteq \Omega$ be the Galois closures of $g^*$, and $f^*$, resp., the block kernel coincides with $\Gal(\Omega/\Omega')$, while the block stabilizer coincides with $\Gal(\Omega/F(b))$, where $F=k(\mZ)$ and  $F(b)$ is the conjugate of $k(\mathcal{Y})$ corresponding to $b$. 

\subsubsection*{The block kernel for polynomials maps} We shall often use the following observation: 
\begin{lem}\label{lem:blocknontriv}
Suppose $f=g\circ h \in k[X]\setminus k$. 
Then the block kernel $K=\ker(\Mon_k(f) \rightarrow \Mon_k(g))$ contains an element of order 
$\deg(h)$.
\end{lem}
We  deduce the lemma from the following: 
\begin{lem}\label{lem:transitive-normal}
Let $L/K$ be a degree-$n$ extension totally ramified over a place $P$ of perfect residue,  with Galois closure $\Omega$, and Galois group $G:=\Gal(\Omega/K)$. Let $M/K$ be a Galois extension in which $P$ {is unramified}. 
Then: 
\begin{enumerate}[leftmargin=*,label={(\arabic*)},ref=(\arabic*)]
\item \label{case:transitive-normal_1}
$N=\Gal(\Omega M/M)$ is a normal subgroup of $G$ containing a cyclic transitive subgroup, upon identifying it with its image under restriction. In particular, $[LM:M]=[L:K]$.
\item For  $n=4$ and $G=S_4$, 
one has $N=S_4$ as well. 
\label{case:transitive-normal_2}
\end{enumerate}
\end{lem}
This is often applied together with the following consequences of Abhyankar's lemma: 
\begin{rem}\label{rem:abh} Let $L_1/K$ and $L_2/K$ be finite extensions of degrees $m,n$, resp. Let $P$ be a place of $K$ with a perfect residue field.  Abhyankar's lemma asserts that the ramification index of every place $Q$ of $L_1L_2$ over $K$ is $\lcm(e_1,e_2)$, where $e_i$ is the ramification index of $Q\cap L_i$ in $L_i/K$, see e.g.\ \cite[Cor.\ 5.2]{DZ}. It follows that:  
\vskip 1mm
\noindent
(1) Assume $n=m$ and let $\Omega$ be the Galois closure of $L_1L_2/K$. Then every place $Q$ of $\Omega$ above $P$ has ramification index $n$ in $\Omega/K$, but is unramified in $\Omega/L_1$. Indeed, the first statement follows by applying Abhyankar's lemma repeatedly to conjugates of $L_1$ and $L_2$ in $\Omega$, and the second one from the multiplicativity of ramification indices. 
\vskip 1mm
\noindent
(2) 
Assume that  $P$  is totally ramified in $L_1/K$ but unramified in $L_2/K$. Then: 
\begin{enumerate}
\item[a)] $L_1/K$ and $L_2/K$ are linearly disjoint.
\item[b)] Any place $Q$ of $L_2$ lying above $P$ is totally ramified in $L_1L_2/L_2$.
\end{enumerate}
Indeed: 
a)  Otherwise, by the multiplicativity of ramifications indices, the ramification index of a place $R$ of $L_1 L_2$ over $P$ would be at most $[L_1L_2:L_2]<m$, contradicting the fact that it must be divisible by the ramification index of $R \cap L_1$ over $P$, which is $[L_1:K]=m$. 
For b), let $R$ be a place of $L_1L_2$ lying above $Q$. By Abhyankar's lemma, the ramification index of $R$ over  $P$ is $m$. Since $P$ is unramified in $L_2/K$, the ramification index of $R$ over $Q$ is also $m$. Since $[L_1L_2:L_2] \leq m$, it follows that $[L_1L_2:L_2]=m$ and that $Q$ is totally ramified in $L_1L_2/L_2$.
\vskip 1mm
\noindent
(3)  Let $P$ be a place of $K$ that is ramified in $L_1/K$. If every place of $L_2$ above $P$ is unramified in $L_1L_2/L_2$, then $P$ is ramified in $L_2/K$. Indeed, otherwise, by the multiplicativity of ramification indices, $P$ would then be unramified in $L_1L_2/K$, a contradiction.
\begin{rem}\label{rem:Gal-closure}
If $L_1/K$ and $L_2/K$ are linearly disjoint extensions, and $\Omega$ is the Galois closure of $L_1/K$. Then the Galois closure of $L_1L_2/L_2$ is $\Omega L_2$, e.g.\ by \cite[Lemma\ 2.12 and Remark\ 2.13]{KN}. We shall therefore often identify $\Gal(\Omega L_2/L_2)$ with its image in $\Gal(\Omega/K)$ under the restriction map. 
\end{rem}

\begin{proof}[Proof of Lemma \ref{lem:transitive-normal}]
1) Since $L/K$ is totally ramified at $P$ while $P$ is unramified in $M/K$, the two extensions are linearly disjoint by Remark \ref{rem:abh}(2). Thus,  Remark \ref{rem:Gal-closure} implies that the Galois closure of $LM/M$ is the compositum $\Omega M$, and that we may 
identify $N:=\Gal(\Omega M/M)$ with its image in $G$. Moreover, since $M/K$ is Galois, $N$ is normal in $G$. 

Since $M/K$ is unramified over $P$, and $L/K$ is totally ramified over $P$, the extension $LM/M$ is also totally ramified (of the same ramification index $n$) over places of $M$  lying over $P$ by Remark \ref{rem:abh}(2).  
Thus, the inertia groups over such places are of order $n$. Consequently, $N$ contains a cyclic transitive subgroup.\\
2) Since $N\lhd S_4$ and  
contains a $4$-cycle by 1), it follows that $N=S_4$. 
\end{proof}

\begin{proof}[Proof of Lemma \ref{lem:blocknontriv}] Let $x$ be a root of $f(X)=t$,  let $w=h(x)$,  let  $\Omega/k(t)$ be the Galois closure of $k(x)/k(t)$, and $\Omega^K$ the fixed field of $K$. Since  $\infty$ {is unramified} in $\Omega^K/k(w)$ by Remark \ref{rem:abh}(1), and is totally ramified in $k(x)/k(t)$, Lemma \ref{lem:transitive-normal} gives the claim.  
\end{proof}

\end{rem}


\subsubsection*{Ritt moves} Recall that Ritt's theorems relate  \emph{complete decompositions} of $f$, that is,  decompositions $f=f_1\circ \dots \circ f_r$ for indecomposable polynomials $f_1,\dots,f_r \in k[X]$ \cite[Theorem\ 2.1.]{MZ}. Two complete decompositions \( f=f_1 \circ \dots \circ f_r \) and \( f=\tilde{f}_1 \circ \dots \circ \tilde{f}_r \) are related by a \emph{Ritt move} if   
$f_j = \tilde{f}_j$ 
for \( j \not\in \{i, i+1\} \), while  
\( f_i \circ f_{i+1} = \tilde{f}_i \circ \tilde{f}_{i+1} \) for some  \( 1 \leq i < r \),
and there is no linear $\mu\in k[X]$ such that $f_i=\tilde f_i\circ \mu$ and $f_{i+1}=\mu^{-1}\circ \tilde f_{i+1}$.
Given a pair of complete decompositions for a polynomial, Ritt's first theorem then asserts that 
it is possible to pass from one to another by a sequence of consecutive Ritt moves. Moreover, a complete decomposition $f=f_1\circ \cdots \circ f_r$,  $f_i\in\oline k[X]$, over $\oline k$ has to be defined over $k$ by Fried--MacRae  \cite{FM}, that is, 
there exist linear  $\ell_1,\dots,\ell_{r-1} \in \overline{k}[X]$ such that $f_1 \circ \ell_1, \ell_{r-1}^{-1} \circ f_r, \ell_{i-1}^{-1} \circ f_i \circ \ell_i \in k[X]$ for all $2\leq i \leq r-1$. 

We shall mainly use the following consequence of Ritt's  theorems, which follows from the above since the degrees in a Ritt move are coprime \cite[Corollary 2.11]{MZ}. 
Say that $v\in k[X]$ is a {\it right-unique} factor if, for every complete decomposition $f=f_1\circ\cdots\circ f_r$, $f_i\in k[X]$, there exists some $i$ and a linear $\mu\in k[X]$ such that $\mu\circ v=f_{i+1}\circ \cdots\circ f_r$. In this case,  we also call a decomposition $f=u\circ v$ right-unique. 
The decomposition is further called {\it strongly-unique} if furthermore it  does not have a right-factor\footnote{In this terminology $X^{p^2}=X^p\circ X^p$ and $T_{p^2}=T_p\circ T_p$ are right-unique but not strongly-unique. } that is linearly related to  
$X^{p^2}$ or $T_{p^2}$ over $\oline k$. 

\begin{cor}\label{cor:ritt}
Let $f=h\circ f_r\in k[X]\setminus k$ be a decomposition with $h,f_r\in k[X]$ and $f_r\in k[X]$ indecomposable. If $f_r$ is not a right-unique factor, then there exist indecomposable $u,u',v'\in k[X]$ for which $u\circ f_r=v'\circ u'$ is a right-factor of $f$ that is a Ritt move. Moreover, $\deg(u)=\deg(u')$ and $\deg(f_r)=\deg(v')$ are coprime. 
\end{cor}
We say that $f\in k[X]\setminus k$ is {\it simply branched} if each of its finite branch points has a unique ramified $f$-preimage and this preimage has ramification index $2$. This is the type of ramification occuring generically. 
\begin{rem}\label{rem:further-Ritt}
In Appendix \ref{app:mn}, we shall also use Ritt's second theorem \cite[Theorem 2.17]{MZ} directly: If $a\circ b=c\circ d$ for $a,b,c,d\in \mC[X]$ of degree $>1$ with $\gcd(\deg(a),\deg(b))=\gcd(\deg(c),\deg(d))=1$, then (up to switching $a,b$ and $c,d$), there exist $\ell_j\in\mC[X]$ of degree $1$ such that $(\ell_1\circ a\circ\ell_2^{-1},\ell_2\circ b\circ \ell_3^{-1}, \ell_1\circ c\circ \ell_4^{-1},\ell_4\circ d\circ \ell_3^{-1})$ is either (1) $(T_m,T_n,T_n,T_m)$ or (2) $(X^n,X^sh(X^n),X^sh(X)^n,X^n$), for some coprime $m,n> 1$, some $s\geq 1$ coprime to $n$, and $h\in \mC[X]\setminus X\mC[X]$.
In particular, simply-branched polynomials of degree $\geq 4$ do not occur as left factors of Ritt moves. 
It also follows immediately  that two left-factors involved in a Ritt step must share a finite branch point. In (1) and (2), these are $-2$ and $0$, resp. 
\end{rem}


We shall also use the following part of Ritt's first theorem which applies more generally to maps (or their corresponding extensions) with a totally ramified point. 
\begin{lem}\label{lem:towerext}
Let  $L/K$ be a finite extension with Galois closure $M/K$, and  $P$ be a place of $K$  that is totally ramified in $L/K$ with a perfect residue field. 
Let $I$ be the inertia group at a place of $M$ above $P$ in $G={\rm Gal}(M/K)$, and let $H={\rm Gal}(M/L)$. Then $G= H I$ and $H \cap I=1$. Moreover, for any intermediate subgroup $H \leq T \leq G$, we have $T= H (T \cap I)$.
\end{lem}
\begin{proof}
The extension $M/L$ is unramified at any place of $L$ over $P$ by Remark \ref{rem:abh}(1), which implies that $H \cap I=1$. Since $P$ is totally ramified in $L/K$, it follows that $G=H I$. The last assertion then follows directly, see \cite[Lemma 2.5]{MZ}.
\end{proof}

\subsection{Reducibility}\label{sss:minredpair} 
Given $f,g\in k(X)\setminus k$,  the curve $f(X)=g(Y)$ is birational to the fiber product  $\mP^1\#_{f,g}\mP^1$ of $f:\mP^1\ra\mP^1$ and $g:\mP^1\ra\mP^1$. Moreover, this fiber product is irreducible over $k$ if and only if the root fields $k(x)$ and $k(y)$ of $f(X)-t$ and $g(Y)-t$ over $k(t)$, respectively, are linearly disjoint over $k(t)$. 
We shall repeatedly use the group-theoretic translation of this condition via the Galois coorespondence: letting $\Omega$ denote the Galois closure of $k(x,y)/k(t)$, and setting $G=\Gal(\Omega/k(t)),H=\Gal(\Omega/k(x)),H'=\Gal(\Omega/k(y))$, the linear disjointness of $k(x)/k(t)$ and $k(y)/k(t)$ is equivalent to the transitivity of $H$ on $G/H'$, that is,  to $H\cdot H'=G$.

The following well-known lemma shows that minimally reducible pairs have a common Galois closure.  As for polynomials, say that two extensions $M/K$ and $L/K$, which are not linearly disjoint, form a \emph{minimally reducible pair} if $M_1/K$ and $L_1/K$ are linearly disjoint for any $K \subset L_1 \subset L$ and $K \subset M_1 \subset M$ such that $L \neq L_1$ or $M \neq M_1$.
\begin{lem}\label{lem:friedarg} Let $L/K$ and $M/K$ be finite separable extensions forming a minimally reducible pair. Then their Galois closures coincide. 
\end{lem}
\begin{rem}\label{rem:genFriarg} 
Letting $\Omega/K$ be the Galois closure of $L/K$, since $L/K$ and $M/K$ are not linearly disjoint, so are $L/K$ and $M\cap \Omega/K$, see e.g.\footnote{The proof is given for curves. Replacing these by their function fields, it clearly applies to general fields.} \cite[Lemma\ 2.11]{KN}. 
\end{rem}

\begin{proof}[Proof of Lemma \ref{lem:friedarg}]
Suppose on the contrary $M$ is not contained in the Galois closure $\Omega$ of $L/K$. Since $L$ and $M$ are not linearly disjoint over $K$, so are $L$ and $M\cap \Omega$ by Remark \ref{rem:genFriarg}. Since $M_1:=M\cap \Omega$ is properly contained in $M$, the fact that $M_1/K$ and $L/K$ are not linearly disjoint contradicts the minimal reducibility of $M/K$ and $L/K$. 
\end{proof}

\begin{cor}\label{cor:min-red-pol}
If $f,g\in k[X]\setminus k$ is a minimally reducible  pair, then  $\deg(f)=\deg(g)$ and the branch loci of $f$ and $g$ coincide. 
\end{cor}

\begin{proof}
Letting $k(x),k(y)$ be root fields of $f(X)-t,g(X)-t$, resp., their Galois closures $\Omega$ coincide by Lemma \ref{lem:friedarg}. By Remark  \ref{rem:abh}(1), the ramification index $\deg(f)$ (resp., $\deg(g)$) of $\infty$ in $k(x)/k(t)$ (resp., $k(y)/k(t)$) coincides with that in $\Omega/k(t)$. Thus   $\deg(f)=\deg(g)$. Moreover, the remark implies the branch loci of $k(x)/k(t)$, $k(y)/k(t)$ and $\Omega/k(t)$ coincide. 
\end{proof}

In particular, one deduces the examples appearing in Theorems \ref{thm:DLS-C} and \ref{thm:DLS}: 
\begin{exam}\label{ex:equtoT4-T4}
Let $f,g:\mathbb{P}^1_k\rightarrow \mathbb{P}^1_k$ be polynomial maps having the same Galois closure $\tilde{X}$, which has genus $0$ and monodromy group $G=D_4$. 
Then one of the following holds:\\
(1) $\Mon_{\overline{k}}(f)=D_4$. In this case, $\{\mu \circ f \circ \eta_1,\mu \circ g \circ \eta_2\}=\{D_{4,\alpha},-\frac{1}{4}D_{4,2\alpha}\}$ for some $\mu, \eta_1,\eta_2 \in k[X]$ of degree $1$ and $\alpha\in k^\times$, with $D_{4,\alpha} = X^4-4\alpha X^2+2\alpha^2$ the degree-$4$ Dickson polynomial with parameter $\alpha$. 
Indeed, it is well known  that $f$ must be of the given shape, see, e.g., \cite[Theorem 5.2]{GMS}. For $g$, one notes that:
\begin{equation}\label{eq:facD4al}
D_{4,\alpha}(X) + \frac{1}{4}D_{4,2\alpha}(Y) = \left(X^2 - XY + \frac{1}{2}Y^2 - 2\alpha\right)\left(X^2 + XY + \frac{1}{2}Y^2 - 2\alpha\right),
\end{equation} so that $D_{4,\alpha}$ and $-(1/4)D_{4,2\alpha}$ is a {minimally reducible} pair. 
Letting $s$ be a reflection and $r$ a rotation in $D_4$, note that  $\langle rs\rangle$ is the only maximal subgroup of $D_4$ which is intransitive on $D_4/\langle s\rangle$ and does not contain a conjugate of $\langle s\rangle$. 
It follows that the unique minimally reducible pair  of degree-$4$ maps with Galois closure $\tilde X$ is $\tilde X/\langle s\rangle\to \tilde X/G$ and $\tilde X/\langle rs\rangle\to \tilde X/G$, up to automorphisms of $\tilde X/G, \tilde X/\langle rs\rangle, \tilde X/\langle s\rangle\cong \mP^1_k$ given by degree-$1$ polynomials. Thus, up to swapping $f,g$, they are related to $D_{4,\alpha}$, $-(1/4)D_{4,2\alpha}$, resp., by  linear $\mu,\eta,\eta'\in k[X]$ as claimed. 
\\
(2) $\Mon_{\overline{k}}(f)=C_4$. In this case, $\{\mu \circ f \circ \eta_1,\mu \circ g \circ \eta_2\}=\{X^4,-\frac{1}{4}X^4\}$ for some $\mu, \eta_1,\eta_2 \in k[X]$ of degree $1$.
Indeed, the Riemann-Hurwitz formula shows instantly that $f$ has exactly one finite branch point $P$, whose ramification index equals $4$. Since $k$-conjugates of branch points of $f\in k[X]$ are also branch points, this shows that $P$ is $k$-rational.
Thus, $f$ is linearly related to $X^4$ over $k$. The pair of projections $\tilde X\to \tilde X/\langle s\rangle, \tilde X\to \tilde X/\langle rs\rangle$ is again the unique minimally reducible pair of degree $4$ polynomial maps with Galois closure $\tilde X$, up to  composing with automorphisms of $\tilde X/G,\tilde X/\langle rs\rangle, \tilde X/\langle s\rangle\cong \mP^1_k$ given by degree-$1$ polynomials. 
Since $X^4, -\frac{1}{4}X^4$ is  a minimally reducible pair of monodromy $D_4$ and geometric monodromy $C_4$, up to swapping $f,g$, it follows as in (1) that these are related  to $X^4,-\frac{1}{4}X^4$ by linear polynomails as claimed. Moreover, in this case one necessarily has $\sqrt{-1}\notin k$, since otherwise the polynomial $X^4-t$ would have Galois group $C_4$, rather than $D_4$, over $k(t)$.  
Note that $X^4=D_{4,0}$, whence this case may be in fact be seen as a degeneration of the family in Case (1).
\end{exam}


The following proposition restricts right factors of minimally reducible pairs. 
\begin{prop}\label{lem:norittstepp}
\label{cor:right-unique}
\label{lem:coprime}
Let $f,g\in k(X)\setminus k$  be a minimally reducible pair with  decompositions $f=f_1\circ f_2$, $g=g_1\circ g_2$ such that $n=\deg(f_2)\geq 2$, and $m=\deg(g_2)\geq 2$. 
\begin{enumerate}[leftmargin=*,label={(\arabic*)},ref=(\arabic*)]
\item 
If $n\neq 4$ is composite, then $f_2$ is not linearly related over $\oline k$ to $X^n,T_n$. 
\label{case:norittstepp_1}
\item  If $n=4$,  then $\Mon_k(f_2)$ is noncyclic. 
\label{case:norittstepp_2}
\item The integers $m,n$ have a nontrivial common divisor. 
\label{case:norittstepp_3}
\item 
When $f,g\in k[X]$ are polynomials and $f_2,g_2\in k[X]$ are indecomposable, then furthermore $m=n$ 
or $\{m,n\}=\{2,4\}$. In particular, $f_2$ is right-unique.
\label{case:norittstepp_4}
\item If $f\in k[X]$ is a polynomial, $g\in k(X)$ is a Siegel function, and $f_2,g_2$ are indecomposable, then $f_2$ is right-unique. 
\label{case:norittstepp_5}
\end{enumerate}
\end{prop}

\begin{proof}
Let $x,y$ be roots of $f(X)-t,g(X)-t$, respectively. Let $\Omega$ be the common Galois closure of $k(x)/k(t)$ and $k(y)/k(t)$ given by Lemma \ref{lem:friedarg}. Set $u=f_2(x)$ and $v=g_2(y)$. Let 
$\tilde \Omega$ be the Galois closure of $k(x)/k(u)$.
Since $k(y,u)/k(u)$ and $k(x)/k(u)$ are not linearly disjoint, $L=k(y,u)\cap \tilde \Omega$ is not linearly disjoint from $k(x)$ over $k(u)$ by Remark \ref{rem:genFriarg}.\\
(1) Suppose on the contrary that $f_2$ is linearly related over $\overline{k}$ to $T_n$ or $X^n$, so that $\Mon_k(f_2)\le \AGL_1(n)$. 
Since $n\neq 4$ is composite, Lemma \ref{lem:onAGL1} and the Galois correspondence imply 
there exists an intermediate field $k(u) \subsetneq k(u') \subsetneq k(x)$ such that $L/k(u)$ and hence $k(y,u)/k(u)$ are not linearly disjoint from  $k(u')/k(u)$, contradicting the minimal reducibility of $f$ and $g$. \\
(2) Assume on the contrary $\Mon_k(f_2)=C_4$. Then $\tilde \Omega=k(x)$ and $L\subseteq k(x)$. Since $L$ and $k(x)$ are not linearly disjoint over $k(u)$, the extension $L/k(u)$ is nontrivial, and hence contains the quadratic subextension $k(u')/k(u)$ of $k(x)/k(u)$, contradicting the minimal reducibility of $f$ and $g$. \\
(3)  Since $f,g$ are minimally reducible,   $k(x)$ and $k(v)$ are linearly disjoint over $k(t)$, so that $[k(x,v):k(u,v)]=[k(x):k(u)]=n$ and similarly $[k(u,y):k(u,v)]=[k(y):k(v)]=m$. If $m$ and $n$ are coprime, the extensions $k(x,v)/k(u,v)$ and $k(u,y)/k(u,v)$ are  linearly disjoint, contradicting the minimal  reducibility of the pair $f,g$.  \\
(4) 
We first claim that $\soc(\Mon_k(f_2))$ and $\soc(\Mon_k(g_2))$ coincide. Denote by $\Omega_u$  and $\Omega_v$ the Galois closure of $k(u)/k(t)$ and of $k(v)/k(t)$, respectively. 
Then the ramification index of $t\mapsto \infty$ in the Galois closure of $k(u,v)/k(t)$,
and hence in particular in $\Omega_v(u)$ divides $\deg(f)/p$ by Remark \ref{rem:abh} (Abhyankar's lemma), for some prime divisor $p$ of $\gcd(m,n)$, which exists by (3). In particular, $\Omega_{v}\tilde \Omega/\Omega_{v}(u)$ must be nontrivial, and hence its Galois group identifies via restriction with a nontrivial normal subgroup of the primitive group $\Mon_k(f_2)$. Thus it must contain the transitive subgroup $\soc(\Mon_k(f_2))$. In particular, this transitivity implies that $\Omega_{v}\tilde \Omega/\Omega_v(u)$ is the Galois closure of $\Omega_v(x)/\Omega_v(u)$. In the same way, the Galois group of the Galois closure of $\Omega_u(y)/\Omega_u(v)$ must contain $\soc(\Mon_k(g_2))$. 
Since $k(x,v)/k(u,v)$ and $k(u,y)/k(u,v)$ are not linearly disjoint,  this implies that $\Mon_k(g_2)$ must contain 
a copy of $\soc(\Mon_k(f_2))$ and vice versa, yielding the  claim by Proposition \ref{prop:indec_poly}.

If one (and hence both) of $\Mon_k(f_2) ,\Mon_k(g_2)$ is solvable, it follows  that $m=n$ or $\{m,n\}=\{2,4\}$ by Proposition \ref{prop:indec_poly}. For nonsolvable $\Mon_k(f_2)$, we may use the classification in \cite{Mul2} to verify that the only instance of noncoprime degrees $m\ne n$ for which $\Mon_k(f_2), \Mon_k(g_2)$ share their nonabelian composition factor is $\{m,n\}=\{6,10\}$ (and $\soc(\Mon_k(f_2))\cong A_6$). But note that the index-$10$ subgroups of $A_6$ still act transitively on $6$ points, contradicting linear disjointness of $k(x,v)/k(u,v)$ and $k(u,y)/k(u,v)$.
Finally, assume $f_2$ is not right-unique. Then   $f=f_1'\circ f_2'$ for an indecomposable $f_2'\in k[X]\setminus k$ of degree $n'$ coprime to $n=\deg(f_2)$ by Corollary \ref{cor:ritt}. But clearly $m=n$ or $\{m,n\}=\{2,4\}$ implies that $m$ is coprime to $n'$, a contradiction. \\
(5) Suppose on the contrary $f_2$ is not right-unique.  Then there must be another right factor $f_2'$ of $f$ of degree coprime to $\deg(f_2)$ by Corollary \ref{cor:ritt}. It follows from (3) that an indecomposable right factor $g_2$ of $g$ must have a primitive monodromy group of degree not a prime power, which is thus necessarily nonsolvable by Galois' theorem \cite[Theorem 14.3.17]{Cox}). 
The same would hold for {\it any} indecomposable right factor of $g$. Using an analog of Ritt's theorem for Siegel functions  \cite[Theorem 1.1]{Pak}, it follows that $g_2$ must be right-unique, since in particular a Ritt move for such functions cannot involve two indecomposables with nonsolvable monodromy group. Since $\deg(g_2)$ is not a prime power, it follows from Proposition \ref{prop:siegel_indec} that $\Mon_k(g_2)$ is either almost-simple or of product-type, and in particular has a unique and nonsolvable minimal normal subgroup $S^d$ for a simple group $S$. Together with the right-uniqueness of $g_2$, it follows from \cite[Proposition 3.3]{KNR24} that $\soc(\ker(\Mon_k(g)\to\Mon_k(g_1)))$ is the unique minimal normal subgroup of $\Mon_k(g)=\Mon_k(f)$, and is a power of the simple group $S$. But by Ritt's theorems one of $f_2$ and $f_2'$ (say, $f_2$) has solvable monodromy group, meaning that $\ker(\Mon_k(f)\to\Mon_k(f_1))$ is a nontrivial solvable normal subgroup of $\Mon_k(f)$, a contradiction.
\end{proof}
The CFSG is used in parts (4)-(5) of the above proof. The proof of  Theorems \ref{thm:DLS-C} and \ref{thm:DLS} apply part (4) only when $\Mon(f)=\Mon(g)$ is solvable, where the CFSG is not applied. 

Finally, we will also need the following description of right quadratic factors\footnote{The assertion can be strengthened as follows: If $f_r$ is a composition of $r$ quadratics, then $g$ has a right factor which is a composition of $r-1$ quadratics. However, we shall not make use of such an extension.}: 
\begin{lem}\label{lem:2-power}
Let $f,g\in k[X]$ be a minimally reducible pair such that $f,g$ have  right composition factors $f_r,g_s$, resp., where $f_r$ is a composition of two quadratic polynomials. Then $g$ has a quadratic right composition  factor. 
\end{lem}
\begin{proof}
Let $\Omega$ be the common Galois closure of $k(x)/k(t)$ and $k(y)/k(t)$ given by Lemma \ref{lem:friedarg}, where $x,y$ are roots of $f(X)-t,g(X)-t$, respectively.
Let $u=f_r(x)$, and $v=g_s(y)$ for an indecomposable right factor $g_s$ of $g$. Note that since $f,g$ are minimally reducible, by Proposition \ref{lem:coprime}\ref{case:norittstepp_4}, $g_s$ is right-unique and $\deg(g_s)$ is $2$ or $4$. Assume on the contrary $\deg(g_s)=4$ and hence $\Mon_k(g_s)=S_4$.

Let $\Omega'$ be the Galois closure of $k(u,v)/k(t)$. Since both $k(u)/k(t)$ and $k(v)/k(t)$ are totally ramified over $\infty$ of the same degree $m:=\deg(f)/4$, this is also  the ramification index of $\infty$ in $\Omega'/k(t)$ by Remark \ref{rem:abh}. Moreover, since $\infty$  is unramified in $\Omega'/k(u)$ and $\Omega'/k(v)$, the remark shows  these extensions are linearly disjoint from $k(x)$ and $k(y)$, resp. 

Let $\Omega_u\subseteq \Omega'$ be the Galois closure of $k(u)/k(t)$. Since  $\Omega$ is  the splitting fields of $f_r(X)-u'$ over $\Omega_u$ when $u'$ ranges over conjugates of $u$ (similarly to Lemma \ref{lem:subdirect} below), 
$\Omega/\Omega_u$ is a compositum of $2$-extensions, and hence  $\Gal(\Omega/\Omega_u)$ and $\Gal(\Omega/\Omega')$ are  $2$-groups. This contradicts that the Galois closure of $\Omega'(y)/\Omega'$ has group $S_4$ by Lemma \ref{lem:transitive-normal}\ref{case:transitive-normal_2}.
\end{proof}

\subsection{Subdirect powers}\label{sec:subdirect}
Say $G\leq \prod_{i\in I}H_i$ is a subdirect product, for some groups $H_i$ indexed by a set $I$, if the projection of $G$ to $H_i$ is surjective for each $i\in I$.
We say that $K$ is a subdirect power of $\oline K$ if it is a subdirect product of $\prod_{i\in I}\oline K$ for some $I$. 
\begin{lem}\label{lem:subdirect}
Let $f=g\circ h:\mX\ra \mZ$ be a composition of maps. Let $K$ be the kernel of the projection $\Mon_k(f)\to \Mon_k(g)$. Let $M\leq K$ be normal in $\Mon_k(f)$, and $\oline M$ its image in $\Mon_k(h)$ via the action on some fixed block. Then $M$ is a subdirect power of $\oline M$. 
\end{lem}
In particular, the lemma applies to $M=K$ and $M=\soc(K)$. 
\begin{proof}
Let $t$ be the  generic point of $\mZ$.
For every $y\in g^{-1}(t)$, let $\mathcal{X}_{y}$ be the set of $x\in f^{-1}(t)$ with $h(x)=y$. 
Let $y_0\in g^{-1}(t)$ be our fixed block. 
Since $M$ is normal in $G=\Mon_k(f)$ and the latter is transitive on $g^{-1}(t)$, for every $y\in g^{-1}(t)$, there is an element of $G$ that sends $y_0$ to $y$, and hence conjugation by that element maps $\oline M$ isomorphically to the image $M_{y}$ of the action of $M$ on $\mX_{y}$. 
Since  $M$ acts faithfully on  $\bigcup_{y\in \mY}\mX_{y}$, it embedds in $\prod_{y\in\mY}M_{y}\cong \prod_{y\in\mY}\oline M$ with surjective projections, so that $M$ is a subdirect power of $\oline M$. 
\end{proof}

The following is our key consequence to relating the $1$-step and $2$-step kernels. 
\begin{cor}\label{cor:3-useful} 
Let $f=f_1\circ f_2\circ f_3: \mX\to \mW$ be a composition of (nonconstant) maps $f_1,f_2,f_3$. 
Let $K,N$ be the kernels in the action of $G:=\Mon_k(f)$ through $\Mon_k(f_1\circ f_2), \Mon_k(f_1)$, resp., and $M\leq N$ be normal in $G$.  
Let $\oline K,\oline M, \oline{K\cap M}$ be the images of $K,M,K\cap M$, resp., in the action of $G_2:=\Mon_k(f_2\circ f_3)$ on a fixed block, 
and $K_2$ the kernel of the action of $G_2$ through $\Mon_k(f_2)$.  
Then $M/(K\cap M)$ is a subdirect power of $\oline M/(K_2\cap \oline M)$. In particular, $\oline M/\oline{K\cap  M}$ is a quotient of a subdirect power of $\oline M/(K_2\cap \oline M)$. 
\end{cor}
\begin{proof}
Let $t$ be the  generic point of $\mW$, and $z_0\in f_1^{-1}(t)$ our fixed block. 
Then $K$ is the kernel of the action of $G$ on $(f_1\circ f_2)^{-1}(t)$, and $N$ is the kernel of the action of $G$ on $f_1^{-1}(t)$. 
Thus, $N/K$ and its subgroup $M/(K\cap M)$ act faithfully on $(f_1\circ f_2)^{-1}(t)$. 
Write $(f_1\circ f_2)^{-1}(t)$ as a disjoint union of blocks $\mathcal Y_z:=f_2^{-1}(z)$, $z\in f_1^{-1}(t)$. 
Since the kernel of the action of $\oline M$ on 
each block $\mathcal Y_z$ is $\oline M\cap K_2$, the image of the action of $M$ on  $\mathcal Y_{z_0}$ is $\oline M/(K_2 \cap \oline M)$. 
Thus, $M/(K\cap M)$ is a subdirect power of $\oline M/(K_2\cap \oline M)$  by Lemma \ref{lem:subdirect} applied to the $G/K$-action on $(f_1\circ f_2)^{-1}(t)$. The last conclusion follows since $\oline M/\oline{K\cap M}$ is a quotient of $M/(K\cap M)$.  
\end{proof}

The following describes the socle of  subdirect powers of $\AGL_1(p)$ and $\AGL_2(2)\cong S_4$. 
\begin{lem}\label{lem:interandsoc}
Let $p$ be a prime, and suppose either $(r,p)=(2,2)$ or $r=1$. Let $K \leq \AGL_r(p)^n$ be a subgroup whose component projections contain $C_{p}$ if $r=1$ (resp.\ the $A_4$ copy in $S_4\cong \AGL_2(2)$ if $r=2$). Then $\soc(K)=K \cap C_p^{rn}$.
\end{lem}
\begin{proof}
Since $K$ is solvable by our assumptions on $r,p$, $\soc(K)$ is the direct product of elementary abelian $q$-subgroups $H_q$ for various primes $q$. Since the component projections of each $H_q$ are abelian normal subgroups of a transitive subgroup of $\AGL_r(p)$, and hence contained in $C_p^r$, it follows that $H_q=1$ for every $q \neq p$. Thus, $\soc(K)\subseteq C_p^{rn} \cap K$. 

For the reverse inclusion, note that 
$C_p^{n}$ is a semisimple module under the action of $K/(C_p^{rn} \cap K)$ if $r=1$. Similarly if 
$r=2$, since the Klein group $V_4\cong C_2^2$ forms an irreducible module under the $A_4$-action and since the component projection contains $A_4$, the module $V_4^n$ is semisimple under the action of $K/(V_4^{n} \cap K)$. 
Thus in both cases, the submodule $K\cap C_p^{rn}$ is also semisimple, and hence its submodule $\soc(K)$ has a complement $N$, which in particular is normal in $K$. By the definition of the socle, this implies $N=1$.
\end{proof}
The smallest subdirect powers are the diagonal ones:  say that a subdirect power $K\leq H^n$ is {\it diagonal} if each of its $n$ projections $K\to H$ is injective. The following remark shows that the diagonaliy of $\soc(K)$ implies that of $K$. 
\begin{rem}\label{rem:diag}
Suppose $K\leq \oline K^m$ is a subdirect power, and $\soc(K)\leq\soc(\oline K)^m$ is diagonal. Then $K\leq\oline K^m$ is diagonal. Indeed, the kernel $C$ of the coordinate projection $K\to \oline K$ is a normal subgroup which is disjoint from the diagonal subgroup $\soc(K)$, and hence $C=1$ by definition of $\soc(K)$, as desired. 
\end{rem}
If $G$ acts on a power $A^d$ of a group $A$ 
by acting transitively on the $d$ copies of $A$ (supported only on one of the coordinates)\footnote{Equivalently, the action of $G$ on $A^d$ has a partition that corresponds to the action of $G$ on $\{1,\ldots,d\}$.}. 
Then there is a unique diagonal $G$-invariant copy of $A$ in $A^d$ which we refer to as {\it the diagonal subgroup of $A^d$ with respect to the $G$-action}. 

{On the other extreme, if $A=C_q$ and $A^d$ is a permutation $G$-module, the following submodule is among the largest submodules.  Throughout the paper, we denote by $I_d(q)$ the $G$-submodule of all $(x_1,\ldots,x_d)\in A^d$ with sum $\sum_{i=1}^dx_i=0$. }
\subsection{Wreath products of affine groups}
For polynomial maps $f,g \in k[X]$ of degrees $d$ and $q$, resp.,  with $q$ prime and $\Mon_k(g)$ solvable we 
have $\Mon_k(f\circ g)\leq A\wr B$, where $C_q\leq A\leq 
\AGL_1(q)$ and $B\leq S_d$ (see \S\ref{ss:polydecom}); more 
precisely, we take $A=\Mon_k(g)$ and $B=\Mon_k(f)$. This section presents preliminary results on  subgroups of $A\wr B$ for such~$A,B$. 

We  describe normalizers and commutator subgroups  of transitive subgroups of $G:=\AGL_1(q)\wr S_d$ as follows. For $C,H\leq G$, let $\mathcal N_C(H)$ denote the elements of $C$ normalizing $H$, and $[C,H]\leq G$ the commutator subgroup, generated by commutators $[c,h],c\in C,h\in H$.  

\begin{prop}\label{lem:normalofHatmostp_new}
Let $V\le S_d$ be a group containing a cyclic transitive subgroup, let $q$ be a prime,  let $G\le  C_q\wr V$ be a subgroup surjecting onto $V$, and
$H=C_q^d \cap G\trianglelefteq G$. 
Let $U\le C_q^d$ be such that $H\le U\le \mathcal{N}_{C_{q}^d}(G)$. 
Then the following hold:
\begin{enumerate}[leftmargin=*,label={(\arabic*)},ref=(\arabic*)]
\item $H$ is an $\mathbb{F}_q[V]$-submodule of $U$ under the action $\overline{\sigma}\cdot u:=\sigma u\sigma^{-1}$, where $G\to V$, $\sigma \mapsto \overline{\sigma}$ is the projection.
\label{case:normalofHatmostp_new_1}
\item Letting 
$Z:={\rm diag}(C_q^d) \cap U$, we have an embedding of modules $U/Z \hookrightarrow H$. 
\label{case:normalofHatmostp_new_2}
\item In particular, $[U:H]$ divides $q$.
\label{case:normalofHatmostp_new_3}
\end{enumerate}
\end{prop}
\begin{proof}
Regarding 1), let $u\in U$ and $\sigma\in G$ be arbitrary. Then $u\sigma^{-1} u^{-1}\in G$. By considering projection modulo $C_q^d$, it follows that $u\sigma^{-1}u^{-1} = \sigma^{-1}h$ for some $h\in C_q^d\cap G = H$.
Hence, 
\begin{equation}
\label{eq:commutator}
\sigma u \sigma^{-1}u^{-1}\in H. 
\end{equation} Thus $\sigma U\sigma^{-1} \subseteq H U = U$, showing i).

Regarding 2),  choose $\sigma \in G$ as a preimage of a generator of a cyclic transitive subgroup of $V$. Note that due to \eqref{eq:commutator}, $H$ contains the $\mathbb{F}_q[\oline \sigma]$-module $[U,\langle \sigma \rangle]$.
Note that when identifying the $\mathbb{F}_q[\oline \sigma]$-permutation module $C_q^d$ with $\mathbb F_q[\oline \sigma]$,  the commutator
$[u,\sigma]=u\sigma u^{-1}\sigma^{-1}=u(\oline\sigma\cdot u^{-1})\in U$ identifies with 
$   u-\overline{\sigma}u=(1-\overline{\sigma})u\in \mathbb F_q[\oline \sigma]$.
Now $\varphi: U\to [U,\langle\sigma\rangle]$, $u\mapsto [u,\sigma]$ is a homomorphism of submodules of the $\mathbb{F}_q[\oline\sigma]$-permutation module, whose kernel $\ker(\varphi)=\ker(1-\overline{\sigma})$ is clearly the diagonal $Z$.

Finally, 3) is an immediate consequence of 2), since $\dim(Z)\le 1$.
\end{proof}

\begin{cor}\label{cor:normalofHatmostp_doneright}
Let $q$ be a prime and $d$ an integer. Let $G\le \AGL_1(q)\wr S_d$, 
and let $N\trianglelefteq G$ be a normal subgroup. Set $G_q:=G\cap C_q^d$ and $N_q:=N\cap C_q^d$. Assume that $N$ contains an element $\sigma$ whose image $\oline \sigma$ in $S_d$ 
is a $d$-cycle, and  $\sigma^d\in N_q$. Then the following hold:
\begin{enumerate}[leftmargin=*,label={(\arabic*)},ref=(\arabic*)]
\item $G_q/Z$ embeds into $N_q$, where $Z\le G_q$ denotes  the intersection of  $G_q$ with the diagonal under the action of $\sigma$. In particular $[G_q:N_q]\divides q$.
\label{case:normalofHatmostp_doneright_1}
\item If $\sigma$ is a $qd$-cycle and $(d,q)=1$, then $G_q=N_q$.
\label{case:normalofHatmostp_doneright_2}
\end{enumerate}
\end{cor}
\begin{proof}
To see 1), note that $\langle N_q,\sigma\rangle\trianglelefteq \langle G_q,\sigma\rangle$, and hence $G_q\le \mathcal{N}_{C_q^{d}}(\langle N_q, \sigma \rangle)$:     
Indeed, for any $a\in G_q$, we have $a\sigma a^{-1}\in N$ since $\sigma\in N\lhd G$, so $a\sigma a^{-1}\sigma^{-1}\in N\cap C_q^{d}=N_q$, and thus $a\sigma a^{-1}\in \langle N_q, \sigma\rangle$. Now 1) follows  from Proposition \ref{lem:normalofHatmostp_new}.

For 2), note that if $(d,q)=1$, the $\mathbb{F}_q[\oline \sigma]$-module $G_q$ is semisimple and, being a submodule of the permutation module $\mathbb F_q[\oline\sigma]$, contains at most one copy of the trivial module. Thus, if $N_q\subsetneq G_q$, then by 1), we must have $\dim(Z)=1$, and by the above, $G_q/Z$ contains no trivial module under the action of $\oline \sigma$. However, when $\sigma$ is a $dq$-cycle, the $d$-th power of $\sigma$ generates such a one-dimensional trivial module inside $N_q$, making $G_q/Z\cong N_q$ impossible. Therefore $G_q=N_q$. 
\end{proof}

The following consequence is useful when dealing with permutation groups of prime power degree.
\begin{cor}
\label{cor:nilp_class}
Let $q$ be a prime and $d=q^r$ a power of $q$ ($r\ge 1$). Let $G\le \AGL_1(q)\wr S_d$ and let $\sigma\in G$ be an element mapping to a $d$-cycle under the projection $G\to S_d$. Set $G_q:=G\cap C_q^d\triangleleft G$ and let $e$ be the rank of the elementary-abelian group $G_q$. Let $Q\le G$ be a $q$-Sylow subgroup. Then $Q$ has nilpotency class at least $e$.
\end{cor}
\begin{proof}
Up to replacing $\langle\sigma\rangle$ by its $q$-Sylow subgroup and conjugating in $G$, we may assume that $\sigma\in Q$ and furthermore that $e\ge 2$. Let $Q_0:=Q$ and $Q_i:=[Q_{i-1},Q]$, $i\ge 1$ be the descending central series of $Q$ (so that the nilpotency class is defined as the smallest index $i$ with $Q_i=\{1\}$). 
Set $\widehat{Q}_{i}:=\langle Q_{i}\cap G_q ,\sigma\rangle$, $i\ge 0$. 
Note that $\widehat{Q}_{i}\trianglelefteq \widehat{Q}_{i-1}$ for all $i$, since indeed $[Q_{i-1}\cap G_q, 
\langle\sigma\rangle]\subseteq [Q_{i-1}, Q]\cap G_q\subseteq Q_i\cap G_q$. Thus $[\widehat{Q}_{i-1}:\widehat{Q}_i ] \divides[\widehat{Q}_{i-1}\cap G_q:\widehat{Q}_i\cap G_q]\divides q$ by Corollary \ref{cor:normalofHatmostp_doneright}.
As $\widehat{Q}_0 \cap G_q=G_q$, it follows that $\widehat{Q}_{e-2} \cap G_q = \langle Q_{e-2}\cap G_q,\sigma\rangle \cap G_q$ must still be of order $\ge q^2$. Since $\langle\sigma \rangle \cap G_q$ is either trivial or the diagonal under the action of $\sigma$, the subgroup $Q_{e-2}\cap G_q$ must still be nondiagonal. As $[\widehat Q_{e-2}:\widehat Q_{e-1}]\divides q$, 
this implies $Q_{e-1}\ne \{1\}$, i.e., $Q$ is of nilpotency class at least $e$.
\end{proof}

\section{Lower bounds on solvable monodromy groups of polynomials}
\label{sec:lowerbounds}
Let $k$ be a field of characteristic $0$. Our method for proving Theorem \ref{thm:DLS} relies on  ``largeness'' properties for monodromy groups of  polynomials from \cite{BKN}. 
For indecomposable $f,g\in k[X]$, we say that a subgroup $N \leq \Mon_k(f \circ g)$ has a \emph{large kernel} if either 
$$
\soc(\Mon_k(g))^{\deg(f)}\le \ker(N\to \Mon_k(f)),
$$ 
or 
$$\soc(\Mon_k(g))\,\,\textrm{is cyclic and}\,\, \soc(\Mon_k(g))^{\deg(f)-1}\le \ker(N\to \Mon_k(f)).
$$
We say that $f\circ g$ has a \emph{large kernel} if $\Mon_k(f\circ g)$ has a large kernel.
\begin{thm}[Theorems 1.2 and 2.1 in  \cite{BKN}]\label{thm:new-large}\label{thm:largemon}
Let $f,g\in k[X]$ be indecomposable with solvable monodromy. Then either $f\circ g$ has a large kernel or it is linearly equivalent over $\overline{k}$ to a monomial, a Chebyshev polynomial or one of only two other cases holds: 
\begin{itemize}
\item $\Mon_{\overline{k}}(f \circ g)=C_3\times S_4\le S_{12}$ with a  Ritt move;
\item $\Mon_{\overline{k}}(f \circ g)={\rm GL}_2(3)\le S_8$ with no  Ritt move.
\end{itemize}
In the large kernel case, if moreover, $f$ is linearly related over $\overline{k}$ to $X^p$ or $T_p$, and $g$ is linearly related over $\overline{k}$ to $X^q$ or $T_q$, for primes $p,q\geq 2$, then $\ker(\Mon_k(f\circ g)\to \Mon_k(f))$ contains $C_q^p$. 
\end{thm}

We shall need the following extension of Theorem \ref{thm:largemon}, which takes into account normal subgroups of $\Mon_k(f\circ g)$ and strengthens the largeness assertions in some cases.
\begin{cor}
\label{cor:largemon_normal}
Let $f,g\in k[X]$ be indecomposable with solvable monodromy of degrees $p:=\deg(f)$ and $q:=\deg(g)$, respectively, and such that 
$f\circ g$ is not among the exceptional cases of Theorem \ref{thm:largemon}. Let $N\trianglelefteq \Mon_{k}(f\circ g)$ be a normal subgroup containing a cyclic transitive subgroup. Then the following hold:
\begin{enumerate}[leftmargin=*,label={(\arabic*)},ref=(\arabic*)]
\item
$\ker(N\to \Mon_{k}(f))$ is large.
\label{case:largemon_normal_1}
\item If either $q=4$ or $p,q$ are distinct primes, then more strongly $\ker(N\to \Mon_{k}(f)) \ge \soc(\Mon_{k}(g))^p$.
\label{case:largemon_normal_2}
\item If $\Mon_{\oline{k}}(g)$ is noncyclic, and either $q=4$ or $p>2$ is prime, then more strongly $\soc(\Mon_{k}(g))^p$ is a minimal normal subgroup of $\Mon_{k}(f\circ g)$.
\label{case:largemon_normal_3}
\end{enumerate}
\end{cor}
\begin{proof}
Since we will show 3) (which implies 2) and 1)) for $q=4$, we first assume $q\ne 4$. Since we are not in the exceptional cases, either $K:=\ker(G\to\Mon_{k}(f))$ contains $C_q^p$ by Theorem \ref{thm:largemon} or ($p=4$ and $K$ contains a subgroup $C_q^3$). Since $N$ contains a cyclic transitive subgroup,  1) and 2) then follow directly from 
Corollary \ref{cor:normalofHatmostp_doneright},  together with the observation that the $\mathbb{F}_2[S_4]$-permutation module has no submodule of dimension $2$, to deal with the case $(p,q)=(4,2)$. 

Assume now that $\Mon_{\oline{k}}(g)$ is noncyclic (allowing also $q=4$) and $p>2$ is prime. Set $G:=\Mon_{k}(f\circ g)$. 
It then follows from refinements of Theorem \ref{thm:largemon}, 
namely  \cite[Theorem 2.1]{BKN} (in case $p,q$ prime) and \cite[Theorem 2.3]{BKN} (in case $q=4$ and $p$ prime), that $\ker(G/\soc(K)\to \Mon_{k}(f))$ contains a subgroup $C_2^{p-1} \le \textrm{Aut}(C_q)^p$ (when $q$ prime) or a subgroup $C_3^{p-1}\le (S_4/V_4)^p\cong S_3^p$ (when $q=4$)). 
Therefore, \cite[Lemma 3.8]{BKN} implies that $K$ contains $K_0:=\soc(\Mon_{k}(g))^p$ as a {\it minimal} normal subgroup, showing 3) in all cases except for $(p,q)=(2,4),(4,4)$. Those last two cases are treated by a direct inspection with Magma.
\end{proof}

We shall need the following addition for compositions of three polynomials.
The reader might choose to skip part (4) as it is applied only with $q=p=3$ (a case that is also verified directly with Magma). 
For a prime $q$ and an integer $n\ge 2$, recall that $I_n(q)$ denotes the augmentation ideal of the module $\mathbb{F}_q^n$, viewed as the permutation module under a fixed cyclic group $C_n$.
\begin{prop}\label{prop:large-three}
Let $f=f_1\circ f_2\circ f_3$, where $f_1,f_2,f_3\in k[X]$ are indecomposable of degrees $m,p,q$, respectively, with  $\Mon_{\oline k}(f)$ solvable and  $\Mon_{\oline k}(f_3)=C_q$ (for $q$ prime).
Assume   $f_1\circ f_2$ and $f_2\circ f_3$ are either strongly-unique  or 
have monodromy group $D_4$ over $k$.
Assume moreover that none of $\Gamma:=\Mon_k(f_1\circ f_2)$ and $G_2:=\Mon_k(f_2\circ f_3)$ equals $\GL[2](3)$. 
Let $N$ be a normal subgroup of $G_3:=\Mon_k(f)$ containing a cycle $\sigma$ of length $m\gp q$, and 
$K,M$ the kernels of the maps $G_3\to \Gamma=\Mon_k(f_1\circ f_2)$ and $G_3\to \Mon_k(f_1)$, respectively.  
\begin{enumerate}[leftmargin=*,label={(\arabic*)},ref=(\arabic*)]
\item If  $\gp\ne q$ is a prime,
then $N\cap K$   properly contains $I_\gp(q)^m$. 
\label{case:large-three_1}
\item If $\gp=4$, 
then $N\cap K$ contains 
$I_4(q)^m$, and if additionally $q>2$, this containment is proper.
\label{case:large-three_2}
\item If $\gp=q$ and $m\ne 4$, then $N\cap K$ contains an index-$q$ 
subgroup  
of $I_\gp(q)^m$. If moreover $\gp=q=2\ne m$, then the subgroup $(N\cap M)^2$, generated by squares in $N\cap M$, contains {$I_2(2)^m\cong C_2^{m}$}. 
\label{case:large-three_3}
\item If $\gp=q$ and $m=4$, then $N\cap K\supseteq C_q^{3q-4}$. More precisely, if $\widehat{N}_0\le S_{3q^2}$ denotes the image of the stabilizer $N_0\le N$ of a root of $f_1(X)-t$ in its natural action on $3q^2$ points, then the image $\widehat{N_0\cap K}$, of $N_0\cap K$ in $S_{3q^2}$, contains an index-$q$ subgroup of $I_q(q)^3$.
\label{case:large-three_4}
\end{enumerate}
\end{prop}

The proofs of parts \ref{case:large-three_3} and \ref{case:large-three_4} require the following lemma.
\begin{lem}\label{lem:triple-wreath}
Let $q$ be a prime, $m\ge 2$ an integer, 
$H\le S_m$ a primitive group and $G_3\leq \AGL_1(q)\wr \AGL_1(q)\wr H$. Let $M,K$ be the kernels of the projections $\psi: G_3\to H$ and $\pi:G_3\to \AGL_1(q)\wr H$, respectively. Assume that 1) $\pi(G_3)$ contains $C_q^m$, and 2) the image of $M\le (\AGL_1(q)\wr \AGL_1(q))^m$ under projection to a component contains $C_q\wr C_q$. Then there exists an index-$q$ subgroup $U$ of $I_q(q)^m$ such that $U\le [M_q,M_q] \cap K$ for a $q$-Sylow subgroup $M_q\le M$. If moreover $q=2$, then $U\le M^2$, where $M^2$ denotes the subgroup generated by all elements $x^2$, for $x\in M$.
\end{lem}
\begin{proof}
The existence of such an index-$q$ subgroup $U$ of $I_q(q)^m$ such that $U\le [M_q,M_q]\cap K$ is an immediate consequence of \cite[Lemma B.3 with Example B.5]{BKN}.  The additional assertion for $q=2$ follows 
{since every commutator $[x,y]=x^2(x^{-1}y)^2y^{-2}$ is a product of squares}. 
\end{proof}

\begin{proof}[Proof of Proposition \ref{prop:large-three}] 

Due to the assumptions on $\Gamma=G_3/K$, one of the following holds by Theorem \ref{thm:largemon} together with the strengthening \cite[Theorem 2.3]{BKN} for Case iii) below: 
\begin{itemize}
\item[i)] $m$ and $\gp$ are both prime, and $G_3/K$ contains  $C_\gp^m.C_m$.
\item[ii)] $\gp$ is prime, $m=4$, and $G_3/K$ contains $V.S_4$, where $V\subset\mathbb{F}_\gp^4$ is a $3$-dimensional submodule.
\item[iii)] $\gp=4$, $m$ is prime, and $G_3/K$ contains $U.C_m$, where $U\le S_4^m$ contains an index-$3$ subgroup of $A_4^m$ (resp., contains $A_4^2$ in the case $m=2$).
\item[iv)] $\gp=m=4$, and $G_3/K$ contains $V_4^4$, where $V_4\cong C_2\times C_2$ is the Klein $4$-group.
\end{itemize}
Similarly, the assumptions on $G_2=\Mon_k(f_2\circ f_3)$ imply that $G_2$ contains $C_q^\gp.C_\gp$ when $\gp$ is prime, resp. contains $W.S_4$ when $\gp=4$, where $W=I_4(q)\subset\mathbb{F}_q^4$ is a $3$-dimensional submodule.  
Since the image $\oline{M}$ of $M$ in $G_2$ under projection to a component is a normal subgroup of $G_2$ containing a cyclic transitive subgroup, Corollary \ref{cor:normalofHatmostp_doneright} yields that $\oline{M}$ contains $I_\gp(q).C_\gp$ for $\gp\ne 4$. Analogously, when $\gp=4$, Corollary \ref{cor:normalofHatmostp_doneright} implies that $\oline{M}$ contains $W.S_4$ with $W=I_4(q)$ as above, since $W$ becomes a permutation module under the action of $A_4\le S_4$ with no submodule of dimension $2$. 

(1) Assume $\gp\neq q$ is a prime. Then the $C_\gp$-module $C_q^\gp\cong \mathbb F_q[C_p]$ is semisimple\footnote{In fact $C_q^\gp\cong \mathbb F_q[x]/(x^p-1)$ where $\sigma$ acts by multiplication by $x$. As $x^p-1$ is separable for $p\neq q$, the module is semisimple with the trivial module $\mathbb F_q[x]/(x-1)$ appearing with multiplicity $1$.}, and a generator of $C_\gp$ acts nontrivially on every irreducible submodule of $I_\gp(q)$. When $m$ is prime, since $G_3/K$ contains the full $C_\gp^m$ by i) and $C_\gp$ acts faithfully on every nontrivial submodule of $I_\gp(q)$, it follows directly from \cite[Lemma 3.9]{BKN} that $G_3$ contains an extension $I_\gp(q)^m.C_\gp^m.C_m$. 

Similarly by ii), when $m=4$,  the group $G_3/K$ contains two elements $x_1,x_2\in C_\gp^m$ whose supports have exactly one element in common; again, \cite[Lemma 3.9]{BKN} implies that $G_3$ contains $I_\gp(q)^m.V.S_4$. Since $I_\gp(q)^m$ does not contain the diagonal submodule, Corollary \ref{cor:normalofHatmostp_doneright} implies that $N\cap K$ still contains $I_\gp(q)^m$. On the other hand, the $m\gp$-th power of the cyclic transitive subgroup from the assumptions  generates the diagonal submodule inside $N\cap K$. Since $p\neq q$, this implies $N\cap K$  properly contains $I_\gp(q)^m$. 

(2) Now assume $\gp=4$ and $G_2\not\cong \GL[2](3)$. When $m\ne 4$, the kernel of $\psi: \Gamma\to \Mon_k(f_1)$ contains an index-$3$ subgroup of $A_4^m$ (resp., contains all of $A_4^2$ for $m=2$) by iii), which in particular projects to an index-$3$ subgroup of $C_3^m\cong A_4^m/V_4^m$. Hence, it induces a subspace of codimension $\leq 1$ in $\mathbb F_3^m$ which intersects any two-dimensional subspace (resp., induces $\F_3^2$). Thus there exist two elements of order $3$ in $\ker(\psi)\le S_4^m$ whose supports have exactly one element in common. Similarly, when $m=4$, a direct check using Magma\footnote{To do so we run over all subgroups $P\le S_4\wr S_4$ surjecting onto $S_4$, with block kernel acting on a block as $S_4$, and with a cyclic transitive subgroup (there are only four different such $P$).} shows that the kernel of $\psi:\Gamma\to\Mon_k(f_1)$ contains a subgroup $C_3^2\subseteq S_4^4$ in which all nonidentity elements are supported on exactly $3$ components. Thus there exist three elements of order $3$ in $\ker(\psi)$ whose supports have only one element in common. In all cases, it follows from \cite[Lemma 3.9]{BKN} for $q\ge 3$, and from \cite[Lemma B.1 with Example B.2]{BKN} for $q=2$, that $K$ contains $I_4(q)^m \cong C_q^{3m}$. 
For $q\ge 3$, 
{since $I_4(q)^m$ does not contain the diagonal, as in Case (1), it follows from Corollary \ref{cor:normalofHatmostp_doneright} that $N\supseteq I_4(q)^m$, and that this containment is even proper.}
For $q=2$, Corollary \ref{cor:normalofHatmostp_doneright}\ref{case:normalofHatmostp_doneright_1} implies $[I_4(2)^m: N\cap I_4(2)^m] \le 2$. 
On the other hand, the intersection of $N$ with the submodule $W_1\cong I_4(2)$ of $I_4(2)^m$, consisting of elements supported only on the first component, would have to be a submodule of $W_1$ of codimension at most $1$. However, since $I_4(2)$ is the augmentation ideal of the $\mathbb{F}_2[S_4]$-permutation module, it has no submodule of codimension $1$. Thus $N$ contains all of $W_1$, and  by transitivity of the action on components,  $N\supseteq I_4(2)^m$.


(3) Next, assume that $m\neq 4$ and $q=\gp$. 
By i), $G_3/K$ contains $C_q\wr C_m$, and $G_2$ contains $C_q\wr C_q$. 

As long as $q\ne 2$, we first claim that the image $\oline M$ of $M$ in $G_2$ contains $C_q\wr C_q$. Indeed, the image of the stabilizer $G_z\leq \Mon_{\overline{k}}(f)$ of a root $z$ of $f_1(X)-t\in \oline k(t)[X]$   maps onto $\tilde{G}_2:=\Mon_{\overline{k}}(f_2\circ f_3)\supseteq C_q\wr C_q$ via its action on the roots of $(f_2 \circ f_3)(X)-z$. 
Since $\Mon_{\overline{k}}(f_1)\in \{C_m, D_m\}$, the image {$\overline{M \cap G_z}$ of $M \cap G_z$} in {$\tilde{G}_2$} is of index $\leq 2$, and hence  {$[\tilde{G}_2:\overline{G_z\cap M}]\leq 2$}. 
As $q \neq 2$ and {$\tilde{G}_2\supseteq C_q \wr C_q$},
the claim follows immediately.

{Suppose that $q>2$, or $q=2$ and 
	$\oline M = C_2\times C_2$}.
	Then, $G_3$ fulfills the assumptions of Lemma \ref{lem:triple-wreath}. Thus, by the lemma,  there exists a $q$-Sylow subgroup $M_q\le M$ and a subgroup $U\le [M_q,M_q]$ of index at most $q$ in $I_q(q)^m$. To show that $U$ is contained in $N\cap K$ as asserted, it therefore suffices to show that $N\supseteq [M_q,M_q]$. We will achieve this via a twofold application of Corollary \ref{cor:normalofHatmostp_doneright}.
	First, an application of Corollary \ref{cor:normalofHatmostp_doneright}\ref{case:normalofHatmostp_doneright_1} with $N/(N\cap K)\cong NK/K$ (resp.\ $G_3/K$) in the role of $N$ (resp., of $G$) shows that the $q$-Sylow group of $(N\cap M)/(N\cap K)$ is of index dividing $q$ in the one of $M/K$. 
	Next, an application of Corollary \ref{cor:normalofHatmostp_doneright}\ref{case:normalofHatmostp_doneright_1} to $N\trianglelefteq NK$ shows that the $q$-Sylow subgroup of $N\cap K$ is of index dividing $q$ in that of $K$. 
	Both observations together show that the Sylow subgroup $N\cap M_q$ of $N\cap M$ is   of index dividing $q^2$ in $M_q$. Moreover, $N\cap M_q\lhd M_q$. It follows that $M_q/(N\cap M_q)$ is abelian, and hence $N\supseteq [M_q,M_q]\supseteq U$, as claimed. When $q=\gp=2\ne m$, the element $\sigma^{2m}\in (N\cap M)^2$ is diagonal and, in particular is supported on an odd number of components of the $\mathbb{F}_2[C_m]$-module $I_2(2)^m\cong C_2^m$. Therefore $(N\cap M)^2$ contains not only the (unique) index-$2$ submodule, but in fact the whole $I_2(2)^m$.

	To complete the proof of (3), it remains to treat the case $q=\gp=2$ when the image of $M$ in $G_2$ is not equal to the full group $C_2\wr C_2$. Since this image nevertheless contains a cyclic transitive subgroup, it must be $C_4$. Therefore, $M\le C_4^m$, and $M/K=C_2^m$ by (i) above. Now, in case $m=2$, 
	since $\langle \sigma^4\rangle\leq N\cap K$ is the diagonal submodule of $I_2(2)^m$, and  is of index $2$ in it, $N\cap K$ indeed  contains an index-$2$ subgroup of $I_2(2)^m$. If instead $m>2$ is odd, note that $N/(N\cap K)\leq G_3/K$ is a normal subgroup containing a $(2m)$-cycle and hence contains $C_2\wr C_m$ by  Corollary \ref{cor:normalofHatmostp_doneright}.(2) and (i) above. 
	Since in addition $M/K=C_2^m$ and $NK\cap M=(N\cap M)K$ (as $K\leq M$), we get that $(N\cap M)/(N\cap K)\cong (NK\cap M)/K=C_2^m$. Hence $N\cap M =C_4^m$, and therefore $(N\cap M)^2=C_2^m$. 
	
	(4)
	In case $\Mon_k(f_1\circ f_2)=G_3/K$ contains all of $C_q^4$, we may argue just as in Case (3) 
	to obtain that $N\cap K$ contains an index-$q$ subgroup of $I_q(q)^4$, and consequently the projection of $N_0\cap K$ to its orbit of length $3q^2$ (where $N_0 =N\cap M.S_3$ denotes the stabilizer in $N$ of a root of $f_1(X)-t$) contains an index-$q$ subgroup of $I_q(q)^3$.
	
	In view of iii) at the beginning of the proof, we may thus assume $\ker(\Mon_k(f_1\circ f_2)\to \Mon_k(f_1))$ to have a $q$-Sylow group isomorphic to $C_q^3$.
	
	In order to be able to apply Lemma \ref{lem:triple-wreath}, we now descend to the stabilizer $H_3 = M.S_3\le G_3$ of a root $z$ of $f_1(X)-t$ in $G_3$. This group acts transitively on $3q^2$ elements, namely, the roots of $f(X)-t$ not mapping to the given root $z$ of $f_1(X)-t$ under $f_2\circ f_3$. Denote by $\widehat{H}_3$ (resp., $\widehat{M}$, resp.\ $\widehat{K}$) the image of $H_3$ (resp., $M$, resp., $K$) in this action, so that $\widehat{M}$ and $\widehat{K}$ preserve a maximal and a minimal block system in $\widehat{H}_3$, respectively, and they are in fact the full block kernels of these actions: for $\widehat{M}$ this is obvious since an element of $G_3$ stabilizing $3$ out of $4$ maximal blocks must stabilize the last as well; if $\widehat{K}$ were not the full block kernel in the action on the minimal block system, there would have to be an element of $\ker(\Mon_k(f_1\circ f_2)\to \Mon_k(f_1))$ acting nontrivially only on the $q$ roots of $f_1(f_2(X))-t$ mapping to the fixed root $z$ of $f_1(X)-t$ (which is not ``seen" by $\widehat{H}_3$) while fixing all others. However, this would imply $\ker(\Mon_k(f_1\circ f_2)\to \Mon_k(f_1))$ having full $q$-part $C_q^4$, which we have excluded above.
	
	We aim to repeat the argument of (3) with $\widehat{H}_3$ in place of $G_3$. This will show the minimal block kernel $\widehat{K}$ of $\widehat{H}_3$ contains a subgroup $C_q^{3(q-1)-1}$, which a fortiori implies the same for $K$ (and with the descent argument from $\widehat{K}$ to $\widehat{N_0\cap K}$ unchanged).
	To verify the assumptions of Lemma \ref{lem:triple-wreath}, first note that, due to $G_3/K$ containing a subgroup $C_q^3$ and surjecting to $S_4$, the group $\widehat{H}_3/\widehat{K}$ contains $C_q^3.C_3 \cong C_q\wr C_3$, since indeed if the projection of $V\cong C_q^3$ to the union of three (out of four) blocks were not injective, we would obtain $\ker(\Mon_k(f_1\circ f_2)\to \Mon_k(f_1))\supseteq C_q^4$, contradicting our assumptions.
	Moreover, the image of the maximal block kernel $\widehat{M}\trianglelefteq \widehat{H}_3$ under projection to any given of the three components equals the image $\pi(M)$ of $M\trianglelefteq G_3$ under projection to any of the four maximal blocks. But the image $\pi(H_3)$ of a maximal block stabilizer $H_3=M.S_3$ still fulfills $C_q\wr C_q \le \pi(H_3) \le AGL_1(q)\wr AGL_1(q)$ by Theorem \ref{thm:largemon}, and $\pi(H_3)/\pi(M)$ must be a quotient of $H_3/M\cong S_3$; for $q\ge 5$, this shows instantly that $\pi(M)\supseteq C_q\wr C_q$, whereas for $q=3$, the same conclusion is easily verified upon noting additionally that $\pi(M)$ acts transitively.
	Now the argument is completely analogous to Case (3); one should merely note 
	that, since $G_3/K$ contains $C_q^3$, the group $\widehat{N_0}/\widehat{K}$ contains all of $C_q\wr C_3$, and in particular contains a $3q$-cycle.
	\end{proof}

	\section{Diagonality of the kernel} 
	\label{sec:diagkernel}
	Let $k$ be a field of characteristic $0$. 
	We start with the following key theorem on the diagonality of the block kernel for minimally reducible pairs. 
	\begin{thm}\label{thm:diagkern}
	Suppose that $f\in k[X]$ and a Siegel function $g \in k(X)$ form a minimally reducible pair, and that $f=h \circ f_r$ for an indecomposable $f_r\in k[X]$ 
	of solvable monodromy. 
	If $g\notin k[X]$, assume further $\deg(f)$ is odd. 
	Then   $K:=\ker(\Mon_k(f)\ra\Mon_k(h))$
	is either a diagonal subgroup of $\Mon_k(f_r)^{\deg(h)}$ or ($\deg(f_r)=2$ and $K\cong C_2\times C_2$).
	%
	\end{thm}
	\begin{rem}
	\label{rem:kernel} 
	\label{rem:diagker}
	We shall moreover see in the proof below that the case ($\deg(f_r)=2$ and $K\cong C_2\times C_2$) can only occur if
	there is a complete decomposition $g=g_1\circ \cdots \circ g_s$  such that either $\Mon_k(g_s)= S_4$ or $\Mon_k(g_{s-1}\circ g_s)=D_4$.\\ 
	\end{rem}
	
	
	The key ingredient in the proof is the following proposition: 
	\begin{prop}\label{prop:intermediate}
	Let $f\in k[X]$, $g\in k(X)$ be a minimally reducible pair, and $\Omega$  the~common splitting field  
	of $f(X)-t,g(X)-t$ over $k(t)$. 
	Assume $f=h_1\circ f_r$, $g=h_2\circ g_s$  for  indecomposable $f_r\in k[X], g_s\in k(X)$ with solvable $\Mon_k(f_r)$. 
	Let $K=\ker(\Mon_k(f)\to\Mon_k(h_1))$ and $\Omega'=\Omega^{\soc(K)}$. Let $y$ be a root of $g(X)-t$.    
	\begin{enumerate}[leftmargin=*,label={(\arabic*)},ref=(\arabic*)]
		\item If $\deg(f_r)$ is prime,  then  $\Omega'(y)=\Omega$; 
		\label{case:intermediate_1}
		\item If $g\in k[X]$ and $\Mon_k(g_s)=S_4$, then  $\Omega'(y)=\Omega$; 
		\label{case:intermediate_2}
		\item  If 
		$g\in k[X]$ and   
		$\Mon_k(g_s)=C_2$, 
		then $\Omega'(y,y')=\Omega$ for some conjugate $y'$ of $y$.
		\label{case:intermediate_3}
	\end{enumerate}
	\end{prop}
	Note that group theoretically, the equality $\Omega'(y)=\Omega$ means that $V:=\Gal(\Omega/k(y))$ and $\soc(K)$ generate a subgroup isomorphic to a semidirect product $\soc(K)\rtimes V$. Also note that in (3), the last assumption implies that $\Omega'(y)/\Omega'$ is quadratic and hence the conclusion is that $\Omega/\Omega'$ is at most biquadratic. 
	
	To prove Proposition \ref{prop:intermediate}, we shall need the following lemmas. 
	\begin{lem}
	\label{lem:lastminute}
	With the notation of Proposition \ref{prop:intermediate}, let $k(v) = \Omega'\cap k(y)$, and assume additionally that $v=v(y)\in k(y)$ is linearly equivalent over $\overline{k}$ to a polynomial map. Then $[k(y):k(v)]$ is either a prime or is $4$. 
	\end{lem}
	\begin{proof}
	Set $d_r:=\deg(f_r)$. Since $f_r$ is indecomposable,  either  $\Mon_k(f_r)\leq \AGL_1(p_r)$ where $d_r=p_r$ is a prime, 
	or $\Mon_k(f_r)=S_4$ and $d_r=4$, by Proposition \ref{prop:indec_poly}. If $d_r=4$, set $p_r=2$. 
	Note that $K$ is a subdirect power of its image  $\oline K\leq \Mon_k(f_r)$ by Lemma \ref{lem:subdirect}. Moreover $\oline K=S_4$ if $d_r=4$ by Lemma \ref{lem:transitive-normal}.(2) and Remark \ref{rem:abh}(1). 
	Hence, $\oline K$ contains $C_{p_r}$  if $d_r=p_r$ (resp.\ $A_4$ if $d_r=4$), by Lemma \ref{lem:blocknontriv}.  
	Hence   $\soc(K)$ is a $p_r$-elementary abelian group by Lemma \ref{lem:interandsoc}. 
	Since $k(y)/k(v)$ and $\Omega'/k(v)$ are linearly disjoint, it follows that $[k(y):k(v)]=[\Omega'(y):\Omega']$ is a power of $p_r$. 
	
	Since $v=v(y)$ is then linearly related over $\overline{k}$ to a polynomial of prime power degree by assumption, $\Mon_k(v)$ contains a cyclic transitive subgroup, but also contains the regular $p_r$-elementary-abelian normal subgroup $\Gal(\Omega'(y)/\Omega')$, where the latter is identified via restriction with a subgroup of $\Mon_k(v)$. By Lemma \ref{lem:reg_normal_sub}, this implies that $[k(y):k(v)]$ equals $p_r$ or $4$.
	\end{proof}
	
	\begin{lem}\label{clm:exisx0}
	Suppose that $f\in k[X]$ and $g\in k(X)$ form a minimally reducible pair. 
	Write $f=h\circ f_r$ and $g=g_1\circ g_2$ for indecomposable $f_r$ and $\deg(g_2)>1$, and assume $\Mon_k(f_r)$ is solvable with $p_r:=\deg(f_r)$ prime. Then  the fiber product $\mP^1\#_{g,h}\mP^1\ra \mP^1$ of $g$ and $h$ factors through $f$. 
	Equivalently, letting $x,y$ be the roots of $f(X)-t$ and $g(X)-t$, resp., in an extension of $k(t)$, and setting ${u=f_r(x)}$, there exists a $k(t)$-conjugate $x_0$ of $x$ such that $k(x_0) \subseteq k(u,y)$. {Furthermore, $x_0$ can be chosen as a $k(u)$-conjugate of $x$.}
	\end{lem}
	\begin{proof}
	{Let $v=g_2(y)$}. Since $f$ and $g$ is a minimally reducible pair, $k(x)$ and $k(v)$ (resp., $k(y)$ and $k(u)$) are linearly disjoint over $k(t)$. 
	Let $\Omega_x/k(u,v)$ (resp., $\Omega_y/k(u,v)$) be the Galois closure of $k(v,x)/k(u,v)$ (resp., $k(u,y)/k(u,v)$). Since $k(u,v)$ and $k(x)$ are linearly disjoint over $k(u)$,  we  identify $\Gal(\Omega_x/k(u,v))$ with a subgroup of $\Mon_k(f_r)$, as in Remark \ref{rem:Gal-closure}. Since $f_r$ is indecomposable of degree $\neq 4$ with solvable monodromy,  these subgroups identify with subgroups of $\AGL_1(p_r)$, where $p_r=\deg(f_r)$ is prime, by Proposition \ref{prop:indec_poly}. Since $p_r$ is prime and $k(v,x)/k(u,v)$ and $k(u,y)/k(u,v)$ are not linearly disjoint, we have $k(v,x)\subseteq \Omega_y$ and hence $\Omega_x \subset \Omega_y$. Since the image of $\Gal(\Omega_y/k(u,y))$ in $\Gal(\Omega_x/k(u,v))\leq \AGL_1(p_r)$ is intransitive, Lemma  \ref{lem:onAGL1} gives a root $x_0$ of $f_r(X)-u$ fixed by this image. This root is a $k(u)$-conjugate of $x$ that is contained in $k(u,y)$, yielding the desired inclusion $k(v,x_0)\subseteq k(u,y)$, cf.\ Figure \ref{diag1}. As $k(u,y)$ is the compositum of the linearly disjoint extensions $k(y)/k(t)$ and $k(u)/k(t)$, it is the function field of the fiber product of $g$ and $h$, so that the inclusion $k({x_0})\subseteq k(u,y)$ implies that this fiber product factors through $f$. 
	%
	\begin{figure}[h!]
		{\centering
			$$
			\xymatrix{
				& \Omega_x \ar@{-}[r]& \Omega_y \ar@{-}[d]\\
				k(x_0)  \ar@{-}[d]_{p_r} \ar@{-}[r]&k(v,x_0) \ar@{-}[u] \ar@{-}[r] \ar@{-}[d]^{p_r}& k(u,y) \ar@{-}[d] \\
				k(u) \ar@{-}[d] \ar@{-}[r]&k(u,v) \ar@{-}[d] \ar@{-}[ru]& k(y) \\
				k(t) \ar@{-}[r]&k(v) \ar@{-}[ru]& 
			}
			$$}
		\caption{Diagram of relevant field extensions in the proof of Lemma \ref{clm:exisx0}}\label{diag1}
	\end{figure}
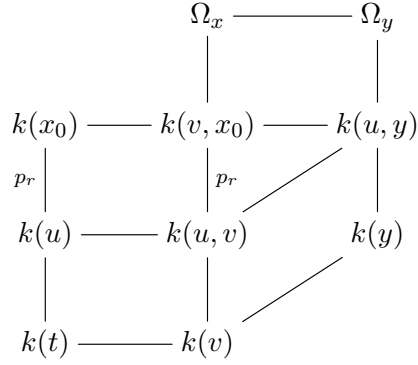
	\end{proof}
	
	\begin{proof}[Proof of Proposition \ref{prop:intermediate}]
	Let $x$ be a root of $f(X)-t\in k(t)[X]$, 
	and set $u=f_r(x)$ and $k(v)=\Omega'\cap k(y)$. 

	(1) Assume first  $p_r=\deg(f_r)$ is prime. Since   in  addition $\Gamma_r=\Mon_k(f_r)$ is solvable, it embeds into $\AGL_1(p_r)$ by Proposition \ref{prop:indec_poly}. 
	Thus, $\Mon_k(f)$ is a subgroup of $\AGL_1(p_r)\wr \Mon_k(h_1)$. 
	As in the proof of Lemma \ref{lem:lastminute},  $K$ is a subdirect power of its image  $\oline K\leq \AGL_1(p_r)$ by Lemma \ref{lem:subdirect},  and  $\oline K$ contains $C_{p_r}$ by Lemma \ref{lem:blocknontriv}.  Hence, by Lemma \ref{lem:interandsoc}, $\soc(K)= C_{p_r}^{{\deg(h_1)}} \cap K$, and in particular $\soc(K)$ is a $p_r$-elementary abelian group.  
	. 

	On the one hand, note that for any $k(t)$-conjugate $u_0$ of $u$ and roots $x_1, x_2$ of $f_r(X)-{u_0}$, we have $\Omega'({x_1})=\Omega'({x_2})$: Indeed, since $\Omega'/k(t)$ is Galois (as $\soc(K)\leq K$ is characteristic) and $\Omega'\neq \Omega$,  we have $k({x_1}) \not\subset \Omega'$, and hence $\Omega'/k({u_0})$ and $k({x_1})/k({u_0})$ are linearly disjoint. This implies that $f_r(X)-{u_0}$ is irreducible over $\Omega'$. As $\soc(K)$ is abelian, $\Omega'({x_1})/\Omega'$ is Galois, and hence $\Omega'({x_1})=\Omega'({x_2})$.
	
	On the other hand, since $k(y)/k(t)$ and $k(u)/k(t)$ are linearly disjoint, $h_1(X)-t$ remains irreducible over $k(y)$, and hence $V:=\Gal(\Omega/k(y))$ acts transitively on the $k(t)$-conjugates of $u$.
	To apply Lemma \ref{clm:exisx0}, note that $k(y)/k(v)$ is  nontrivial since  $\Omega$ is the Galois closure of $k(y)/k(t)$ 
	and since $\Omega'\neq \Omega$ (as $p_r\divides \soc(K)$). Thus the lemma implies that some $k(u)$-conjugate $x_0$ of $x$ is contained in $k(y,u)$. 
	Hence, $k(v,x_0^\sigma)\subset k(u^\sigma,y) \subset \Omega'(y)$ for every $\sigma \in V=\Gal(\Omega/k(y))$. 
	Note that since $\Omega'(y)/\Omega'$ is abelian, and $x_0^\sigma$ is contained $\Omega'(y)$,  so is every $k(u^\sigma)$-conjugate of $x_0^\sigma$. Combining this with the transitivity of $V=\Gal(\Omega/k(y))$ on $k(t)$-conjugates of $u$, we see that $\Omega'(y)$ contains all $k(t)$-conjugates of $x$, that is, $\Omega'(y)= \Omega$. 
	
	(2) Assume $g\in k[X]$ and  $\Mon_k(g_s)=S_4$. 
	Since the conclusion follows from (1) if $\deg(f_r)$ is prime, we may assume $\Mon_k(f_r)=S_4$  in view of Proposition \ref{prop:indec_poly}. 
	Let $\Omega_u,\Omega_{v}$ and $\Omega_x, \Omega_y$ be the Galois closures of $k(u)/k(t),k(v)/k(t)$, and $k(x)/k(u), k(y)/k(v)$, respectively.

	Since $f,g$ form a minimally reducible pair, $g_s$ is right-unique by Proposition \ref{lem:norittstepp}\ref{case:norittstepp_5}. Thus, Lemma \ref{lem:lastminute} implies that $k(v)=k(g_s(y))$. Moreover, $\deg(f)=\deg(g)$ by Corollary \ref{cor:min-red-pol}. Denote this degree by $n$. 
	Since $\deg(h_1)=\deg(h_2)=n/4$, this is also the ramification index of  $\Omega_u\Omega_{v}/k(t)$ over $\infty$ by Remark \ref{rem:abh}(1) (Abhyankar's lemma). 
	Hence $\Gal(\Omega_x\cdot \Omega_u\Omega_{v}/\Omega_u\Omega_{v})$ and $\Gal(\Omega_y\cdot \Omega_u\Omega_{v}/\Omega_u\Omega_{v})$ are also $S_4$ by Lemma \ref{lem:transitive-normal}(2), and hence so are $\Gal(\Omega_x\Omega_u(v)/\Omega_u(v))$, $\Gal(\Omega_y\Omega_u(v)/\Omega_u(v))$. 
	
	
	We next show that $\Omega_y\Omega_u\supseteq \Omega_x$. 
	Since $k(x,v)/k(u,v)$ and $k(u,y)/k(u,v)$ are not linearly disjoint but both are linearly disjoint from $\Omega_u(v)$, the extensions $\Omega_u(x,v)/\Omega_u(v)$ and 
	$\Omega_u(y)/\Omega_u(v)$ are not linearly disjoint as well. As $\Omega_x\Omega_u(v)/\Omega_u(v)$ and $\Omega_y\Omega_u/\Omega_u(v)$ are both $S_4$-extensions, this implies $\Omega_x\Omega_u(v)=\Omega_y\Omega_u$. In particular $\Omega_y\Omega_u\supseteq \Omega_x$. 
	
	We claim that $\Omega_y\Omega_u=\Omega$. As $V$ preserves $\Omega_u$ and fixes $y$ and $v$, it also preserves $\Omega_u(v),\Omega_u(y)$ and hence  also the Galois closure $\Omega_y\Omega_u$ of $\Omega_u(y)/\Omega_u(v)$. By the minimal reducibility assumption, $k(y)/k(t)$ and $k(u)/k(t)$ are linearly disjoint, and hence $V$ is transitive on the conjugates of $u$. As $\Omega_y\Omega_u\supseteq k(x)$,  for every conjugate $u'$ of $u$, by applying an element of $V$, we get that $\Omega_y\Omega_u$ contains $\Omega_u(x')$ for a root $x'$ of $f_r(X)-u'\in k(u')[X]$. 
	Moreover, since  $\Omega_y\Omega_u$ contains the Galois closure $\Omega_x$ of $k(x)/k(u)$, this implies that $\Omega_y\Omega_u$ also contains the Galois closure of $k(x')/k(u')$. 
	Since $V$ is transitive on $k(t)$-conjugate of $u$, it follows that $\Omega_y\Omega_u$ contains all conjugates $x'$ of $x$, 
	proving the claim.

	Finally to derive (2), note that  $\Omega'$ and $k(y)$ are linearly disjoint over $k(v)$, and hence the compositum $\Omega_{y}\Omega'$ coincides with the Galois closure of $\Omega'(y)/\Omega'$ by Remark \ref{rem:Gal-closure}.
	However, as $\soc(K)$ is a $2$-elementary abelian group, $\Omega'(y)/\Omega'$ is Galois, and hence $\Omega'(y)=\Omega_y\Omega'$ which contains $\Omega$ by the last claim.
	
	(3)  
	Since the conclusion follows from (1) if $\deg(f_r)$ is prime, we may assume $\Mon_k(f_r)=S_4$ by Proposition \ref{prop:indec_poly}. Since $g\in k[X]$, the right factor $g_s$ is right unique by Proposition \ref{lem:coprime}\ref{case:norittstepp_4}. Due to Lemma \ref{lem:lastminute}, $[k(y):k(v)]$ is either $2$ or $4$; in the latter case, however, $v=v(y)$ would be a composition of two quadratic polynomials (since $g_s$ is a right-unique factor), contradicting Lemma \ref{lem:2-power}. Thus, $[k(y):k(v)]=2$, and therefore $k(v)=k(g_s(y))$. 
	Since $k(x,v)$ contains the quadratic extension $k(u,y)/k(u,v)$ by minimal reducibility and since $\Omega'\supseteq k(u,v)$, we have $\Omega'(x)\supseteq \Omega'(y)$. Since $\Omega'/k(t)$ is Galois, $[k(y'):\Omega'\cap k(y')]=2$ for any $k(t)$-conjugate $y'$ of $y$, so $\Omega'(y') \subseteq \Omega'(x)$ by the same argument. Thus $\Omega'(x)=\Omega$. Therefore either $[\Omega'(x):\Omega']=2$, in which case $\Omega=\Omega'(x)=\Omega'(y)$, or  $[\Omega'(x):\Omega']=[k(x):k(u)]=4$, in which case $\Omega=\Omega'(x)=\Omega'(y,y')$ for some $k(t)$-conjugate $y'$ of $y$. 
	%
	%
	\end{proof}

	\begin{proof}[Proof of Theorem \ref{thm:diagkern}]
	Let $x,y$ be roots of $f(X)-t,g(X)-t$, respectively,  in an extension of $k(t)$. 
	As in Proposition \ref{prop:intermediate}, let
	$u=f_r(x)$,  let $\Omega'$ be the fixed field of $\soc(K)$ and $k(v)=\Omega'\cap k(y)$.
	
	Assume at first that $g$ is not a polynomial, but merely a Siegel function, and thus by assumption that $\deg(f)$ is odd. 
	In particular, $\Mon_k(f_r)\leq \AGL_1(p_r)$ for an odd prime $p_r=\deg(f_r)$ by Proposition \ref{prop:indec_poly}. 
	It follows that $\soc(K)=C_{p_r}^{\deg(h)} \cap K$ by Lemmas \ref{lem:subdirect} and \ref{lem:interandsoc}, and in particular $[k(y):k(v)]=[\Omega'(y):\Omega']=p_r^\ell$ is odd.
	
	We claim that  $v=v(y)$ is linearly related  over $\overline{k}$ to a polynomial of odd degree. 
	Indeed, write $g=h \circ v$ for some $h \in k(X)\setminus k$ and pick $\lambda \in h^{-1}(\infty)$. 
	Since $\Omega$ is the Galois closure of $k(y)/k(t)$, and its ramification over $\infty$ is the same as of the totally ramified extension $k(x)/k(t)$ by Remark \ref{rem:abh}(1) (Abhyankar's lemma),  the ramification index for $g$ over $\infty$ is odd as well. 
	Hence so is the ramification index for $v$ at each place in $v^{-1}(\lambda)$. If $|v^{-1}(\lambda)|=2$, it follows that $\deg(v)$ is even, contradicting $\deg(v)=p_r^\ell$ is odd.
	Since $|v^{-1}(\lambda)| \leq 2$ by assumption, we get $|v^{-1}(\lambda)|=1$, and hence $v$ is indeed linearly related over $\oline k$ to a polynomial of odd degree. 
	
	Lemma \ref{lem:lastminute} now shows that $v$ is in fact of prime degree $p_r$. This implies $|\soc(K)|=p_r$ by Proposition \ref{prop:intermediate}, and thus the diagonality of $K$ by Remark \ref{rem:diag}.

	Henceforth assume $g$ is a polynomial. The argument in the last paragraph  applies when  $p_r=\deg(f_r)$ is odd, even when $g$ is a polynomial. 
	
	We may therefore assume  $ \{\deg(f_r),\deg(g_s)\}\subseteq \{2,4\}$ by Proposition \ref{lem:coprime}\ref{case:norittstepp_4}. 
	By Lemma \ref{lem:lastminute}, $v$ is either indecomposable or the composition of two quadratic polynomials. 
	Moreover, if $v=v_1\circ v_2$ with $v_1,v_2$ quadratic, then $p_r=\deg(f_r)=2$ by Lemma \ref{lem:2-power} and the fact that $f_r$ is right-unique by Proposition \ref{lem:norittstepp}\ref{case:norittstepp_5}. 
	In particular, if $\Mon_k(f_r)=S_4$, then $v$ is indecomposable of degree $2$ or $4$ and hence $\Omega=\Omega'(y,y')$ is a quadratic or biquadratic extension of $\Omega'$ for a conjugate $y'$ of $y$, by Proposition \ref{prop:intermediate}, parts \ref{case:intermediate_2} and \ref{case:intermediate_3}. Thus in the latter case either $\soc(K)\cong C_2$ or $C_2\times C_2$ as desired. Hence  $\soc(K)$ is diagonal by Lemma \ref{lem:subdirect}, and hence so is $K$ by Remark \ref{rem:diag}. 
	
	Henceforth assume $\Mon_k(f_r)=C_2$, so that $\Omega=\Omega'(y)$ by Proposition \ref{prop:intermediate}\ref{case:intermediate_1}. 
	If $v$ is indecomposable and $\Mon_k(g_s)=C_2$ (resp., $\Mon_k(g_s)=S_4$), 
	it further follows that $|K|=|\soc(K)|=[\Omega:\Omega']=[k(y):k(v)]=2$ (resp., $=4$), 
	and hence $K$ is diagonal by Lemma \ref{lem:subdirect} (resp., $K\cong C_2\times C_2$). 
	Finally assume that $v=v_1\circ v_2$ is a composition of two quadratics.  \
	It then follows that $|K|=|\soc(K)|=[k(y):k(v)]=4$ and  $K\cong \soc(K)\cong C_2\times C_2$. 
	Note that in this case $\Mon_k(v)\neq C_4$ by Proposition \ref{lem:coprime}\ref{case:norittstepp_2}, and hence $\Mon_k(v)=D_4$.
	\end{proof}

	\section{The solvable case of Theorem  \ref{thm:DLS-C}}
	\label{sec:solvable}
	Let $k$ be a field of characteristic $0$. In this section, we focus on the solvable case of Theorem \ref{thm:DLS-C}, while the nonsolvable case is discussed in \S\ref{sec:nonsolv}.

	\begin{thm}\label{thm:poly-solvable}
	Suppose 
	$f(X)-g(Y)\in k[X,Y]$ is reducible 
	for $f,g\in k[X]\setminus k$  with solvable monodromy groups. Then either 1) $f=h\circ f_1,g=h\circ g_1$ have a nontrivial common left composition factor $h\in k[X]$, $\deg(h)\geq 2$ for some $f_1,g_1\in k[X]\setminus k$; or 2) 
	$f=(\mu\circ D_{4,\alpha})\circ f_1$ and $g=(\mu\circ (-\frac{1}{4}D_{4,2\alpha}))\circ g_1$, for some 
	linear $\mu\in k[X]$ and $\alpha\in k$.
	\end{thm}
	We shall need the following lemmas which impose restrictions on the right-most composition factors of a polynomial whose monodromy group has a ``diagonal kernel property" in the sense of Section \ref{sec:diagkernel}.
	
	\begin{lem}
	\label{lem:2step_solv}
	Let $f=h\circ f_{r-1}\circ f_r\in k[X]$ be  such that $f_r, f_{r-1}$ are indecomposable of solvable monodromy, $q:=\deg(f_r)$ is prime, and $f_{r-1}\circ f_r$ is either strongly-unique or has monodromy group $D_4$. Let $\gp:=\deg(f_{r-1})$ and $K:=\ker(\Mon_k(f)\to \Mon_k(h\circ f_{r-1}))$. If $(q,\gp) \ne (2,4)$, then $|K|$ is divisible by $q^{\gp-2}$. If moreover $q\ne \gp$ are distinct primes, then $|K|$ is even divisible by $q^{\gp}$. 
	In particular, if $K$ is diagonal or $K\cong C_2\times C_2$, then $(q,\gp) \in \{(2,2), (3,3), (2,4)\}$.
	\end{lem}
	
	\begin{proof}    
	Set $G:=\Mon_k(f)$, let $N:=\ker(G\to \Mon_k(h))\le (\Mon_k(f_{r-1}\circ f_{r}))^{\deg(h)}$, and let  
	$N_0$ be the normal subgroup of $G$ generated by $ I_{\infty}\cap N$, where $I_{\infty}\le G$ is an inertia group at $t\mapsto \infty$. Write $\overline{N}_0$ for the image of $N_0$ under projection to a component.
	Note that $\overline{N}_0$ is a normal subgroup of $\Mon_k(f_{r-1}\circ f_{r})$ containing a cyclic transitive subgroup, namely the projection of $I_{\infty}\cap N$. Let $K_2={\rm ker}(\Mon_k(f_{r-1} \circ f_r) \rightarrow \Mon_k(f_{r-1}))$.
	
	We distinguish the following three cases:
	
	1) $\gp$ a prime different from $q$. 
	Let $\overline{K\cap N_0}\leq \oline{N}_0$ be the image of $K\cap N_0$ on a component. Then $\overline{N}_0/\overline{K\cap {N}_0}$ is a quotient of $N_0/(K \cap N_0)$, which by Corollary \ref{cor:3-useful}, is a subdirect power of $\overline{N}_0/(K_2 \cap \overline{N}_0)=C_\gp$. On the other hand, since $\gcd(q,\gp)=1$, Corollary \ref{cor:largemon_normal}\ref{case:largemon_normal_2} implies that $|\overline{N}_0|$ is divisible by $q^\gp$. This is only possible if $|\overline{K\cap {N}_0}|$, and a fortiori $|K|$, is also divisible by $q^\gp$.  In particular, $K$ cannot be diagonal, nor $C_2\times C_2$ if $q=2$, as $p \geq 3$.
	
	2) $\gp=q>3$. As above, $\overline{N}_0/\overline{K\cap N_0}$ is a quotient of a subdirect power of $C_\gp=C_q$, and hence must be abelian. On the other hand,  by Corollary \ref{cor:largemon_normal}\ref{case:largemon_normal_1}, $\ker(\overline{N}_0\to C_\gp) = K_2\cap \overline{N}_0$ contains $I_q(q)$, the unique codimension-$1$ submodule of the $C_q$-module $\mathbb{F}_q^q$.
	Since the only quotient module of $I_q(q)$ with trivial $C_q$-action is the one-dimensional one, it follows that the $q$-Sylow subgroup of $(K_2\cap \overline{N}_0)/(\overline{K\cap N_0})$ is of order dividing $q$, and thus $|\overline{K\cap N_0}|$, and a fortiori $|K|$, is divisible by $\frac{1}{q}\cdot |I_q(q)|=q^{q-2}$. In particular, $K$ cannot be diagonal since $q>3$.
	
	3) $\gp=4$ and $q\ne 2$. Then, by Corollary \ref{cor:largemon_normal}\ref{case:largemon_normal_1}, $\ker(\overline{N}_0\to S_4)$ contains a submodule $V\cong C_q^{3}$. However, $\overline{N}_0/\overline{K\cap N_0}$ is again a quotient of a subdirect power of $S_4$.
	For $q\ge 5$, since $q$ is coprime to $|S_4|$, this  implies that $|K|$ is divisible by $q^3$. For $q=3$, note that the above forces $\overline{N}_0/\overline{K\cap N_0}$ to have an abelian $3$-Sylow group. Since the action of a $3$-cycle in $S_4$ on a two-dimensional quotient of $V$ would be nontrivial, this implies that the quotient  $V/(\overline{K}\cap V)$ is at most one-dimensional. Therefore $|\overline{K\cap N_0}|$, and a fortiori $|K|$, is divisible by $q^2$. In total, $|K|$ is divisible by $q^{p-2}$ for all $q\geq 3$, and $K$ is not diagonal.
	\end{proof}

	We shall also need the following lemma for exceptional compositions of three factors. 
	\begin{lem}\label{lem:3-factors}
	Suppose $f=h\circ f_{r-2}\circ f_{r-1}\circ f_r$ for indecomposable $f_i\in  k[X]$   with cyclic $\Mon_{\overline{k}}(f_r)=C_q$ and solvable $\Mon_k(f_{r-2}\circ f_{r-1})$. Set $p:=\deg(f_{r-1})$ and $m:=\deg(f_{r-2})$.
	Assume $f_{r-2}\circ f_{r-1}$ is right-unique, $f_{r-1}\circ f_r$ is either strongly-unique or 
	has monodromy group $\Mon_k(f_{r-1}\circ f_r)=D_4$, and either:
	{(1) $p=q$ 
		or,  
		(2) $q=2, p=4$, and $\Mon_k(f_{r-1}\circ f_r)\not\cong \GL[2](3)$.} 
	If  $K:=\ker(\Mon_k(f)\to \Mon_k(h\circ f_{r-2}\circ f_{r-1}))$ is  
	diagonal or  $K\cong C_2\times C_2$, then 
	$f_{r-2}\circ f_{r-1}\circ f_r$ is of degree $8$ or $16$. 
	\end{lem}
	For the reader who wishes to skip the details of low-degree computation, we indicated throughout the proof when  these were also verified using Magma.  
	\begin{proof}
	Let $N,K,K_3$ be the kernels of the  morphisms  $\Mon_k(f)\to \Mon_k(h)$, $\Mon_k(f)\to \Mon_k(h\circ f_{r-2}\circ f_{r-1})$, and $\Mon_k(f_{r-2}\circ f_{r-1}\circ f_r)\to \Mon_k(f_{r-2}\circ f_{r-1})$,  respectively.
	Let  $N_0\leq N$ be the normal subgroup of $\Mon_k(f)$ generated by $I_\infty\cap N$,
	where $I_\infty$ is an inertia group at $\infty$. Let $\oline N_0,\oline K$ denote the images of $N_0,K$ in $G_3=\Mon_k(f_{r-2}\circ f_{r-1}\circ f_r)$, respectively, so that $\oline N_0$ contains a cyclic transitive subgroup as in Lemma \ref{lem:2step_solv}. 
	Note that $N_0/(K \cap N_0)$ is a subdirect power of $\oline N_0/(K_3\cap \oline N_0)$ by Corollary \ref{cor:3-useful}. 
	
	(1) Assume $p=q$. We first prove the assertion  in the case where $m\ne q$ is prime.
	Since the inertia subgroup over $\infty$ in $\Mon_k(f_{r-2})\leq \AGL_1(m)$ is cyclic of order $m$ and  $\oline K\leq K_3$, the image of $N_0$ in $\Mon_k(f_{r-2})$ is also cyclic of order $m$. Thus, $\oline N_0/(K_3\cap \oline N_0)$ has an elementary abelian $q$-Sylow subgroup. As $\oline N_0/\oline{K\cap N_0}$ is a quotient of a subdirect power of it, its $q$-Sylow is elementary abelian as well. 
	
	
	Let $I_q(q)$ and $D$ denote the augmentation and diagonal submodules of $\mathbb F_q[C_q]$, respectively, 
	with respect to the action of the $q$-cycle in $\Mon_k(f_{r-1})\supseteq C_q$. 
	Since $\oline N_0/(K_3\cap \oline N_0)$ contains an $(mq)$-cycle, it contains the full socle $C_{q}^{m}$ of $\Mon_k(f_{r-1})^{m}$ by Corollary \ref{cor:largemon_normal}\ref{case:largemon_normal_2}.
	Since  $f_{r-2}\circ f_{r-1}$ is strongly-unique (as $m\ne p$) and $f_{r-1}\circ f_r$ is either strongly-unique or has monodromy group $D_4$,  
	Proposition \ref{prop:large-three}\ref{case:large-three_3} shows that the $q$-Sylow subgroup $Q$ of $K_3\cap \oline N_0$ contains  
	an index-$q$ subgroup of  
	$I_q(q)^{m}$. 
	If  $q\geq 3$, we claim that the action of $C_q^{m}\leq \oline N_0/(K_3\cap \oline{N}_0)$ on $Q/(\oline K \cap Q)\leq (K_3\cap \oline N_0)/(\oline K \cap \oline N_0)$ 
	is nontrivial, contradicting the fact that the $q$-Sylow subgroup of {$\oline N_0/\oline{K\cap N_0}$}, and hence a fortiori the one of {$\oline N_0/\oline{ K}\cap \overline{N_0}$} is abelian.
	To see the claim, 
	let $\sigma \in \oline N_0$ be a lift of the generator $\oline\sigma$ of some component $C_q$ of $C_q^m$, and $\Delta\subset \{1,\dots, mq\}$ the support of $\oline\sigma$ (of size $q$). Since
	the projection of $Q\trianglelefteq \overline{N}_0$ to $\Delta$ must equal $I_q(q)$, there exists $x\in Q$ whose projection to $\Delta$ is nondiagonal. Then $[x,\sigma]\in Q$ is not supported outside of $\Delta$, but
	is nontrivial in $\Delta$, and thus $[x,\sigma]$ cannot lie in the diagonal $\overline{K}$, showing that $\sigma$ acts nontrivially on $Q/(\overline{K}\cap Q)$. 
	
	If $q=2$, then the subgroup $\oline N_0(2)^2$, generated by squares in the $2$-Sylow subgroup $\oline N_0(2)$ of $\oline N_0$, contains $C_2^m$ by Proposition \ref{prop:large-three}\ref{case:large-three_3}, hence $\oline N_0(2)^2/(\oline K\cap \oline N_0(2)^2)$ is nontrivial as $m\geq 3$, contradicting the fact that the $2$-Sylow subgroup of $\oline N_0/(\oline K\cap \oline N_0)$ is elementary abelian.  
	
	Assume now that $m=4$ or $m=q$. 
	Note that due to Lemma \ref{lem:2step_solv} and since we  avoid degree $8$ or $16$ maps,  we  reduce to the case $q=3$. Note that this case was also verified using Magma.  
	Assume first that, additionally, $f_{r-2}\circ f_{r-1}$ is even strongly-unique. 
	We may then invoke Proposition \ref{prop:large-three}\ref{case:large-three_3} (for $m=3$) or \ref{prop:large-three}\ref{case:large-three_4} (for $m=4$), which shows that $K_3\cap \overline{N}_0$ contains a subgroup $C_3^5$. It thus follows \footnote{In particular, when $m=4$, we need to use the restriction of the $3$-Sylow subgroup to its orbit of length $27$; hence the technical wording of Proposition \ref{prop:large-three}\ref{case:large-three_4}.} from Corollary \ref{cor:nilp_class} that the $3$-Sylow group of $\overline{N}_0$ is of nilpotency class $\ge 5$, and hence that the nilpotency class of its quotient by the diagonal subgroup $\overline{K}$ is $\ge 4$. On the other hand, the $3$-Sylow group of $\overline{N}_0/(K_3\cap \overline{N}_0)\le S_3\wr S_4$ is of class at most $3$, a contradiction.
	
	Assume finally that $f_{r-2}\circ f_{r-1}$ is right-unique but not strongly-unique. By definition, this is only possible for $m=q$, and in this case, $\Mon_k(f_{r-2}\circ f_{r-1})\le \AGL_1(q^2)$, i.e., $f_{r-2}\circ f_{r-1}$ is linearly related over $\overline{k}$ to $T_{q^2}$ or $X^{q^2}$. But then $\ker(\Mon_k(f_{r-2}\circ f_{r-1}\circ f_r)\to \Mon_k(f_{r-2}\circ f_{r-1}))$ contains the full group $C_q^{q^2}$ by \cite[Proposition 6.1 and Remark 6.2]{BKN}, and thus $K_3\cap \overline{N}_0\supseteq C_{q}^{q^2-1}$ by Corollary \ref{cor:normalofHatmostp_doneright}. This again yields an immediate contradiction to the fact that $N_0/(K\cap N_0)$ is a subdirect power of $\overline{N}_0/(K_3\cap \overline{N}_0)=C_{q^2}$.
	
	(2) Assume $q=2, p=4$, and $\Mon_k(f_{r-1}\circ f_r)\not\cong \GL[2](3)$. First assume that $m$ is an odd prime.
	As above, the projection of $\oline N_0$ to $\Mon_k(f_{r-2})$ is $C_m$. Since the image of the action of $\oline N_0/(K_3\cap \oline N_0)$ in $\Mon_k(f_{r-1})$ is normal and contains a $4$-cycle (generating an inertia group over $\infty$), this image must be the full group $\Mon_k(f_{r-1})=S_4$. Hence a $2$-Sylow subgroup $P_2$ of $\oline N_0/(K_3\cap \oline N_0)$ has nilpotency class $2$. 
	Let $I_4(2)\leq C_2\wr S_4$ denote the  $S_4$-invariant subgroup of $C_2^4$  consisting of elements which sum to $0$. Then  $\oline N_0\supseteq I_4(2)^m$ by Proposition \ref{prop:large-three}\ref{case:large-three_2}. Since $[P_2,P_2]\leq \oline N_0/(K_3\cap \oline N_0)$ projects to $[D_4, D_4]\ne \{1\}$ on each of the $m$ coordinates, it is supported on all of these coordinates. Since moreover the action of each component $S_4$ on $I_4(2)$ is faithful, the commutator (of the lifts to $\oline N_0$) of $[P_2,P_2]$ with $I_4(2)^m\leq K_3\cap \oline N_0$ is of $2$-rank at least $m\geq 3$. As $\oline K$ is of rank $\leq 2$, we get that the $2$-Sylow subgroup of $\oline N_0/\oline{K\cap N_0}$ has nilpotency class $>2$. This contradicts the fact that {$\overline{N_0}/\overline{K\cap N_0}$} is a quotient of a subdirect power of $\oline N_0/(K_3\cap \oline N_0)$, since $P_2$ is of nilpotency class $2$. 
	%
	%
	
	Asume finally that $m=4$. Note that this case is also verified using Magma. 
	Then Proposition \ref{prop:large-three}\ref{case:large-three_2} yields that $\overline{N}_0\supseteq I_4(2)^4\cong C_2^{12}$. It follows from Corollary \ref{cor:nilp_class} that the $2$-Sylow subgroup of $\overline{N}_0$ is of nilpotency class at least $12$, and thus that the $2$-Sylow group of $\overline{N}_0/\overline{K\cap N_0}$ still has nilpotency class at least $10$. However, the $2$-Sylow group $C_2\wr C_2\wr C_2\wr C_2$ of $S_4\wr S_4$ has nilpotency class $8$, again contradicting the fact that $\overline{N}_0/\overline{K \cap N_0}$ is  a quotient of a subdirect power of $\overline{N}_0/(K_3\cap \overline{N}_0) {\leq S_4\wr S_4}$.
\end{proof}

\begin{proof}[Proof of Theorem \ref{thm:poly-solvable}]
By possibly replacing $f,g$ by left factors of theirs, we may  assume that $f,g$ is a minimally reducible pair. Hence $f(X)-t$ and $g(X)-t$ have the same Galois closure over $k(t)$ by Lemma \ref{lem:friedarg} and $\deg(f)=\deg(g)$ by Corollary \ref{cor:min-red-pol}. 
Set $n:=\deg(f)$ and write $f=f_1\circ \cdots\circ f_r$ for indecomposable $f_i\in k[X]$. For  $r=1$, either $n$ is prime or $n=4$ by Proposition \ref{prop:indec_poly}. The claim then follows from Lemma \ref{lem:onAGL1} when $n$ is prime and since index-$4$ subgroups are conjugate in $S_4$ when $n=4$. 

Henceforth assume $r\geq 2$.  
By Proposition \ref{cor:right-unique}, $f_r$ is right-unique, and moreover $f_{r-1}\circ f_r$ is either strongly-unique or 
$\Mon_k(f_{r-1}\circ f_r)=D_4$. 
Set $p_i:=\deg(f_i)$ so that $p_i$ is either prime or $4$ by Proposition \ref{prop:indec_poly}.


By Theorem \ref{thm:diagkern}, the kernel $K$ 
of the projection $\Mon_k(f)\to\Mon_k(f_1\circ \cdots \circ f_{r-1})$ 
is either diagonal, or  has socle $C_2\times C_2$. In the following, we claim that this implies $r\geq 3$ and that $f_{r-2}\circ f_{r-1}\circ f_r$ is either a composition of three quadratic maps,
or $\deg(f_r)=2$ and $\deg(f_{r-2}\circ f_{r-1})=8$.
We argue for $f$ but the same argument applies to $g$. 

When $r=2$, since $K$ is diagonal or has socle $C_2 \times C_2$, since $f_2$ is  right-unique, and since  $f=f_1\circ f_2$ is strongly unique or $\Mon_k(f)=D_4$,
Theorem \ref{thm:largemon} implies  that either ($p_1=p_2=2$ and $\Mon_k(f)=D_4$) or ($p_2=2$, $p_1=4$ and $\Mon_k(f)=\GL[2](3)$). The former case implies that $(\mu\circ f\circ \eta_1,\mu\circ g\circ \eta_2)= (D_{4,\alpha},-\frac{1}{4}D_{4,2\alpha})$ for  $\mu,\eta_1,\eta_2\in k[X]$ of degree $1$ and $\alpha\in k$; see Example \ref{ex:equtoT4-T4}. The latter case cannot occur as a direct   consideration\footnote{E.g., any faithful transitive degree-$8$ action of $\GL[2](3)$ must have stabilizer of order $6$ intersecting the center trivially; but all such subgroups project to conjugate subgroups $S_3\le S_4$, implying that a corresponding reducible pair $(f,g)$ must have a pair of reducible left factors of degree $4$.} or Magma computation show.

Henceforth assume $r\ge 3$. 
Due to Proposition \ref{lem:norittstepp}\ref{case:norittstepp_1},\ref{case:norittstepp_2} and Lemma \ref{lem:2step_solv}, we either have $p_r=4$ or $(p_r, p_{r-1})\in \{(2,2), (3,3), (2,4)\}$.

Assume first that $\Mon_{\overline{k}}(f_r)$ is noncyclic. In this case, we are left to consider the cases $p_r=4$, or ($\Mon_{\oline k}(f_r)=S_3$ and $p_{r-1}=3$). 
Let $\oline G = \Mon_k(f_{r-1}\circ f_r)$ and $\oline K=\ker(\oline G\to \Mon_k(f_{r-1}))$. 
Since $f_r$ is right-unique, it follows from Corollary \ref{cor:largemon_normal}\ref{case:largemon_normal_3} that 
$\oline{G}$ contains $\soc(\oline{K})=V_4^{p_{r-1}}$ in case $p_r=4$, (resp. $C_{p_r}^{p_{r-1}}$ otherwise), as a minimal normal subgroup.
In particular, the image of $\soc(K)$ in $\soc(\oline K)$ is therefore the full subgroup, and hence has rank $\geq 3$, 
contradicting that the assumption that $\soc(K)$ is diagonal or $C_2\times C_2$.

Henceforth we assume 
$\Mon_{\overline{k}}(f_r)\cong C_{p_r}$ is cyclic. Recall that in this case we are reduced to considering $p_{r-1}=p_r\in \{2,3\}$ or $(p_r, p_{r-1})=(2,4)$. 
If $f_{r-2}\circ f_{r-1}$ admits a Ritt move, then $p_{r-2}$ and $p_{r-1}$ are coprime by Corollary \ref{cor:ritt}, hence so are $p_{r-2}$ and $p_{r}$ and we may apply the Ritt move to  replace $f_{r-1}$ by a factor of degree coprime to $p_r$. 
This case has already been  
treated above using Lemma \ref{lem:2step_solv}. 
Hence we may assume $f_{r-2}\circ f_{r-1}$ is right-unique, and in particular that $f_{r-2}\circ f_{r-1}\circ f_r$ is of degree $8$ or $16$ if $\oline G=D_4$ {(by Lemma \ref{lem:3-factors})}. If 
$\oline G\neq D_4$, as noted above, $f_{r-1}\circ f_r$ is a strongly-unique decomposition. Thus, if $\oline G\not\cong \GL[2](3)$,  we may apply Lemma \ref{lem:3-factors} 
to deduce that $f_{r-2}\circ f_{r-1}\circ f_r$ is again of degree $8$ or $16$.

Consider therefore finally the case $\oline G=\GL[2](3)$. Denote by $Q_8\triangleleft \oline G$ the quaternion subgroup of $\oline G$, and let $M = \ker(\Mon_k(f)\to \Mon_k(f_1\circ \cdots\circ f_{r-2})) \cap Q_8^{p_{1}\cdots p_{r-2}}\triangleleft \Mon_k(f)$. Since $f_{r-2}\circ f_{r-1}$ is strongly-unique, it follows from Corollary \ref{cor:largemon_normal}\ref{case:largemon_normal_3} that $\Mon_k(f_{r-2}\circ f_{r-1})$  contains $V_4^{p_{r-2}}$ as a minimal normal subgroup, 
and hence the image of $M$ in $\Mon_k(f_{r-2}\circ f_{r-1})$ equals $V_4^{p_{r-2}}$. 
Restricting the central extension $\pi: \GL[2](3)\to S_4$ to $\pi^{-1}(V_4)=Q_8\to V_4$ one obtains a Frattini extension, forcing the image $\overline{M}$ of $M$ in $G_3:=\Mon_k(f_{r-2}\circ f_{r-1}\circ f_r)$ to contain $Q_8^{p_{r-2}}$.
Due to Corollary \ref{cor:3-useful}  (applied with $f_{r-2}\circ f_{r-1}$ in the role of $f_2$), $M/K$ is a subdirect power of the image of $M$ in $\Mon_k(f_{r-2}\circ f_{r-1})$, and hence a subdirect power of $V_4$. In particular, $M/K$, and a fortiori $\overline{M}/\overline{K}$, is elementary-abelian. But this enforces $\overline{K} = Z(Q_8)^{p_{r-2}}$, contradicting that $K$ is diagonal or $C_2\times C_2$ as soon as $p_{r-2}>2$. 
This once again leaves only the degree-$16$ option $p_r=p_{r-2}=2$, $p_{r-1}=4$.

As the same argument applies to $g$, the right  factors $f_{r-2}\circ f_{r-1}\circ f_r$ and $g_{s-2}\circ g_{s-1}\circ g_s$ of  $f$ and $g$, respectively, 
fulfill $\deg(f_r)=\deg(g_s)=2$ and $\deg(f_{r-2}\circ f_{r-1}), \deg(g_{s-2}\circ g_{s-1})\in \{4,8\}$.
The assertion therefore follows from the following  Proposition \ref{prop:special-func}. 
\end{proof}

\begin{prop}\label{prop:special-func}
Suppose that $f,g\in k[X]$  are of the form $f=f_0\circ u$, $g=g_0\circ v$ such that $u=u_1\circ u_2$ and $v=v_1\circ v_2$ with $\deg(u_2)=\deg(v_2)=2$ and each of $u_1,v_1$ of degree $4$ or $8$ with solvable monodromy.
Then $f,g$ are not minimally reducible. 
\end{prop}

\begin{lem}\label{lem:field-definition}
Suppose $f\in k[X]$ is of prime power degree, $\Mon_k(f)$ is solvable, 
and $x$ a root of $f(X)-t\in k(t)[X]$. Let $L/k(t)$ be a field in which  $\infty$ is unramified. Then the lattice of intermediate fields in $k(x)/k(t)$ is in one-to-one correspondence with that of $L(x)/L$ via the map $M\to ML$ given by compositum with $L$ in the Galois closure of $L(x)/k(t)$. 
\end{lem}
\begin{proof}
Suppose $\deg(f)=p^n$ for a prime $p$ and integer $n\in\mathbb N$. Note, since $\Mon_k(f)$ is solvable, $f$ is either  a composition of degree-$p$ polynomials, or $p=2$ and degree-$4$ indecomposable polynomials also occur, by Proposition \ref{prop:indec_poly}. 
The places of $L$ lying over $\infty$  are totally ramified in $L(x)$ by Remark \ref{rem:abh}. 
Since $[L(x):L]=p^n$ is a prime power,  
this implies the intermediate fields of $L(x)/L$ 
are totally ordered 
by Lemma \ref{lem:towerext}. 

If $f$ is a composition of  degree-$p$ polynomials, the extension $k(x)/k(t)$ contains a maximal chain of degree-$p$ extensions. Since these fields are linearly disjoint from $L/k(t)$, their compositum induces a maximal chain of degree-$p$ extensions in $L(x)/L$. Since the intermediate fields in $L(x)/L$ are totally ordered the assertion follows.

Otherwise, $k(x)/k(t)$ has an intermediate degree-$4$ subextension $k(x')/k(t')$ with no intermediate fields. Letting $\Omega_x', \Omega'$ be the Galois closures of $k(x')/k(t')$, $L(t')/k(t)$, respectively, we have $\Gal(\Omega'_x/k(t'))=S_4$ by Proposition \ref{prop:indec_poly}. Since $\infty$ is unramified in $\Omega'/k(t')$ by Remark \ref{rem:abh}, we see that $\Gal(\Omega'_x\Omega'/\Omega')=S_4$ as well by Lemma \ref{lem:transitive-normal}\ref{case:transitive-normal_2}. It follows  that the degree-$4$ extension $L(x')/L(t')$ has no nontrivial intermediate fields since such a  field would yield a nontrivial intermediate field of $\Omega'(x')/\Omega'$. 
Thus, the same argument as for degree-$p$ compositions yields  that the composition of $L$ with the maximal chain of intermediate fields  in $k(x)/k(t)$  gives the  maximal  chain of intermediate fields of $L(x)/L$.
\end{proof}

\begin{proof}[Proof of Proposition \ref{prop:special-func}]
Assume on the contrary $f,g$ is a minimally reducible pair, and let $x,y$ be roots of $f(X)-t,g(X)-t\in k(t)[X]$, respectively. 
Set $u=u(x)$, $v=v(y)$, and let $\Omega_{x}$, $\Omega_y$ and $\Omega'$ denote the Galois closures of $k(x)/k(u)$,  $k(y)/k(v)$, and $k(u,v)/k(t)$, respectively. 
By the minimal reducibility assumption $k(x,v)$ and $k(u,y)$ are not linearly disjoint over $k(u,v)$, and hence admit intermediate extensions $k(u,v)\lneq F_1\leq k(x,v)$ and $k(u,v)\lneq F_2\leq k(y,u)$ which are not linearly disjoint and admit the same Galois closure by Lemma \ref{lem:friedarg}. 

Assume first $n:=\deg(u)=\deg(v)$, so that $k(u,v)/k(u)$ is unramified over $\infty$ by Remark \ref{rem:abh}(1). 
Then the extension $F_1/k(u,v)$ is defined over $k(u)$ by Lemma \ref{lem:field-definition}, that is, there exists $k(u)\lneq k(x')\leq k(x)$ which is linearly disjoint from $k(u,v)/k(u)$ such that $k(x',v)=F_1$. 
If $F_1$ is a proper subfield of $k(x,v)$,  already the extensions $k(x'),k(y)$ are not linearly disjoint over $k(t)$ by the above claim, contradicting minimal reducibility.     Hence we may assume $F_1=k(x,v)$ and similarly $F_2=k(u,y)$, so that these extensions of $k(u,v)$ have the same Galois closure $\Omega_c$. Set $H:=\Gal(\Omega_c/k(u,v))$. 

Since $k(u,v)$ and $k(x)$ are linearly disjoint over $k(u)$ by minimal reducibility,  $\Omega_x(v)$ coincides with the Galois closure $\Omega_c$ of $k(x,v)/k(u,v)$ and hence $H$ can be viewed as  a transitive subgroup of $\Mon_k(u)=\Gal(\Omega_x/k(u))$ via restriction. 
Similarly, $H$ 
can be viewed as a transitive subgroup of $\Mon_k(v)$. 
Moreover, $H$ contains a cyclic transitive subgroup (in both actions) due to ramification over $\infty$, as noted above.  
Let $H_x,H_y\leq H$ be the point stabilizers in the two actions, so that $H_x$ acts intransitively on $H/H_y$. By running in Magma over transitive groups $H$ (containing a cyclic transitive group) of degree $8$ (resp., $16$) in two actions, on $H/H_x$ and $H/H_y$, we see that if $H_x$ is intransitive on $H/H_y$, then there are overgroups $H_x\leq H_x'\leq H$ and $H_y\leq H_y'\leq H$ such that at least one of them is a proper subgroup of $H$, and $H_x'$ is intransitive on $H/H_y'$, contradicting the minimal reducibility assumption once more. 

Now assume $u,v$ are of distinct degrees, and without loss of generality $\deg(u)=16$, $\deg(v)=8$. We may and will furthermore assume that $u$ has an indecomposable left factor $u_0$ of monodromy group $S_4$, since in case $u_1=u_0\circ u_1'$ with $\deg(u_0)=2$, we may simply reduce to the case $\deg(u)=\deg(v)$ via replacing $u$ by $u_1'\circ u_2$. 
Setting $H=\Gal(\Omega_x(v)/k(u,v))$ and $N:=\Gal(\Omega_x\Omega'/\Omega')$, 
we identify $H$ as a  subgroup of $\Mon_k(u)$  and $N$ as a subgroup of $H$, via restriction. 
Note that $H$ is a transitive since $k(u,v)/k(u)$ and $k(x)/k(u)$ are linearly disjoint by minimal reducibility, Also note that $N\lhd \Mon_k(u)$, and that $N$ contains the normal subgroup of $\Mon_k(u)$ generated by the inertia groups $I\leq N$ over $\infty$. One has $\#I=8$  since $I$  contains the square of a full cycle by Remark \ref{rem:abh}. 

We first claim that there exists a unique proper maximal subfield $k(u,v)\lneq F_1\lneq k(x,v)$, and it is defined over $k(u)$. We give a direct argument to show this, and note that the underlying group theoretic statement is also verified directly with Magma. Let $U< \Mon_k(u)$ denote a point stabilizer. Note that, since $\Mon_k(u_0)=S_4$ and (the order-$8$ cyclic group) $I$ maps to the subgroup of $S_4$ generated by a double transposition, the $2$-Sylow group $N_2$ of $N$ is transitive and projects to $V_4\triangleleft S_4$ under the map $\Mon_k(u)\to \Mon_k(u_0)$. Moreover, since all elements of $V_4\setminus\{id\}$ are conjugate in $S_4$, the group $N_2$ has a quotient $C_2\times C_2$ all of whose nonidentity elements lift to an order-$8$ element in $N$. In particular, any maximal subgroup $M$ of the $2$-group $N_2$ must contain at least one conjugate of a generator of $I$, i.e., an element of order $8$. 
Letting $U\cap N_2 \subsetneq M\subsetneq N_2$ be the stabilizer of a maximal block $\Delta$ of $N_2$ in its action as a transitive subgroup of $\Mon_k(u)$, this implies that all stabilizers of blocks of $N_2$ which are contained in $\Delta$ must be totally ordered by inclusion, by Lemma \ref{lem:towerext} (as places over $\infty$ are totally ramified in the corresponding extension). 
Since intersections of blocks are again blocks, this implies that $N_2$ stabilizes a unique block system of each block length $2$ and $4$. In particular, since $N_2\leq H\leq \Mon_k(u)$, every block of these groups is a block of $N_2$, and hence $H$ and $\Mon_k(u)$  have unique minimal block systems. Since $u$ has a quadratic right factor $u_2$, the size of such a minimal block of $\Mon_k(u)$ and $H$ must equal $\deg(u_2)=2$. In particular, there exists a unique minimal intermediate group $U\cap H\lneq U'\lneq H$ in $H$, and $U'$ is the intersection with $H$ of the minimal block stabilizer $UU'$ in $\Mon_k(u)$ fixing a root of $u_2$. Thus, letting $F_1$ be the fixed field of $U'$, the claim follows from the usual Galois correspondence. 
%
%

It follows that if a proper subfield of $k(x,v)$ is not linearly disjoint from $k(u,y)$ over $k(u,v)$, then so is $F_1$. However, as $F_1/k(u,v)$ is defined  over $k(u)$, this contradicts the minimal reducibility assumption. 
Thus, $k(x,v)/k(u,v)$ has the same Galois closure $\Omega_c$ as  $F_2/k(u,v)$ for some subfield $F_2\leq k(u,y)$ by  Lemma \ref{lem:friedarg}. 
Letting $V=\Gal(\Omega_c/F_2)$, we obtain a  faithful transitive action of $H$ and hence  of $N_2\subseteq H$ on $H/V$. It is an action of degree dividing $8$. But now let $\tilde{M}\subset N_2$ be the stabilizer of a maximal block in this action. This contradicts that  $\tilde{M}$ contains an element of order $8$ as shown above. Thus there exists no  minimal reducible pair $f,g$ as desired.  
\end{proof}

\begin{rem}
\label{rem:critical_cases}
While not strictly required for the proof of any of the theorems, it may be of interest for related applications to give a precise list of possible monodromy groups $G_3:=\Mon_{\overline{k}}(f_{r-2}\circ f_{r-1}\circ f_r)$ such that $\ker(\Mon_k(h\circ f_{r-2}\circ f_{r-1}\circ f_r)\to \Mon_k(h\circ f_{r-2}\circ f_{r-1}))$ can be diagonal or isomorphic to $C_2\times C_2$ for some $h\in k[X]\setminus k$. 

Lemma \ref{lem:3-factors} and the proof of Theorem \ref{thm:poly-solvable} show that all such examples fulfill $$(\deg(f_r),\deg(f_{r-1}), \deg(f_{r-2}))\in \{(2,2,2), (2,2,4), (2,4,2)\}.$$
We then  use Magma (in particular the database of transitive groups of degrees $8$ and $16$) 
to run over quotients $\oline N/\oline K$ of  normal subgroups $\oline N\lhd G_3$ containing a full cycle by  subgroups $\oline K\leq \oline N$ isomorphic to $C_2$ or  $C_2\times C_2$, and keep track of cases in which the nilpotency class 
of the $q$-Sylow subgroup of $\oline N/\oline K$ is at most as large as that of the $q$-Sylow of $\oline N/(K_3\cap \oline N)$, where $K_3=\ker(G_3\to\Mon_k(f_{r-2}\circ f_{r-1}))$. This gives a reasonably short list of possible groups $G_3$, of which some can furthermore be excluded on the ground of not possessing a polynomial genus-$0$ system. This leaves only the monodromy groups $G_3$ of Id
$8T\theta$ for $\theta\in \{6, 27, 28, 35\}$, and 
$16T\theta$ for $\theta\in \{773, 1499, 1651, 1832, 1866, 1871, 1872, 1886\}$. 
Thus, except for these groups, we obtain a contradiction to the fact that $\oline N/\oline K$ is a quotient of a subdirect power of 
$\oline N/(K_3\cap \oline N)$.
\end{rem}
In particular, there are no such examples of odd degree.

We similarly obtain the following intermediate result on not necessarily polynomial reducible pairs with solvable monodromy group.

\begin{thm}\label{thm:solv} 
Let $f\in k[X]\setminus k$ be a polynomial of odd degree with solvable monodromy group, 
and let  $g \in k(Y) \setminus k$ be a Siegel function of degree at least $2$. Then $f(X)=g(Y)$ is reducible if and only if $f$ and $g$ have a nontrivial common left composition  factor, that is, $f = h\circ f_1$ and $g = h\circ g_1$ for some $h,f_1 \in k[X]$ and $g_1\in k(X)$ such that $\deg(h)>1$.
\end{thm}

\begin{proof}
By possibly replacing  $f,g$ by left factors of theirs, we may assume it is a minimally reducible pair. 
Now write $f=f_1\circ \cdots \circ f_r$ for indecomposable polynomials $f_i\in k[X]$. Since $f$ is assumed to be solvable of odd degree,  $p_i=\deg(f_i)$ is prime for all $i=1,\dots, r$ by Proposition \ref{prop:indec_poly}. 
Then, Theorem \ref{thm:diagkern} implies that {the socle of} $K=\ker(\Mon_k(f) \rightarrow \Mon_k(f_1 \circ \dots \circ f_{r-1}))$ is either $C_{p_r}$ (with $p_r=\deg(f_r)$) or $C_2\times C_2$.

Assume that $r\ge 2$. 
Pick roots $x$ and $y$ of $f(X)-t=0$ and $g(Y)-t=0$, respectively. By Lemma \ref{lem:friedarg}, the extensions  $k(x)/k(t)$ and $k(y)/k(t)$ have a common Galois closure $\Omega$. Since $\deg(f_{r-1}\circ f_r)$ is not divisible by $4$, Proposition \ref{lem:norittstepp}\ref{case:norittstepp_1} implies  $f_{r-1} \circ f_r$ is not linearly related to {$X^{p_r p_{r-1}}$} or {$T_{p_r p_{r-1}}$} over $\oline k$, and thus is strongly-unique by Ritt's theorems. Lemma \ref{lem:2step_solv} thus already forces $p_r=p_{r-1}=3$, and this possibility is furthermore ruled out in the proof of Theorem \ref{thm:poly-solvable}, as in Remark \ref{rem:critical_cases}.

Thus we get $r=1$. In such case $\Mon_k(f)\leq \AGL_1(p_1)$ and hence  $k(y)$  contains a root of $f(X)-t$ by  Lemma \ref{lem:onAGL1}, so that $g$ factors through $f$ as needed. 
\end{proof}

\section{Proofs of main results}\label{sec:nonsolv}
\label{sec:nonsolv-DLS}
\subsection{DLS over general fields}\label{sec:DLS} Let $k$ be a field of characteristic $0$. Theorem \ref{thm:DLS-C} is the special case $k=\mC$ of the following theorem. Recall that $D_{n,\alpha}$ denote the $n$-th Dickson polynomial with parameter $\alpha\in k$. 
\begin{thm}\label{thm:DLS} 
Let $f,g\in k[X]$ be polynomials of degree $>1$. 
Then $f(X)-g(Y)$ is reducible in $k[X,Y]$  
if and only if one of the following occurs for some polynomials $f_1,g_1\in k[X]$:
\begin{enumerate}[leftmargin=*,label={(\arabic*)},ref=(\arabic*)]
\item $f$ and $g$ have a common left composition  factor  $h\in k[X]$ of degree at least $2$, that is, $f = h\circ f_1$ and $g = h\circ g_1$; 
\label{DLS_1}
\item $f=(\mu\circ h_1\circ \lambda)\circ f_1$ and $g=(\mu\circ h_2\circ \lambda)\circ g_1$, for some 
linear $\mu,\lambda\in\overline{k}[X]$, where $(h_1,h_2)$ is 
one of the pairs of polynomials of degrees 7,11,13,15,21,31  
given in \cite[\S 5]{CC}.
\label{DLS_2}
\item $f=(\mu\circ D_{4,\alpha})\circ f_1$ and $g=(\mu\circ (-\frac{1}{4}D_{4,2\alpha}))\circ g_1$, for some 
linear $\mu\in k[X]$ and $\alpha\in k$. 
\label{DLS_3}
\end{enumerate}
\end{thm}
Note that when $k=\overline{k}$ is algebraically closed, the case $\alpha=0$ in (3) falls into Case (1), whereas all cases with $\alpha\ne 0$ can be expressed as $\{f,g\} = \{(\mu\circ T_4)\circ f_1, \mu\circ (-T_4)\circ g_1\}$.
\begin{rem}
For the converse direction of Theorem \ref{thm:DLS}, it is straightforward to see that (1) implies the reducibility of $f(X)-g(Y)\in k[X,Y]$. For (3), reducibility follows from  \eqref{eq:facD4al}. For (2), since ${\rm Mon}_k(\mu \circ h_i \circ \lambda)={\rm Mon}_{\overline{k}}(\mu \circ h_i \circ \lambda)$, reducibility over $k$ is equivalent to that over $\oline k$ since these are equivalent to the intransitivity in one action of the stabilizer in the other action. Thus, in this case the reducibility over $k$ follows from the reducibility of $h_1(X)-h_2(Y)\in \oline k[X,Y]$, whose factorizations are given in  \cite[Section 5]{CC}.  
\end{rem}

For maps with nonsolvable monodromy, we take a simpler approach based on \cite{KN}. The  key argument is given by the following lemma. We shall  apply this argument to compare the Galois closures of maximal nonsolvable left factors of maps. 
\begin{lem}
\label{lem:galcl_containment}
Let $f,g\in k(X)$ be a minimally reducible pair with nonsolvable $\Mon_k(f)$. Write $f=f_1\circ\dots\circ f_r$, and $g=g_1\circ\dots\circ g_s$ 
for indecomposable  $f_i,g_j\in k(X)$. Set $\Gamma_i=\Mon_k(f_i)$ (resp., $\Gamma'_j = \Mon_k(g_j))$, $\hat f_i:=f_1\circ\cdots\circ f_i$. 
(resp., $\hat g_j=g_1\circ\cdots \circ g_j$), and let $\Omega_{i}$ (resp., $\Omega'_{j}$) be the splitting field of  $\hat  f_i(X)-t$ (resp., $\hat g_j(X)-t$), with $\Omega_0:=\Omega'_0:=k(t)$. Now choose the indices $i\in \{1,\dots, r\}$ and $j\in \{1,\dots, s\}$ maximal such that $\Gal(\Omega_i/\Omega_{i-1})$ and $\Gal(\Omega'_j/\Omega'_{j-1})$ are nonsolvable.
Assume that 
$\Gamma_i$ has a unique nonabelian minimal normal subgroup $\soc(\Gamma_i)$ and that $f_i$ is a right-unique factor of $\hat f_i$. 
Then the following hold:
\begin{enumerate}[leftmargin=*,label={(\arabic*)},ref=(\arabic*)]
\item $\Omega'_{j}\supseteq \Omega_{i}$ and $\soc(\Gal(\Omega_i/\Omega_{i-1})) \cong \soc(\Gamma_i)^m$ for some $m \geq 1$. 
\label{case:galcl_containment_1}
%
\item 
\label{case:galcl_containment_2}
{If moreover $f\in k[X]$ is a polynomial, then} 
$\soc(\Gal(\Omega_i/\Omega_{i-1}))$ maps isomorphically to a normal subgroup of a quotient of $\Mon_k(g_j)$. 
\end{enumerate}
\end{lem}
\begin{rem}\label{rem:addendum}
1) By Lemma \ref{lem:friedarg}, the minimal reducibility of $f,g$ implies $f(X)-t$ and $g(X)-t$ have the same splitting field $\Omega$ over $k(t)$. Beyond this implication, the assumption of reducibility is unnecessary for  Assertion \ref{case:galcl_containment_1} of Lemma \ref{lem:galcl_containment}. \\
2) 
We  shall in fact see  that $j$ is the minimal index for which $\Omega_j'\supseteq \Omega_i$. \\
3) The assumption that $f$ is a polynomial is only  used  to guarantee that $\soc(\Gamma_i)$ is simple and primitive, by Proposition \ref{prop:indec_poly}. 
\end{rem}  

\begin{proof}[Proof of Lemma \ref{lem:galcl_containment}]
Let $\Omega$ be the splitting field of $f(X)-t$ and $g(X)-t$ over $k(t)$, as in Remark \ref{rem:addendum}. Let $\hat K:=\Gal(\Omega_i/\Omega_{i-1}) = \ker(\Mon_k(\hat f_i)\to\Mon_k(\hat f_{i-1}))$, $\hat{K}_g:=\Gal(\Omega'_j/\Omega'_{j-1})$, and $K:=\soc(\hat K)$. \\ 
(1) Since $\hat K \neq 1$ and $f_i$ is a right-unique factor of $\hat f_i$, it follows from \cite[Proposition 3.3]{KNR24} that $K$ is the unique minimal normal subgroup of $\Mon_k(\widehat{f}_i)$, and that $K\cong \soc(\Gamma_i)^m$ for some $m \geq 1$. Moreover, it is contained as a  subdirect power in $\soc(\Gamma_i)^{d_{i}}$, where $d_{i}=\deg(\hat f_{i-1})$.

Assume on the contrary that $\Omega'_{j}$ does not contain $\Omega_{i}$. Since $K$ is the unique minimal normal subgroup of $\Mon_k(\hat{f}_i)$, it follows that $\Omega'_{j} \cap \Omega_{i} \subseteq \Omega_{i}^{K}$. As $\soc(\Gamma_i)$, and hence $K$, is nonsolvable, $\Gal(\Omega_{i}/\Omega'_{j} \cap \Omega_{i} )$ is also nonsolvable. On the other hand, this group is isomorphic to the quotient $\Gal(\Omega_i\Omega_j'/\Omega_j')$ of $\Gal(\Omega/\Omega'_{j})$, which is solvable, a contradiction. Therefore $ \Omega'_{j}\supseteq \Omega_{i}$.
We moreover note that $\Omega'_{j-1}$ does not contain $\Omega_{i}$ as in Remark \ref{rem:addendum}. Indeed, since $\hat{K}_g$ is nonsolvable, ${\rm Gal}(\Omega/\Omega'_{j-1})$ is nonsolvable as well. As $\textrm{Gal}(\Omega/\Omega_{i})$ is solvable, it follows that $\Omega'_{j-1}$ cannot contain $\Omega_{i}$. \\
(2) Identifying $U_i:=\textrm{Gal}(\Omega_{i}(y_{j-1})/k(y_{j-1}))$ with a subgroup of $\Mon_k(\hat f_i)$ via restriction below, we claim that $K$ is contained in $U_i$ as its unique minimal normal subgroup. 
For that, let $x$ be a root of $f(X)-t$, let $x_i$ (resp., $y_j$) be a root of $\widehat{f}_{i}(X)-t$ (resp., $\widehat{g}_j(X)-t$), and set $x_{i-1}=f_i(x_i)$ (resp., $y_{j-1}=g_j(y_j)$). Moreover, define 
$$
U=\textrm{Gal}(\Omega/k(y_{j-1}))\text{ and }U_{i-1}:=\textrm{Gal}(\Omega_{i-1}(y_{j-1})/k(y_{j-1}))
.$$ By minimal reducibility of $f$ and $g$, the extensions $k(y_{j-1})/k(t)$ and $k(x)/k(t)$ are linearly disjoint. The same holds a fortiori for $k(y_{j-1})/k(t)$ with $k(x_i)/k(t)$ and with $k(x_{i-1})/k(t)$. Consequently, $U$, $U_i$ and $U_{i-1}$ are the Galois groups of the Galois closures of  $k(y_{j-1},x)/k(y_{j-1})$, $k(y_{j-1},x_i)/k(y_{j-1})$, and $k(y_{j-1},x_{i-1})/k(y_{j-1})$, respectively, by Remark \ref{rem:Gal-closure}. Under restriction, $U_i$ identifies with $ \Gal(\Omega_{i}/k(y_{j-1}) \cap \Omega_{i}) \leq \Mon_k(\widehat{f}_i)=\textrm{Gal}(\Omega_{i}/k(t))$. 
Note that via this identifcation $K\leq U_i$ since $k(y_{j-1})\cap \Omega_{i} \subseteq \Omega'_{j-1}\cap \Omega_{i} \subseteq \Omega_{i}^K$, where  the last inclusion follows from the minimal normality of $K$ and since $\Omega_i \not\subseteq\Omega'_{j-1}$. 
Moreover, $U_{i-1}$ restricts to a subgroup of $\Mon_k(\widehat{f}_{i-1}) = \textrm{Gal}(\Omega_{i-1}/k(t))$, and since $\Omega_{i-1}\subseteq \Omega_{i}^K$, we deduce that $K$ in contained in $K_U:= \ker(U_i\to U_{i-1})$. 

We  now use \cite{KN} to show that $K=\soc(K_U)=U_i\cap \soc(\Gamma_i)^{d_{i}}$ 
is a minimal normal subgroup of $U_i$. 
Since $K\subseteq K_U$ and the image of $K$ in each of the $\Gamma_i$-components is $\soc(\Gamma_i)$, 
the Galois group $\Gamma_U$ of the Galois closure of $k(x_i,y_{j-1})/k(x_{i-1},y_{j-1})$ (which is also the image in $\Gamma_i$ of a block stabilizer)  contains the image of $K_U$ in $\Gamma_i$, and hence contains  $\soc(\Gamma_i)$. Since $\soc(\Gamma_i)$ is primitive and simple, $\Gamma_U$ is primitive almost-simple. Thus, viewing $U_i$ as a subgroup of $\Gamma_U\wr U_{i-1}$,   Lemma 3.1 and Corollary 3.4 of \cite{KN} imply that $\soc(K_U)=U_i\cap \soc(\Gamma_i)^{d_{i}}$ is a minimal normal subgroup of $U_i$. Since $K\leq U_i\cap \soc(\Gamma_i)^{d_{i}}$ is normal in $U_i$, we  deduce that $K=\soc(K_U)$ is a minimal normal subgroup of $U_i$. Since $K$ is the unique  minimal normal subgroup of $\Mon_k(\hat f_i)$, its centralizer is trivial. Thus,  there is no other minimal normal subgroup of $U_i$, proving the claim.

To  deduce (2), consider the image of $\pi:U\to \Gamma'_j$, that is, the image of the action of $U$ on the $k(y_{j-1})$-conjugates of $y_j$. 
Let $\tilde K:=\Gal(\Omega/\Omega_i^K(y_{j-1}))\leq U$ be the preimage of $K\leq U_i$ and $N:=\Gal(\Omega/\Omega_i(y_{j-1}))$, so that $N\lhd U$ is solvable and $\tilde K/N\cong K$. 
Since $N$ is contained in the kernel of  the map $U\to \Gamma_j'/\pi(N)$ and $U_i\cong U/N$, 
we obtain an  induced map $\tilde \pi:U_i\to \Gamma_j'/\pi(N)$. 
We claim that $\tilde\pi(K)\neq 1$.  
Assuming the claim,  since  $K$ is the unique minimal normal subgroup of $U_i$, we deduce that $K$ maps isomorphically onto $\tilde\pi(K)\cong \pi(\tilde K)/\pi(N)$. Thus,  $\pi(\tilde K)/\pi(N)$ is the desired normal subgroup of $\Gamma_j'/\pi(N)$, yielding  (2).

To show the remaining claim $\tilde\pi(K)\neq 1$, assume on the contrary $\pi(\tilde K)=\pi(N)$. Since $N$ is solvable, so is $\pi(N)=\pi(\tilde K)$. In particular, the image of the action of $\tilde K$ on conjugates of $y_{j}$ is solvable and hence $\Omega_i\Omega_{j-1}'/\Omega_i^K\Omega_{j-1}'$ has a solvable Galois group. 
On the other hand since $\Omega_{j-1}'\cap \Omega_i\subseteq \Omega_i^K$, the Galois group of $\Omega_i\Omega_{j-1}'/\Omega_i^K\Omega_{j-1}'$ identifies with $K$ via restriction and hence is nonsolvable, a contradiction. 
\end{proof}
The assumptions of Lemma \ref{lem:galcl_containment} can be guaranteed as follows. 
\begin{rem}\label{rem:fury}
1) If $f$ is a polynomial, then a nonsolvable monodromy group $\Gamma_i=\Mon_k(f_i)$ has to be almost-simple by Proposition \ref{prop:indec_poly}. 
Moreover the right-uniqueness assumption on $f_i$ may be assumed to hold up to performing Ritt moves.\\
2) Similarly, if $f$ is a Siegel function with nonsolvable monodromy group such that $f$ does not factor through an indecomposable with {\it affine} nonsolvable monodromy, then $\Gamma_i=\Mon_k(f_i)$ has a unique nonabelian minimal normal subgroup by Proposition \ref{prop:siegel_indec}. The right-uniqueness assumption may be assumed to hold via performing Ritt moves, using the generalization \cite[Theorem 1.1]{Pak5} of Ritt's theorem to Siegel functions. 
\end{rem}

Lemma \ref{lem:galcl_containment} has the following surprising consequence.

\begin{cor}
\label{cor:only_one_nonsolv}
Assume that $f,g\in k[X]$ form a minimally reducible pair. Then there exists a decomposition $f=f_1\circ h$ of $f$, where $f_1\in k[X]$ is indecomposable and $\Mon_k(h)$ is solvable.
\end{cor}
\begin{proof}
We may assume that $\Mon_k(f)$ is nonsolvable. Write $f=f_1\circ\dots\circ f_r\in k[X]$ and $g=g_1\circ \dots\circ g_s$ with all $f_i$ and $g_j\in k[X]$ indecomposable. Let $i\in \{1,\dots, r\}$ (resp., $j\in \{1,\dots, s\}$) be maximal such that $\Mon_k(f_i)$ (resp., $\Mon_k(g_j)$) is nonsolvable. Set $\widehat{f}_i=f_1\circ\dots\circ f_i$ and $\widehat{g}_j = g_1\circ\dots\circ g_j$. After performing Ritt moves if necessary, we may assume that $f_i$ is a right-unique factor of $\widehat{f}_i$ as in Remark \ref{rem:fury}. 
It then follows from \cite[Corollary 4.4]{KNR24} that $\ker(\Mon_k(\hat f_i)\to\Mon_k(\hat f_{i-1}))$ has socle $K\cong \Gamma_i^{\deg(\hat f_{i-1})}$, where $\Gamma_i:=\soc(\Mon_k(f_i))$ is a nonabelian simple group, cf.\ Proposition \ref{prop:indec_poly}.

On the other hand, by Lemma \ref{lem:galcl_containment}, $K$ maps isomorphically to a normal subgroup of a quotient of $\Mon_k(g_j)$. Since $\Mon_k(g_j)$ is nonsolvable and hence almost-simple by Proposition \ref{prop:indec_poly}, it follows  that $\deg(\hat f_{i-1})=1$, and hence that $i=1$.
\end{proof}
Finally, we combine Corollary \ref{cor:only_one_nonsolv} with the classification of indecomposable pairs $f_1,g_1$ (and hence the CFSG) to deduce Theorem \ref{thm:DLS}. 
\begin{proof}[Proof of Theorem \ref{thm:DLS}]  
By possibly replacing $f$ and $g$ by left factors of theirs, we may assume $f,g$ form a minimally reducible pair and hence admit the same Galois closure by Lemma \ref{lem:friedarg} and hence $\Mon_k(f)\cong \Mon_k(g)$ as abstract groups. 
If $\Mon_k(f)$ is solvable the conclusion follows from  Theorem \ref{thm:poly-solvable}, henceforth assume it   is nonsolvable. 
By Corollary \ref{cor:only_one_nonsolv},
we then have decompositions $f=f_1\circ h_1$, $g=g_1\circ h_2$ for {solvable $\Mon_k(h_1),\Mon_k(h_2)$, and} indecomposable $f_1,g_1$ with  $\Mon_k(f_1), \Mon_k(g_1)$ nonsolvable, and thus almost-simple by Proposition \ref{prop:indec_poly}.

By a symmetric application of Lemma \ref{lem:galcl_containment}, the splitting field of $f_1(X)-t$ equals that of $g_1(X)-t$. Thus $\Mon_k(f_1)\cong \Mon_k(g_1)$ is an almost-simple group admitting a cyclic transitive subgroup in its action on the roots of $f_1(X)-t$, as well as in its action on the roots of $g_1(X)-t$.  Thus,  \cite[Theorem]{Mul2} shows that the actions on roots of $f_1(X)-t$ and of $g_1(X)-t$, respectively, are permutation-equivalent, but are not equivalent (since $f_1,g_1$ have no common left composition factor by minimal reducibility). In particular, every element fixing a point in one of the two actions also fixes a point in the other action. Thus, by Burnside's lemma, the stabilizer of a root of $g_1(X)-t$ acts intransitively on the roots of $f_1(X)-t$. Thus, $f_1(X)-g_1(Y)\in k[X,Y]$ is  reducible, and by the minimal reducibility assumption,  $h_1,h_2$ are of degree $1$ and  $f,g$ are indecomposable. 

Moreover for pairs $f_1,g_1$ as above, \cite{Mul2}  gives  the  full list of possible degrees $\deg(f_1)$ and monodromy groups $\Mon_k(f_1)$, namely,  
$\deg(f_1)=7,11,13,15,21,31$ and $\Mon_k(f_1) = \PSL_3(2), \PSL_2(11), \PSL_3(3), \PSL_4(2), {\rm P\Gamma L}_4(3)$, and $\PSL_5(2)$, respectively. 
The corresponding polynomials $f_1,g_1$ were shown to be linearly related over $\oline k$ to those in \cite[\S 5]{CC}.
\end{proof}
\begin{rem}\label{rem:CFSG} The CFSG was not used in the proof of Corollary \ref{cor:only_one_nonsolv} but only in the last step of Theorem \ref{thm:DLS}, where the monodromy of the indecomposable factors $f_1,g_1$ was needed. 
\end{rem}
\subsection{Proof of Theorem \ref{thm:Hilbert}}\label{sec:siegel}
By \cite[Corollary\ 2.5]{KN}, $\textrm{Red}_f(\mathbb{Z})$ is the union of {a finite set and} value sets $g(\mathbb{Q})\cap \mathbb{Z}$ of Siegel functions $g$ over $\mathbb{Q}$ such that the fiber product of $f:\mP^1_\mQ\to\mP^1_\mQ$ and $g:\mP^1_\mQ\to\mP^1_\mQ$ is reducible.  
Equivalently, $\textrm{Red}_f(\mathbb{Z})$ is the union of a finite set and value sets $\tilde{g}(\mathbb{Q})\cap \mathbb{Z}$ of Siegel functions $\tilde{g}$ over $\mathbb{Q}$ such that $\tilde{f}$ and $\tilde{g}$ form a  {\it minimally} reducible pair, for some left factor $\tilde{f}$ of $f$. 
The case of solvable $\Mon_{\mathbb{Q}}(\tilde{f})$ is immediate from Theorem \ref{thm:solv}. To deal with the nonsolvable case, we begin with the following generalization of Corollary \ref{cor:only_one_nonsolv}. It requires the classification of indecomposable Siegel functions from \cite{Mul3} and hence the CFSG.

\subsubsection*{Setup} Throughout the proof, we use the notation of Section \ref{sec:DLS}.  We write $f=f_1\circ\dots\circ f_r$ and $g=g_1\circ \dots \circ g_s$ for indecomposable $f_1,\ldots,f_r\in k[X]$ and $g_1,\dots, g_s\in k(X)$ over a number field $k$. Write $\hat g_j=g_1\circ \dots\circ g_j$ for $j=1,\ldots,s$.  Let $j\in \{1,\dots, s\}$ be maximal such that $\ker(\Mon_{k}(\widehat g_j)\to \Mon_{k}(\widehat g_{j-1}))$ is nonsolvable. Without loss of generality, we may assume that $g_j$ is a right-unique factor of $\hat{g}_j$ via performing Ritt moves (\cite[Theorem 1.1]{Pak5}). In the part where the argument is specific to $\mQ$, we shall assume $k=\mQ$.

We next show $\Mon_{k}(f)$ contains only the single nonabelian composition factor $\soc(\Mon_{k}(f_1))$. In particular there is only a single  index $j$ as above.
\begin{lem}
\label{lem:only_one_nonsolv2}
Assume that a polynomial $f \in k[X]$ and a Siegel function $g\in k(X)$ form a minimally reducible pair. Then there exists a decomposition $f=f_1\circ h$ with  indecomposable $f_1\in k[X]$ and some $h\in k[X]$ of solvable monodromy. 
\end{lem}
\begin{proof}
We use the notation of the above Setup. The proof of Corollary \ref{cor:only_one_nonsolv} applies if   $g_j$ has an almost-simple monodromy group, or more generally when $\Mon_k(g_j)$ has only one nonabelian composition factor counting with multiplicity. But by \cite[Theorem 4.8]{Mul3}, the only other case is $\Mon_k(g_j) = S_n\wr C_2$, in which case it suffices to rule out the possibility $\deg(\hat f_{i-1})=2$, i.e., $f=f_1\circ f_2\circ h$ where $\deg(f_1)=2$ and $\Mon_k(f_2)$ is almost-simple. Further assume  $g_j$ is a right-unique factor of $\hat g_j$ as in the setup above. 
Thus we may apply Lemma \ref{lem:galcl_containment}\ref{case:galcl_containment_1} symmetrically in $f,g$ to conclude that $\hat g_j(X)-t$ and $(f_1\circ f_2)(X)-t$ have the same splitting field. Therefore, $\Mon_k(\widehat{g_j})=\Mon_k(f_1\circ f_2)\le  S_n\wr C_2$ for $n:=\deg(f_2)$. Since $\Mon_k(g_j)$  is a subquotient of $\Mon(\hat g_j)$ and it is already isomorphic to $\Mon(\hat g_j)\cong S_n\wr C_2$, it follows that $\hat g_j=g_j$ and  $j=1$. We now compare inertia groups at $t\mapsto \infty$ for the maps $g_1$ and $f_1\circ f_2$. Since their Galois closures coincide, these inertia groups must have the same order; the latter one is obviously of order $2n$, but the former one is generated by an element with exactly two orbits in the primitive product-type group $S_n\wr C_2 \le S_{n^2}$. Since $n\ge 5$, we have $n^2 > 4n$, so that at least one of these two orbits is of length $>2n$,  contradicting that the order is $2n$.
\end{proof}

To conclude the proof of Theorem \ref{thm:Hilbert}, it therefore remains to classify minimally reducible pairs $f\in \mathbb{Q}[X]$, $g\in \mathbb{Q}(X)$  a Siegel function over $\mathbb{Q}$, where $f=f_1\circ h$ with $\Mon_{\mathbb{Q}}(f_1)$ almost-simple and $\Mon_{\mathbb{Q}}(h)$ solvable.  We divide the remainder of the proof according to the shape of $\Mon_{k}(g_j)$. Firstly if $|\hat g_j^{-1}(\infty)|=1$, {since $f_1(X)-t$ and $\hat{g}_j(X)-t$ have the same splitting field (by Lemma \ref{lem:galcl_containment}\ref{case:galcl_containment_1}),} the proof of Theorem \ref{thm:DLS} applies  and yields (over arbitrary number fields $k$) that $f,g$ must be indecomposable. On the other hand, $|\hat g_{j-1}^{-1}(\infty)|=1$ holds by Lemma \ref{lem:siegel_by_cyclic} and the nonsolvability of $\Mon_{k}(g_j)$. In total, we shall henceforth  assume: 
\begin{equation}\label{ass:g_j} 
\Mon_{k}(f_1)\text{ is almost-simple};  \Mon_{k}(h)\text{ is solvable};  |\hat g^{-1}_{j-1}(\infty)|=1;\text{ and }|\hat g_j^{-1}(\infty)|=2.
\end{equation}
We distinguish the cases of nonaffine and affine $\Mon_{\mathbb{Q}}(g_j)$. In the former case, we have a result over arbitrary number fields $k$:
\begin{lem}
\label{lem:poly_nonaffinesiegel}
Let $f = f_1\circ \dots\circ f_r\in k[X]$ and $g=g_1\circ\dots\circ g_s\in k(X)$ be a minimally reducible pair, where $g$ is a Siegel function with $|g^{-1}(\infty)|=2$.
If $\Mon_k(f)$ is nonsolvable and $g$ does not factor through an indecomposable with affine nonsolvable monodromy group, then one of the following holds:
\begin{enumerate}[leftmargin=*,label={(\arabic*)},ref=(\arabic*)]
\item $f$ is indecomposable, $\Mon_k(f)\in \{A_5,S_5\}$, and $\deg(g)=10$. More precisely, $f$ is one of the polynomials in \cite[(1.8)]{DF};
\label{case:poly_nonaffinesiegel_1}
\item $f$ is of composition length $2$, and more precisely $\Mon_k(f)$ is one of the groups $C_2\times \PSL_3(2), C_3\times A_5, C_3\rtimes S_5, C_2.P\Gamma L_2(9), C_2\times M_{11}$ of  Magma Id 14T17, 15T15, 15T21, 20T265,22T26, resp. 
\label{poly_nonaffinesiegel_2}
\end{enumerate}
Moreover, in 2), $g$ is not a Siegel function over $\mathbb{Q}$.
\end{lem}
\begin{proof}
We use the Setup from the beginning of the section,  and following the above discussion assume \eqref{ass:g_j} holds. Since $\Mon_k(f)$ has only a single nonabelian composition factor, 
and since $\Mon_k(g_j)$ is assumed to be nonaffine, it follows that $\Mon_k(g_j)$ is almost-simple by Proposition \ref{prop:siegel_indec}. Moreover, $g_{j+1}\circ \dots\circ g_s$ has cyclic monodromy group over $\overline{k}$ by Lemma \ref{lem:siegel_by_cyclic}. Thus,  Proposition \ref{lem:norittstepp}\ref{case:norittstepp_1} implies that $g_{j+1}\circ \dots\circ g_s$ is either i) trivial (and thus $s=j$), or ii) of prime degree (and thus $s=j+1$), or iii) of degree $4$ (and thus $s=j+2$). 
Furthermore, since $f=f_1\circ h$ and $\Mon_k(h)$ is solvable, a symmetric application of Lemma \ref{lem:galcl_containment}\ref{case:galcl_containment_1} implies that  $f_1(X)-t$ and $\hat g_j(X)-t$ 
have the same splitting field $\Omega'$.

Case i) therefore implies that $f=f_1$ is indecomposable. Here, it remains to determine  the indecomposable polynomials whose monodromy group $G\le S_n$ is almost-simple and admits another faithful transitive permutation action (not factoring through an action permutation-equivalent to the first one) in which a cyclic transitive subgroup of $G$ is mapped to a cyclic group with exactly two orbits, and moreover a point stabilizer in one  action is intransitive in the other action. Since the ``two-orbit condition" in particular enforces the second action to be of degree $\le 2n$, \cite[Theorems 5.2.A and 5.2.B]{DM96} show that $G\in \{A_n,S_n\}$ is impossible as soon as $n>8$, whence the entire check reduces to the short finite list of cases by \cite[Theorem]{Mul2}. One directly verifies using Magma that the only possibilities for $(G,\tilde{G}, [a,b])$, where $\tilde{G}$ denotes the image in the second action and $[a,b]$ denotes the two orbit lengths, are $(A_5, A_5(10), [5,5])$, and $(S_5, S_5(10), [5,5])$, where $A_5(10),S_5(10)$ denote the degree-$10$ actions of $A_5,S_5$, resp. 
Such polynomials $f$ are shown in \cite[Theorem 1.1]{DF} to lie in the family \cite[(1.8)]{DF}. 

For Case ii), recall that $f(X)-t$ and $g(X)-t$ (resp., $f_1(X)-t$ and $\hat g_j(X)-t$) have the same splitting field $\Omega$ (resp., $\Omega'$). Thus, since $\Mon_{\overline{k}}(g_s)$ is cyclic of prime order $p$ by Lemma \ref{lem:siegel_by_cyclic}, and hence  the ramification index of $\Omega/\Omega'$ at $\infty$ is $\deg(f_{2}\circ \dots \circ f_r)=\deg(g_s)$ by Remark \ref{rem:abh} (Abhyankar's lemma). In particuar, $r=2$. Since $\Mon_k(f_2)\le \AGL_1(p)$ for a prime $p$, the group $K:=\textrm{Gal}(\Omega/\Omega')$ is a subdirect power of  a subgroup of $\AGL_1(p)$ containing $C_p$.
Let $y$ be a root of $g(X)-t$.
By Proposition \ref{prop:intermediate}, $y$ generates all of $\Omega$ over the fixed field of $\soc(K)$. But $y$ must have a minimal polynomial of degree $p$ already over $\Omega'$. This implies the diagonality of $\soc(K)$, and hence of $K$ via Remark \ref{rem:diag}. Note furthermore that $f=f_1\circ f_2$ must be right-unique due to Proposition \ref{lem:coprime}\ref{case:norittstepp_5}. The list of groups $\Mon_k(f_1\circ f_2)$ with this diagonality property is classified  in \cite[Theorem 1.2]{BKN}, see  \cite[Table 1]{BKN} for the finite list. Running with Magma on the entries of that list, we find the the groups giving rise to a minimally reducible pair $f,g$  with $f\in k[X]$ and $g\in k(X)$ Siegel. This is the list given in Case 2) of the assertion. Since all the occurring Siegel functions have an inertia group at $\infty$ with two different orbit lengths, none of the functions $g$ is  Siegel over $\mathbb{Q}$. 

For Case iii), note that over $\oline k$, one has $\Mon_{\overline{k}}(g_{s-1}\circ g_s) = C_4$ and hence $(f,g)$ cannot be a minimally reducible pair over $\overline{k}$ by Proposition \ref{lem:norittstepp}\ref{case:norittstepp_2}. Due to Lemma \ref{lem:friedarg}, a corresponding minimally reducible pair over $\overline{k}$ must involve  proper left-factors $\tilde{f}$ of $f$ and $\tilde{g}$ of $g$, so that $(\tilde{f},\tilde{g})$ fall into Case i) or ii), and so $\Mon_{\overline{k}}(\tilde{f})$ must either be as in Assertion (1) or one of the even-degree cases in Assertion (2). But since $A_5\le \Mon_{\overline{k}}(\tilde{f})\trianglelefteq \Mon_{k}(\tilde{f})\le S_5$ in (1), and $\Mon_{\overline{k}}(\tilde{f})=\Mon_k(\tilde{f})$ for all the even-degree cases in (2), cf.\ Table 1 of \cite{BKN}, reducibility of $(\tilde{f},\tilde{g})$ over $\overline{k}$ implies reducibility even over $k$, contradicting the minimal reducibility of $(f,g)$.
\end{proof} 

Finally we consider cases where {$\Mon(g_j)$ is affine}. This relies more concretely on the classification of indecomposable Siegel functions over $\mathbb{Q}$ in \cite[Theorem 5.2]{Mul3}.

\begin{lem}
\label{lem:no_affine_siegel}
There is no minimally reducible pair $f\in \mathbb{Q}[X]$, $g\in \mathbb{Q}(X)$, for which $g$ is a Siegel function over $\mathbb{Q}$ factoring through an indecomposable function in $\mQ(X)$ with a nonsolvable affine monodromy group.
\end{lem}
\begin{proof}
We use the setup from the beginning of the section and, as above, assume \eqref{ass:g_j}. 
In particular, $f=f_1\circ h$ for an indecomposable $f_1\in\mQ[X]$ and $h\in \mQ[X]$ with solvable monodromy and   $\Mon_{\mathbb{Q}}(f)$ has only the single nonabelian composition factor $\soc(\Mon_{\mathbb{Q}}(f_1))$.  
As in the setup, since $\Mon_{\mathbb{Q}}(f)\cong\Mon_{\mathbb{Q}}(g)$ as abstract groups, there is only a single index $j$ such that $\Mon_{\mathbb{Q}}(g_j)$ is nonsolvable.

Assume on the contrary that $\Mon_{\mathbb{Q}}(g_j)$ is affine. By \cite[Theorem 5.2]{Mul3}, the only affine monodromy groups of indecomposable Siegel functions over $\mathbb{Q}$ are $\AGL_3(2)$ and $C_2^4\rtimes S_5(\le S_{16})$. The case $\Mon_{\mathbb{Q}}(g_j)=\AGL_3(2)$ does not occur since the simple composition factor $\PSL_3(2)$ of $\AGL_3(2)$ does not occur as a composition factor of the monodromy of a polynomial over $\mathbb{Q}$ by \cite{Mul2}. (This can also be seen using Fried's branch cycle lemma, cf.\ \cite[Section 3]{Fri73}.) Thus, we are in the second case, that is, $\Mon_{\mathbb{Q}}(g_j)=C_2^4\rtimes S_5$. Moreover,  $\Mon_{\mathbb{Q}}(f_1)=S_5$ or $\PGL_2(5) (\le S_6)$, since these are the only faithful actions of the almost-simple group $S_5$ with a cyclic transitive subgroup. Since $K=\ker(\Mon_{\mathbb{Q}}(f)\to \Mon_{\mathbb{Q}}(f_1))$ is solvable, it has to be contained in the kernel of any $S_5$-quotient of $\Mon_\mQ(f)$ (since $S_5$ has no nontrivial solvable normal subgroups). Hence the quotient $\Mon_{\mathbb{Q}}(f_1)$ is the unique (nonsolvable) almost-simple quotient of $\Mon_{\mathbb{Q}}(f)$. Note that the ramification index over $t\mapsto \infty$ in the Galois subextension corresponding to this quotient is either $5$ or $6$. Thus, every normal subgroup $N\trianglelefteq \Mon_{\mathbb{Q}}(f)=\Mon_{\mathbb{Q}}(g)$ similarly has a unique $S_5$-quotient, if any, and the ramification index over a place extending $t\mapsto \infty$ in the extension corresponding to this $S_5$-quotient must divide either $5$ or $6$ by Remark \ref{rem:abh} (since its corresponding cover  must be the pullback of the Galois closure of $f_1$). But $N=\ker(\Mon_{\mathbb{Q}}(g)\to \Mon_{\mathbb{Q}}(\widehat{g}_{j-1}))$ is such a normal subgroup, whose $S_5$-quotient has ramification index $4$ (as is easily seen from the fact that the splitting field of $g_j(X)-t$ is ramified over $t\mapsto \infty$ with ramification index $8$, see \cite[Theorem 5.2]{Mul3}). This contradiction shows that $\Mon_{\mathbb{Q}}(g_j)$ must be almost-simple. 
\end{proof}

\begin{proof}[Proof of Theorem \ref{thm:Hilbert}]
Using the setup from the beginning of the section. By Lemma \ref{lem:only_one_nonsolv2} we may assume \eqref{ass:g_j} holds. 
The conclusion then follows from Lemma \ref{lem:poly_nonaffinesiegel} if $\Mon(g_j)$ is nonaffine, and Lemma \ref{lem:no_affine_siegel} if it is affine.
\end{proof}

\begin{rem}
\label{rem:limitations}
(1) As seen,  Theorem \ref{thm:solv} and Lemma \ref{lem:poly_nonaffinesiegel} also yield a solution to the Hilbert-Siegel problem over arbitrary number fields $k$ under the assumption that $f\in k[X]$ does not factor through a solvable of degree $2$ or $4$, and the {\it additional} assumption that the corresponding Siegel function $g\in k(X)$ does not factor through a nonsolvable indecomposable with affine monodromy group. Since there is only a short finite list of such affine monodromy groups of Siegel functions, cf.\ \cite[Theorem 4.8]{Mul3}, we have a complete solution to the problem for polynomials $f$ which do not factor through an indecomposable $h\in k[X]$ such that $\Mon_k(h)$ has a nonsolvable composition factor occurring in that list (the largest degree exception being $\Mon_k(h)=\PSL_5(2)\le S_{31}$).
\\
(2)
The case of even degree solvable polynomials, left open in Theorem \ref{thm:Hilbert}, requires additional ideas. In particular, the diagonality assertion of Theorem \ref{thm:diagkern} is no longer valid for minimally reducible pairs $(f,g)$ with $f$ a polynomial and $g$ a Siegel function, in case $\deg(f_r)\in \{2,4\}$, {as the following paragraph shows}. 

By \cite[Section 5.3.2]{Mul3}, there is a family of indecomposable degree-$16$ Siegel functions $g$ over $\mathbb{Q}$ with monodromy group $\Mon_{\mQ}(g) = S_4\wr C_2$ in the product-type action, and such that the {\it imprimitive} wreath product action of the same monodromy group gives rise to a polynomial ramification type. Comparing point and block stabilizers in these two actions, one furthermore sees that the pairs $(f,g)$ of polynomials $f$ and Siegel functions $g$ arising in this way are minimally reducible. 
But  the polynomial $f$ is then of the form $f=f_1\circ f_2$ with $\Mon_k(f_1)=C_2$, $\Mon_k(f_2)=S_4$, and $\soc(\ker(\Mon_k(f)\to \Mon_k(f_1)))=\soc(S_4^2)=C_2^4$.

This is also an example where 
$\Red_f(\Z)$ contains an infinite set $g(\mQ)\cap \Z$ that is not contained in $f_1(\mQ)$ for the (unique) indecomposable left factor $f_1$ of $f$.
\end{rem}

\subsection{Functional equations}\label{sec:functional}
Let $k$ be an algebraically closed field of characteristic $0$. 
\begin{proof}[Proof of Corollary \ref{cor:genus0}]
Let $F(X,Y)\in k[X,Y]$ be a right-reduced irreducible factor of $f(X)-g(Y)$ defining a genus-$0$ curve  birational to $F(X,Y)=0$. Let   $k(x,y)$ be its genus-$0$ function field, so that $F(x,y)=0$. Set $t=f(x)=g(y)$. Then $k(x)\cap k(y)$ is  a rational function field by a theorem of L\"uroth.  
Furthermore, we may write $k(x)\cap k(y)=k(s)$ where $s=f_1(x)=g_1(y)$ and $t=w_1(s)$ for some polynomials   $f_1,g_1,w_1\in k[X]\setminus k$. 
Since the factor is assumed to be right-reduced, we get $k(s)=k(t)$, and  hence $\deg(w_1)=1$. 
The  reducibility of $f(X)-g(Y)\in k[X,Y]$ thus implies that of $f_1(X)-g_1(Y)\in k[X,Y]$. 
Theorem \ref{thm:DLS} then implies that one of the following holds: 
\vskip 1mm
\noindent
(a) $f_1$ and $g_1$ have a common left factor $h\in k[X]$ with $\deg(h)>1$, or 
\vskip 1mm
\noindent
(b)  $f_1=w_2\circ h_1\circ f_2$ and $g_1=w_2\circ h_2\circ g_2$, where $\deg(w_2)=1$, and $\{h_1,h_2\}$ is either $\{T_4, -T_4\}$ or is one of the pairs of polynomials of degree $7,11,13,15,21,31$ in  \cite[\S 5]{CC}. 

A direct computer check, using Magma for the possibilities of ramification  in \cite[Theorem]{Mul2}, shows that the only pairs $\{h_1,h_2\}$ in \cite{CC} with a genus-$0$ factor are certain pairs of degree $7$ or $13$, cf.\ Remark \ref{rem:genus0} for a concrete description. 

In the first case (a), we have 
$f_1=h\circ f_2$ and $g_1=h\circ g_2$, for $h,f_2,g_2\in k[X]\setminus  k$. Since $k(x)\cap k(y)=k(s)$, it follows that $k(x_2)\cap k(y_2)=k(s)$, where $x_2=f_2(x)$ and $y_2=g_2(y)$. Since $h(x_2)=h(y_2)=s$, we obtain two distinct roots  $x_2$ and $y_2$ of $h(X)-s\in k(s)[X]$, so that $h(X)-h(Y)\in k[X,Y]$ has a nondiagonal irreducible factor $H(X,Y)\in k[X,Y]$ whose corresponding curve  $H(X,Y)=0$ is birational to a  genus-$0$ curve $\mathcal D$. Since $k=\oline k$, it now follows from \cite[Theorem\ 1]{AZ2} that\footnote{In \cite{AZ2}, $w_2$ is denoted by $A$, and $u_1$ corresponds to $S$ in (1) and $M$ in (2)-(5).} there exist $w_2,h',u_1\in k[X]\setminus k$ such that $h=w_2\circ h'\circ u_1$ and $h'$ is either 
\vskip 1mm
\noindent
i) $X^d$ for $d\geq 3$, or 
\vskip 1mm
\noindent
ii) $T_d$ for  $d\geq 3$, or 
\vskip 1mm
\noindent
iii) one of the polynomials $ P_i$, $i=1,2,3$  appearing in the statement of the corollary. 
\vskip 1mm
\noindent
Moreover, \cite{AZ2} shows that $\deg(u_1)=1$ in cases ii) and iii), 
and that $k(h'(u_1(x_2)))=k(h'(u_1(y_2)))$.
Since  $k(x_2)\cap k(y_2)=k(s)$, it follows that $k(h'(u_1(x_2)))$ is of degree $1$ over $k(s)$, and hence $\deg(w_2)=1$.
For simplicity, in what follows for case (a) we replace $h,f_2,g_2$ by  $h'$,   $u_1\circ f_2$ and $u_1\circ g_2$, respectively, and consequently $x_2,y_2$ by their new values $f_2(x),g_2(y)$, respectively. We modify $H$ and $\mathcal D$ so that they still correspond to a nondiagonal factor of $h(X)-h(Y)$, and a curve birational to $H=0$, resp. 

We next note that case i) does not occur: For, both $x_2$ and $y_2$ are roots of $X^d-s\in k(s)[X]$. Since $k=\oline k$ this implies that $k(x_2)=k(y_2)$ is of degree $d\geq 3$ over $k(s)$, contradicting $k(x_2)\cap k(y_2)=k(s)$.  

For the remaining cases ii) and iii) of a), as well as the cases in b),  it remains to determine the possibilities for the right factors $f_2$ and $g_2$.
In case b), set $h= h_1$, and let $\mathcal D$ be a curve birational to $H(X,Y)=0$, for an irreducible factor $H(X,Y)$ of $h_1(X)-h_2(Y)\in k[X,Y]$. 
Since $\mathcal D$ is of genus $0$ and $k=\oline k$, the function field $k(\mathcal D)=k(z)$ is  rational. 
Furthermore, $k(z)/k(x_2)$ is unramified over the place  $x_2\mapsto \infty$ by Remark \ref{rem:abh}(1). 
Since $k(x,z)$ is contained in the genus-$0$ field $k(x,y)$, it is also of genus $0$. 
Since $k(x)/k(x_2)$  is totally ramified over $x_2\mapsto \infty$, whereas $k(z)/k(x_2)$ is unramified over $x_2\mapsto \infty$, the fields $k(x)$ and $k(z)$ are linearly disjoint over $k(x_2)$ by Remark \ref{rem:abh}(2), and moreover every place of $k(z)$ lying over $x_2\mapsto \infty$ is totally ramified in $k(x,z)/k(z)$ by Remark \ref{rem:abh}(2). 
If $m:=[k(z):k(x_2)]>2$, then the Riemann--Hurwitz formula for the degree-$\deg(f_2)$ extension $k(x,z)/k(z)$ shows the genus $\tilde g=0$ of  $k(x,z)$ satisfies: 
\begin{equation}\label{equ:RH}
2(\tilde g-1)\geq -2\deg(f_2)+\sum_{Q}(\deg(f_2)-1)=(m-2)\deg(f_2)-m,
\end{equation}
where $Q$ runs over the places of $k(z)$ lying over $x_2\mapsto\infty$. 
If $m>2$ and $\deg(f_2)>1$, \eqref{equ:RH} contradicts that  $\tilde g=0$ as above. 
Since $m>2$ in case (a) iii) and in case (b), except for $\{h_1,h_2\} = \{T_4,-T_4\}$, it follows that $\deg(f_2)=1$ in all those cases, and by symmetry $\deg(g_2)=1$, leading to cases (2) and (3) of the assertion with  $\mu:=w_1\circ w_2$, $\lambda_1:=f_2$ and $\lambda_2=g_2$. 

It remains to treat the cases  $h= T_d$ for $d\geq 3$ in a), or $\{h_1,h_2\} = \{T_4,-T_4\}$ in b). 
In these cases we have  $m=2$, and there are exactly two places $Q_1,Q_2$ of $k(z)$ lying over $x_2\mapsto\infty$. 
In particular, \eqref{equ:RH} is an equality, and hence $k(x,z)/k(z)$ is unramified away from  $Q_1,Q_2$. 
Thus the {\it only} finite branch points of $k(x)/k(x_2)$  are the two branch points $x_2\mapsto -2,2$ of $k(z)/k(x_2)$, by Remark \ref{rem:abh}(3). Moreover, Abhyankar's lemma implies that the ramification index of every place in $f_2^{-1}(\pm 2)$ is at most $2$. 
Since $x_2\mapsto -2,2$ are the only unramified places  over two finite branch points $s$ by \cite[Lemma 3.2.]{MZ}, 
it follows that $h\circ f_2$ has only two finite branch points and the $(h\circ f_2)$-preimages of these are all of ramification index $\leq 2$. Such ramification forces that $h_1\circ f_2=\pm T_n\circ v_1$ for $n:=\deg(f)$ and some $v_1\in k[X]$ of degree $1$, 
by \cite[Lemma 3.2.]{MZ}. Similarly, $h_2\circ g_2=\pm T_m\circ v_2$  for $m:=\deg(g)$ and $v_2\in k[X]$ of degree $1$. 

We may therefore write $f=\mu\circ (\pm T_n)\circ \lambda_1$ and $g=\mu\circ (-T_m)\circ \lambda_2$ for linear $\mu,\lambda_1,\lambda_2\in k[X]$, where $d:=\gcd(m,n)\ge 3$. 
So far we allow any combination of plus and minus signs for $f$. But note that $f,g$ cannot have a common left factor $T_2$, since it is linearly related to $X^2$ and thus would contradict minimality as for Case i) above. In case $d$ is even this leaves only the possibility of different signs in $f,g$, as asserted in (1). In case  $d$ is odd, we may assume  without loss of generality that  $n$ odd. Then, by  possibly composing $\mu$, $\lambda_1$, $\lambda_2$ with  $X\mapsto \pm X$ and using the identity $T_n(-X)=-T_n(X)$, we obtain the combination of signs given in (1).
\end{proof}
\begin{rem}\label{rem:genus0}
(a) A computer check further reveals that in case \ref{case:genus0_3} of Corollary \ref{cor:genus0}, the degree-$7$ (resp., degree-$13$) polynomials $h_1,h_2$ have three branch points with branch cycles of orders $2,3,7$ or $2,4,7$ (resp., $2,3,13$). From this, it is possible to obtain concrete parameterizations, namely \begin{equation}\label{equ:h1a}h_1=X(X+1)^3(X+a+3)^3, \text{ with } a^2+a+2=0,\end{equation}
\begin{equation}\label{equ:h1b}h_1=X^4(X-2)^2(X-a), \text{ with } a^2+a+2=0,\text{ and }\end{equation} 
\begin{equation}\label{equ:h1c}h_1 = X^3(X-27)(X^3 + (-16a^3 - 12a^2 - 29a + 51)X^2 + (16a^3 - 96a^2 - 25a - 618)X \end{equation} $$+ \frac{1}{3}(21872a^3 + 29148a^2 + 58408a - 53112))^3, \text{ with } a=\zeta_{13}+\zeta_{13}^3+\zeta_{13}^9$$
in the three respective cases. In all cases, $h_2$ is then obtained as $h_2(X)=\gamma/\overline{\gamma} \cdot \overline{h_1}(X)$, where $\gamma$ denotes the unique finite nonzero branch point of $h_1$, and $\overline{z}$ denotes the complex conjugate of $z$.\\
(b) For the converse of Corollary \ref{cor:genus0} we show that (1)-(3) indeed have right-reduced irreducible factors defining curves of genus $0$.   For case (1) with $d:=\gcd(m,n)$, the polynomials $T_d(X)\pm T_d(Y)\in\mC[X,Y]$ admit\footnote{In fact any factor of it defines a genus-$0$ curve.} an irreducible factor $H(X,Y)=X^2-2XY\cos(\pi/d)+Y^2-4\sin^2(\pi/d)$  \cite[Prop.\ 2.2]{AZ2}, so that $H(T_{n/d}(X),T_{m/d}(Y))$ is an irreducible factor of $T_n(X)\pm T_m(Y)$    defining a right-reduced genus-$0$ curve. 
Composing with linear polynomials one obtains the desired factor of $f(X)-g(Y)$. 
In case (2), the irreducible nondiagonal component of genus $0$ is the unique one:  $(f(X)-f(Y))/(X-Y)$.  In case (3), the degrees of the irreducible factors defining a genus-$0$ curve are: 
$3$ and $4$ (i.e.\ two factors of genus $0$) in \eqref{equ:h1a}; 
$3$ in \eqref{equ:h1b}; and  $4$ in \eqref{equ:h1c}. Since the polynomials in (2) and (3) are indecomposable, the factors are automatically right-reduced. 
\end{rem}

\appendix 
\section{Fried's $(m,n)$-problem}\label{app:mn}
As a consequence of Theorem \ref{thm:DLS}, we have the following answer to Fried's $(m,n)$-problem.  
\begin{cor}
\label{rem:m-n} Let $P,Q \in \mC[X]$ be simply branched polynomials of degrees $m:=\deg(Q)\ge 2$ and $n:=\deg(P)\ge \max\{m,3\}$ such that $P$ and $Q$ do not satisfy a relation $P=Q\circ \mu$ with $\mu\in \mathbb{C}[X]$ linear. If $n=3$, then assume moreover that the sets of finite branch points of $P$ and of $Q$ are disjoint. Then $Q(f(X))-P(g(Y))$ is irreducible for all $f,g\in \mathbb{C}[X]\setminus\mathbb{C}$.
\end{cor}
Note that the assumption $P\neq Q\circ\mu$ is required for  $P(X)-Q(Y)\in k[X,Y]$ to be irreducible, and clearly holds if the branch loci of $P$ and $Q$ are disjoint.
\begin{proof}
Assume on the contrary that $Q(f(X))-P(g(Y)) \in \mC[X,Y]$ is reducible for some $f(X),g(X) \in \mC[X] \setminus \mC$.  Let $(h_Q, h_P)$ be a corresponding minimally reducible pair, so that in particular $h_Q$ is a left-factor of $Q\circ f$ and $h_P$ is a left-factor of $P\circ g$.

Note first that, as in Remark \ref{rem:further-Ritt}, 
a simply branched polynomial of degree $n\ge 4$ cannot be the left-factor of a Ritt move. Therefore, when $n\ge 4$, the polynomial $h_P$ must admit (the simply branched, hence indecomposable) $P$ as a left-factor.  
By Theorem \ref{thm:DLS-C}, $P$ is either a left factor of $Q\circ f$, or $P$ is one of the exceptional polynomials in  Theorem \ref{thm:DLS-C}\ref{case:DLS-C_2}. The latter case cannot occur since the exceptional polynomials in Theorem \ref{thm:DLS-C}\ref{case:DLS-C_2} are not simply-branched.  In the former case, the same reasoning shows that $P=Q\circ \mu$ for a linear $\mu\in\mathbb{C}[X]$, contradicting our assumptions. 

Next, assume $n=3$. Since  $P$ is simply branched of degree $3$, it is linearly related to $T_3$ over $\mC$. 
Since the only Ritt moves with $T_3$ as a left factor are  $T_3\circ (T_m\circ\ell)=T_m\circ (T_3\circ \ell)$  or $(-T_2)\circ (T_3\circ \ell)=T_3\circ (-T_2\circ \ell)$, for $m\geq 2$ and $\ell\in\mC[X]$ of degree $1$, by Ritt's second theorem,
and since the branch loci of $T_m$ and $-T_2$ are contained in that of $T_3$, it follows that 
any indecomposable left-factor of $P\circ g$ must have a branch point set which is contained in the one of $P$. On the other hand,  every non-linear left factor of $Q\circ f$ must share at least one finite branch point with $Q$, as in Remark \ref{rem:further-Ritt}. 
Together with the assumption that the sets of finite branch points of $P$ and $Q$ are disjoint, these two observations yield that in every pair $(\tilde{h}_P,\tilde{h}_Q)$ of indecomposable left-factors  of $(h_P, h_Q)$, the branch point set of $\tilde{h}_P$ must be different from the one of $\tilde{h}_Q$. This contradicts the fact that, for all pairs in Theorem \ref{thm:DLS-C}, these two branch point sets coincide.
\end{proof}

\bibliography{biblio1}

@book {Cox,
    AUTHOR = {Cox, David A.},
     TITLE = {Galois theory},
    SERIES = {Pure and Applied Mathematics (Hoboken)},
   EDITION = {Second},
 PUBLISHER = {John Wiley \& Sons, Inc., Hoboken, NJ},
      YEAR = {2012},
     PAGES = {xxviii+570},
      ISBN = {978-1-118-07205-9},
   MRCLASS = {12F10 (00-01 01A05)},
  MRNUMBER = {2919975},
       DOI = {10.1002/9781118218457},
       URL = {https://doi-org.technion.idm.oclc.org/10.1002/9781118218457},
}

@article {AZ,
    AUTHOR = {Avanzi, R. M. and Zannier, U. M.},
     TITLE = {Genus one curves defined by separated variable polynomials and
              a polynomial {P}ell equation},
   JOURNAL = {Acta Arith.},
  FJOURNAL = {Acta Arithmetica},
    VOLUME = {99},
      YEAR = {2001},
    NUMBER = {3},
     PAGES = {227--256},
      ISSN = {0065-1036,1730-6264},
   MRCLASS = {11D41 (11C08 12E10 14H52)},
MRREVIEWER = {Yuri\ Bilu},
       DOI = {10.4064/aa99-3-2},
       URL = {https://doi.org/10.4064/aa99-3-2},
}

@article {AZ2,
    AUTHOR = {Avanzi, R. M. and Zannier, U. M.},
     TITLE = {The equation {$f(X)=f(Y)$} in rational functions {$X=X(t)$},
              {$Y=Y(t)$}},
   JOURNAL = {Compositio Math.},
  FJOURNAL = {Compositio Mathematica},
    VOLUME = {139},
      YEAR = {2003},
    NUMBER = {3},
     PAGES = {263--295},
      ISSN = {0010-437X,1570-5846},
   MRCLASS = {14H45 (11C08 11G30 11R09 12E05)},
MRREVIEWER = {Yuri\ Bilu},
       DOI = {10.1023/B:COMP.0000018136.23898.65},
       URL = {https://doi.org/10.1023/B:COMP.0000018136.23898.65},
}

@article {Bilu,
    AUTHOR = {Bilu, Y. F.},
     TITLE = {Quadratic factors of {$f(x)-g(y)$}},
   JOURNAL = {Acta Arith.},
  FJOURNAL = {Acta Arithmetica},
    VOLUME = {90},
      YEAR = {1999},
    NUMBER = {4},
     PAGES = {341--355},
      ISSN = {0065-1036,1730-6264},
   MRCLASS = {12E05 (14H30)},
MRREVIEWER = {Jean-Marc\ Couveignes},
       DOI = {10.4064/aa-90-4-341-355},
       URL = {https://doi.org/10.4064/aa-90-4-341-355},
}

@article {BiTi,
    AUTHOR = {Bilu, Y. F. and Tichy, R. F.},
     TITLE = {The {D}iophantine equation {$f(x)=g(y)$}},
   JOURNAL = {Acta Arith.},
  FJOURNAL = {Acta Arithmetica},
    VOLUME = {95},
      YEAR = {2000},
    NUMBER = {3},
     PAGES = {261--288},
      ISSN = {0065-1036,1730-6264},
   MRCLASS = {11D41 (14G05)},
MRREVIEWER = {Yann\ Bugeaud},
       DOI = {10.4064/aa-95-3-261-288},
       URL = {https://doi.org/10.4064/aa-95-3-261-288},
}

@article {HT,
    AUTHOR = {Hajdu, L. and Tijdeman, R.},
     TITLE = {The {D}iophantine equation {$f(x) = g(y)$} for polynomials
              with simple rational roots},
   JOURNAL = {J. Lond. Math. Soc. (2)},
  FJOURNAL = {Journal of the London Mathematical Society. Second Series},
    VOLUME = {108},
      YEAR = {2023},
    NUMBER = {1},
     PAGES = {309--339},
      ISSN = {0024-6107,1469-7750},
   MRCLASS = {11N32 (11B75 11D41 11P05)},
MRREVIEWER = {Maciej\ Ulas},
       DOI = {10.1112/jlms.12746},
       URL = {https://doi.org/10.1112/jlms.12746},
}

@article {BT,
    AUTHOR = {Bukh, B. and Tsimerman, J.},
     TITLE = {Sum-product estimates for rational functions},
   JOURNAL = {Proc. Lond. Math. Soc. (3)},
  FJOURNAL = {Proceedings of the London Mathematical Society. Third Series},
    VOLUME = {104},
      YEAR = {2012},
    NUMBER = {1},
     PAGES = {1--26},
      ISSN = {0024-6115,1460-244X},
   MRCLASS = {11B30 (11B75)},
MRREVIEWER = {Mei\ Chu\ Chang},
       DOI = {10.1112/plms/pdr018},
       URL = {https://doi.org/10.1112/plms/pdr018},
}

@article {CC,
    AUTHOR = {Cassou-Nogu\`es, P. and Couveignes, J.-M.},
     TITLE = {Factorisations explicites de {$g(y)-h(z)$}},
   JOURNAL = {Acta Arith.},
  FJOURNAL = {Acta Arithmetica},
    VOLUME = {87},
      YEAR = {1999},
    NUMBER = {4},
     PAGES = {291--317},
      ISSN = {0065-1036,1730-6264},
   MRCLASS = {11C08 (12Y05)},
MRREVIEWER = {Maurice\ Mignotte},
       DOI = {10.4064/aa-87-4-291-317},
       URL = {https://doi.org/10.4064/aa-87-4-291-317},
}

@article {DF,
    AUTHOR = {D\`ebes, P. and Fried, M. D.},
     TITLE = {Integral specialization of families of rational functions},
   JOURNAL = {Pacific J. Math.},
  FJOURNAL = {Pacific Journal of Mathematics},
    VOLUME = {190},
      YEAR = {1999},
    NUMBER = {1},
     PAGES = {45--85},
      ISSN = {0030-8730,1945-5844},
   MRCLASS = {14G32 (12E25)},
MRREVIEWER = {Peter\ M\"{u}ller},
       DOI = {10.2140/pjm.1999.190.45},
       URL = {https://doi.org/10.2140/pjm.1999.190.45},
}

@article{Eren, 
AUTHOR = {Ehrenfeucht, A.},
TITLE = {Kryterium absolutnej nierozkladaln\'{o}sci wielomia\'ow},
JOURNAL = {Prace Mat.},
VOLUME = {2},
YEAR = {1958},
PAGES = {167-169},
}

@article {Fried-Schur,
    AUTHOR = {Fried, M. D.},
     TITLE = {On a conjecture of {S}chur},
   JOURNAL = {Michigan Math. J.},
  FJOURNAL = {Michigan Mathematical Journal},
    VOLUME = {17},
      YEAR = {1970},
     PAGES = {41--55},
      ISSN = {0026-2285,1945-2365},
   MRCLASS = {10.76},
MRREVIEWER = {A.\ L.\ Whiteman},
       URL = {http://projecteuclid.org.technion.idm.oclc.org/euclid.mmj/1029000374},
}

@article {AD-Survey,
    AUTHOR = {Benedetto, R. and Ingram, P. and Jones, R. and
              Manes, M. and Silverman, J. H. and Tucker, T.
              J.},
     TITLE = {Current trends and open problems in arithmetic dynamics},
   JOURNAL = {Bull. Amer. Math. Soc. (N.S.)},
  FJOURNAL = {American Mathematical Society. Bulletin. New Series},
    VOLUME = {56},
      YEAR = {2019},
    NUMBER = {4},
     PAGES = {611--685},
      ISSN = {0273-0979,1088-9485},
   MRCLASS = {37P05 (11G50 37P15 37P20 37P25 37P30 37P45 37P55)},
       DOI = {10.1090/bull/1665},
       URL = {https://doi.org/10.1090/bull/1665},
}

@article {DLS,
    AUTHOR = {Davenport, H. and Lewis, D. J. and Schinzel, A.},
     TITLE = {Equations of the form {$f(x)=g(y)$}},
   JOURNAL = {Quart. J. Math. Oxford Ser. (2)},
  FJOURNAL = {The Quarterly Journal of Mathematics. Oxford. Second Series},
    VOLUME = {12},
      YEAR = {1961},
     PAGES = {304--312},
      ISSN = {0033-5606,1464-3847},
   MRCLASS = {14.49 (10.13)},
MRREVIEWER = {B.\ J.\ Birch},
       DOI = {10.1093/qmath/12.1.304},
       URL = {https://doi.org/10.1093/qmath/12.1.304},
}

@article {DS,
    AUTHOR = {Davenport, H. and Schinzel, A.},
     TITLE = {Two problems concerning polynomials},
   JOURNAL = {J. Reine Angew. Math.},
  FJOURNAL = {Journal f\"ur die Reine und Angewandte Mathematik. [Crelle's
              Journal]},
    VOLUME = {214/215},
      YEAR = {1964},
     PAGES = {386--391},
      ISSN = {0075-4102,1435-5345},
   MRCLASS = {12.30},
MRREVIEWER = {M.\ J.\ Wonenburger},
       DOI = {10.1515/crll.1964.214-215.386},
       URL = {https://doi.org/10.1515/crll.1964.214-215.386},
}

@article {Fried12,
    AUTHOR = {Fried, M. D.},
     TITLE = {Variables separated equations: strikingly different roles for
              the branch cycle lemma and the finite simple group
              classification},
   JOURNAL = {Sci. China Math.},
  FJOURNAL = {Science China. Mathematics},
    VOLUME = {55},
      YEAR = {2012},
    NUMBER = {1},
     PAGES = {1--72},
      ISSN = {1674-7283,1869-1862},
   MRCLASS = {12E30 (11G18 11R58 12E05 12F10 14H30 20D05)},
MRREVIEWER = {Peter\ M\"uller},
       DOI = {10.1007/s11425-011-4324-4},
       URL = {https://doi.org/10.1007/s11425-011-4324-4},
}

@article {Fried23,
    AUTHOR = {Fried, M. D.},
     TITLE = {Taming genus 0 (or 1) components on variables-separated
              equations},
   JOURNAL = {Albanian J. Math.},
  FJOURNAL = {Albanian Journal of Mathematics},
    VOLUME = {17},
      YEAR = {2023},
    NUMBER = {2},
     PAGES = {19--80},
      ISSN = {1930-1235},
   MRCLASS = {14H50 (14H55 20B15 20E22)},
MRREVIEWER = {Jos\'e\ Javier\ Etayo},
}

@article {FG,
    AUTHOR = {Fried, M. D. and Gusi\'c, I.},
     TITLE = {Schinzel's problem: imprimitive covers and the monodromy
              method},
   JOURNAL = {Acta Arith.},
  FJOURNAL = {Acta Arithmetica},
    VOLUME = {155},
      YEAR = {2012},
    NUMBER = {1},
     PAGES = {27--40},
      ISSN = {0065-1036,1730-6264},
   MRCLASS = {11C08 (12E05 13P05)},
MRREVIEWER = {Mihai\ Cipu},
       DOI = {10.4064/aa155-1-3},
       URL = {https://doi.org/10.4064/aa155-1-3},
}

@book {Har,
    AUTHOR = {Hartshorne, Robin},
     TITLE = {Algebraic geometry},
    SERIES = {Graduate Texts in Mathematics},
    VOLUME = {No. 52},
 PUBLISHER = {Springer-Verlag, New York-Heidelberg},
      YEAR = {1977},
     PAGES = {xvi+496},
      ISBN = {0-387-90244-9},
   MRCLASS = {14-01},
  MRNUMBER = {463157},
MRREVIEWER = {Robert\ Speiser},
}

@article {FM,
    AUTHOR = {Fried, M. D. and MacRae, R. E.},
     TITLE = {On the invariance of chains of fields},
   JOURNAL = {Illinois J. Math.},
  FJOURNAL = {Illinois Journal of Mathematics},
    VOLUME = {13},
      YEAR = {1969},
     PAGES = {165--171},
      ISSN = {0019-2082},
   MRCLASS = {12.45},
MRREVIEWER = {E.\ Inaba},
       URL = {http://projecteuclid.org/euclid.ijm/1256053748},
}

@article {Fried74,
    AUTHOR = {Fried, M. D.},
     TITLE = {On {H}ilbert's irreducibility theorem},
   JOURNAL = {J. Number Theory},
  FJOURNAL = {Journal of Number Theory},
    VOLUME = {6},
      YEAR = {1974},
     PAGES = {211--231},
      ISSN = {0022-314X,1096-1658},
   MRCLASS = {12A20 (10M05)},
MRREVIEWER = {G.\ Thomas\ Frey},
       DOI = {10.1016/0022-314X(74)90015-8},
       URL = {https://doi.org/10.1016/0022-314X(74)90015-8},
}

@article{Fri73,
author = {Fried, M. D.},
year = {1973},
month = {03},
pages = {128--146},
title = {Fields of definition of function fields and a problem in the reducibility of polynomials in two variables},
volume = {17},
journal = {Illinois Journal of Mathematics},
doi = {10.1215/ijm/1256052044}
}

@article {KN,
    AUTHOR = {K\"onig, J. and Neftin, D.},
     TITLE = {Reducible fibers of polynomial maps},
   JOURNAL = {Int. Math. Res. Not. IMRN},
  FJOURNAL = {International Mathematics Research Notices. IMRN},
      YEAR = {2024},
    NUMBER = {6},
     PAGES = {5373--5402},
      ISSN = {1073-7928,1687-0247},
   MRCLASS = {14D05},
       DOI = {10.1093/imrn/rnad251},
       URL = {https://doi.org/10.1093/imrn/rnad251},
}

@misc{BKN,
      title={Monodromy groups of polynomials of composition length 2}, 
      author={Behajaina, A. and König, J. and Neftin, D.},
      year={2026},
      primaryClass={math.NT} 
}

@article {Mul,
    AUTHOR = {Müller, P.},
     TITLE = {Kronecker conjugacy of polynomials},
   JOURNAL = {Trans. Amer. Math. Soc.},
  FJOURNAL = {Transactions of the American Mathematical Society},
    VOLUME = {350},
      YEAR = {1998},
    NUMBER = {5},
     PAGES = {1823--1850},
      ISSN = {0002-9947,1088-6850},
   MRCLASS = {11C08 (12E05 20D99)},
MRREVIEWER = {Robert\ M.\ Guralnick},
       DOI = {10.1090/S0002-9947-98-02123-0},
       URL = {https://doi.org/10.1090/S0002-9947-98-02123-0},
}

@article {DZ,
    AUTHOR = {Ding, Zhiguo and Zieve, Michael E.},
     TITLE = {Extensions of absolute values on two subfields},
   JOURNAL = {J. Algebra},
  FJOURNAL = {Journal of Algebra},
    VOLUME = {598},
      YEAR = {2022},
     PAGES = {105--119},
      ISSN = {0021-8693,1090-266X},
   MRCLASS = {12J20 (11S15 13J05)},
  MRNUMBER = {4379276},
MRREVIEWER = {G\'erard\ Leloup},
       DOI = {10.1016/j.jalgebra.2022.01.016},
       URL = {https://doi.org/10.1016/j.jalgebra.2022.01.016},
}

@incollection {Mul2,
    AUTHOR = {Müller, P.},
     TITLE = {Primitive monodromy groups of polynomials},
 BOOKTITLE = {Recent developments in the inverse {G}alois problem
              ({S}eattle, {WA}, 1993)},
    SERIES = {Contemp. Math.},
    VOLUME = {186},
     PAGES = {385--401},
 PUBLISHER = {Amer. Math. Soc., Providence, RI},
      YEAR = {1995},
      ISBN = {0-8218-0299-2},
   MRCLASS = {20B15 (11C08 12F10 20B20)},
MRREVIEWER = {Robert\ M.\ Guralnick},
       DOI = {10.1090/conm/186/02193},
       URL = {https://doi.org/10.1090/conm/186/02193},
}

@article{Mul3,
    AUTHOR = {Müller, P.},
     TITLE = {Permutation groups with a cyclic two-orbits subgroup and
              monodromy groups of {L}aurent polynomials},
   JOURNAL = {Ann. Sc. Norm. Super. Pisa Cl. Sci. (5)},
  FJOURNAL = {Annali della Scuola Normale Superiore di Pisa. Classe di
              Scienze. Serie V},
    VOLUME = {12},
      YEAR = {2013},
    NUMBER = {2},
     PAGES = {369--438},
      ISSN = {0391-173X,2036-2145},
   MRCLASS = {20B15 (12E25 14H30)},
MRREVIEWER = {Johannes\ Siemons},
}

@article {GS,
    AUTHOR = {Guralnick, R. M. and Shareshian, J.},
     TITLE = {Symmetric and alternating groups as monodromy groups of
              {R}iemann surfaces. {I}. {G}eneric covers and covers with many
              branch points},
      NOTE = {With an appendix by Guralnick and R. Stafford},
   JOURNAL = {Mem. Amer. Math. Soc.},
  FJOURNAL = {Memoirs of the American Mathematical Society},
    VOLUME = {189},
      YEAR = {2007},
    NUMBER = {886},
     PAGES = {vi+128},
      ISSN = {0065-9266,1947-6221},
   MRCLASS = {14H30 (14H55 30F20)},
  MRNUMBER = {2343794},
MRREVIEWER = {Andrei\ B.\ Bogatyr\"ev},
       DOI = {10.1090/memo/0886},
       URL = {https://doi-org.technion.idm.oclc.org/10.1090/memo/0886},
}

@misc{MZ,
      title={On {R}itt's polynomial decomposition theorems}, 
      author={Zieve, M. E. and Müller, P.},
      year={2008},
      eprint={0807.3578},
      archivePrefix={arXiv},
      primaryClass={math.AG},
      note={https://arxiv.org/pdf/0807.3578}
}

@misc{Ostrov, 
title={M.{S}c. {T}hesis, {F}ibonacci numbers as special values of polynomials},
author={Ostrov, A.}, 
year={2025}
}

@misc{Cas68,
 author = {Cassels, J. W. S.},
 title = {Factorization of polynomials in several variables},
 year = {1968},
 language = {English},
 howpublished = {Proc. 15th {Scand}. {Congr}. {Oslo} 1968, {Lect}. {Notes} {Math}. 118, 1-17 (1970)},
 doi = {10.1007/bfb0060248},
 zbMATH = {3323127},
 Zbl = {0203.35001}
}

@misc{KNR24,
      title={Polynomial compositions with large monodromy groups and applications to arithmetic dynamics}, 
      author={König, J. and Neftin, D. and Rosenberg, S.},
      year={2024},
      eprint={2401.17872},
      archivePrefix={arXiv},
      primaryClass={math.NT},
      note={https://arxiv.org/abs/2401.17872}, 
}

@book {Schinzel,
    AUTHOR = {Schinzel, A.},
     TITLE = {Polynomials with special regard to reducibility},
    SERIES = {Encyclopedia of Mathematics and its Applications},
    VOLUME = {77},
      NOTE = {With an appendix by Umberto Zannier},
 PUBLISHER = {Cambridge University Press, Cambridge},
      YEAR = {2000},
     PAGES = {x+558},
      ISBN = {0-521-66225-7},
   MRCLASS = {11R09 (11C08 12D05 12E05 12E25)},
MRREVIEWER = {Art\=uras\ Dubickas},
       DOI = {10.1017/CBO9780511542916},
       URL = {https://doi.org/10.1017/CBO9780511542916},
}

@book {Sch,
    AUTHOR = {Schinzel, A.},
     TITLE = {Andrzej {S}chinzel selecta. {V}ol. {I}},
    SERIES = {Heritage of European Mathematics},
    EDITOR = {Iwaniec, Henryk and Narkiewicz, W\l adys\l aw and Urbanowicz,
              Jerzy},
      NOTE = {Diophantine problems and polynomials},
 PUBLISHER = {European Mathematical Society (EMS), Z\"urich},
      YEAR = {2007},
     PAGES = {xiv+858},
      ISBN = {978-3-03719-038-8},
   MRCLASS = {11-06 (01A75)},
MRREVIEWER = {R.\ C.\ Baker},
}

@misc{Tao12,
  author       = {Tao, T.},
  title        = {When is ${P}(x)-{Q}(y)$ irreducible?},
  year         = {2012},
  note         = {MathOverflow post},
  howpublished = "\url{https://mathoverflow.net/q/105747}",
}

@misc{Red12,
  author       = {Zieve, M. E.},
  title        = {Criteria for irreducibility of polynomial},
  year         = {2012},
  note         = {MathOverflow post},
  howpublished = "\url{https://mathoverflow.net/questions/105304/criteria-for-irreducibility-of-polynomial/}",
}

@unpublished{DR25,
  author       = {Derickx, M. and Rawson, J.},
  title        = {Functions on curves with infinitely many split fibres},
  note         = {Work in preparation},
  year         = {2025},
}

@unpublished{Fried86a,
  author       = {Fried, M. D.},
  title        = {Rigidity and applications of the classification of simple groups to monodromy {P}art {I}{I}-Applications of connectivity: {D}avenport and {H}ilbert--{S}iegel problems},
  note         = {Preprint},
  year         = {1986},
}

@unpublished{Fried86b,
  author       = {Fried, M. D.},
  title        = {The {H}ilbert--{S}iegel problems and {G}roup {T}heory solving cases of them},
  note         = {Preprint, \url{https://www.math.uci.edu/~mfried/paplist-cov/Hilb-Sieg86.pdf}},
  year         = {1986},
}

@misc{Rur12,
  author = {Rurik},
  title  = {Criteria for irreducibility of polynomial},
  howpublished = {Question on MathOverflow},
  year   = {2012},
  note   = {User ``Rurik'' on MathOverflow. 
            \url{https://mathoverflow.net/questions/105304/criteria-for-irreducibility-of-polynomial}
            (accessed 2026-03-04)}
}

@article{Magma,
 author = {Bosma, W. and Cannon, J. and Playoust, C.},
 title = {The {Magma} algebra system. {I}: {The} user language},
 fjournal = {Journal of Symbolic Computation},
 journal = {J. Symb. Comput.},
 issn = {0747-7171},
 volume = {24},
 number = {3-4},
 pages = {235--265},
 year = {1997},
 doi = {10.1006/jsco.1996.0125},
 zbMATH = {1077111},
 Zbl = {0898.68039}
}

@article {Pak2,
    AUTHOR = {Pakovich, F.},
     TITLE = {On algebraic curves {$A(x)-B(y)=0$} of genus zero},
   JOURNAL = {Math. Z.},
  FJOURNAL = {Mathematische Zeitschrift},
    VOLUME = {288},
      YEAR = {2018},
    NUMBER = {1-2},
     PAGES = {299--310},
      ISSN = {0025-5874,1432-1823},
   MRCLASS = {14H50 (14G05 14H45 37F10)},
MRREVIEWER = {Maria\ Montanucci},
       DOI = {10.1007/s00209-017-1889-9},
       URL = {https://doi.org/10.1007/s00209-017-1889-9},
}

@article {Pak,
    AUTHOR = {Pakovich, F.},
     TITLE = {On the equation {$P(f)=Q(g)$}, where {$P,Q$} are polynomials
              and {$f,g$} are entire functions},
   JOURNAL = {Amer. J. Math.},
  FJOURNAL = {American Journal of Mathematics},
    VOLUME = {132},
      YEAR = {2010},
    NUMBER = {6},
     PAGES = {1591--1607},
      ISSN = {0002-9327,1080-6377},
   MRCLASS = {30D05 (39B32)},
MRREVIEWER = {Ilpo\ Laine},
}

@Misc{Pak3,
 Author = {Pakovich, F.},
 Title = {On intersection of lemniscates of rational functions},
 Year = {2023},
 Language = {English},
 HowPublished = {arXiv:2309.04983, Preprint}
 }

@article{Pak5,
title = {Prime and composite {L}aurent polynomials},
journal = {Bulletin des Sciences Math\'ematiques},
volume = {133},
number = {7},
pages = {693-732},
year = {2009},
issn = {0007-4497},
doi = {https://doi.org/10.1016/j.bulsci.2009.06.003},
author = {Pakovich, F.},
}

@Misc{Zieve1,
 Author = {Do, T. and Hallett, J. and Huang, X. and Jiang, Y. and Weiss, B. and Wells, E. and Zieve, M. E. },
 Title = {On the functional equation f(u)=g(v) in complex polynomials f,g and meromorphic functions u,v, {I}: the irreducible case},
 Year = {2012},
 Language = {English},
 HowPublished = {Preprint}
 }

@Misc{Zieve2,
 Author = {Carney, A. and Do, T. and Hallett, J. and Sun, Q. and Weiss, B. and Wells, E. and Xia, S.  and Zieve, M. E.},
 Title = {On the functional equation f(u)=g(v) in complex polynomials f,g and meromorphic functions u,v, {I}{I}: the reducible case},
 Year = {2012},
 Language = {English},
 HowPublished = {Preprint}
 }

@article {Tao2,
    AUTHOR = {Tao, T.},
     TITLE = {Expanding polynomials over finite fields of large
              characteristic, and a regularity lemma for definable sets},
   JOURNAL = {Contrib. Discrete Math.},
  FJOURNAL = {Contributions to Discrete Mathematics},
    VOLUME = {10},
      YEAR = {2015},
    NUMBER = {1},
     PAGES = {22--98},
      ISSN = {1715-0868},
   MRCLASS = {11T06 (05C75)},
MRREVIEWER = {Tom\ Sanders},
}

@article {Fried87,
    AUTHOR = {Fried, M. D.},
     TITLE = {Irreducibility results for separated variables equations},
   JOURNAL = {J. Pure Appl. Algebra},
  FJOURNAL = {Journal of Pure and Applied Algebra},
    VOLUME = {48},
      YEAR = {1987},
    NUMBER = {1-2},
     PAGES = {9--22},
      ISSN = {0022-4049,1873-1376},
   MRCLASS = {14A10 (12E99 20C30)},
MRREVIEWER = {Doru\ \c Stef\u anescu},
       DOI = {10.1016/0022-4049(87)90104-6},
       URL = {https://doi-org.technion.idm.oclc.org/10.1016/0022-4049(87)90104-6},
}

@article {KMS,
    AUTHOR = {Kulkarni, M. and Müller, P. and Sury, B.},
     TITLE = {Quadratic factors of {$f(X)-g(Y)$}},
   JOURNAL = {Indag. Math. (N.S.)},
  FJOURNAL = {Koninklijke Nederlandse Akademie van Wetenschappen.
              Indagationes Mathematicae. New Series},
    VOLUME = {18},
      YEAR = {2007},
    NUMBER = {2},
     PAGES = {233--243},
      ISSN = {0019-3577,1872-6100},
   MRCLASS = {12E05 (11C08 12F15 12F20)},
MRREVIEWER = {Amartya\ K.\ Dutta},
       DOI = {10.1016/S0019-3577(07)80019-X},
       URL = {https://doi.org/10.1016/S0019-3577(07)80019-X},
}

@article {Mul4,
    AUTHOR = {Müller, P.},
     TITLE = {Hilbert's irreducibility theorem for prime degree and general
              polynomials},
   JOURNAL = {Israel J. Math.},
  FJOURNAL = {Israel Journal of Mathematics},
    VOLUME = {109},
      YEAR = {1999},
     PAGES = {319--337},
      ISSN = {0021-2172,1565-8511},
   MRCLASS = {12E25},
MRREVIEWER = {N.\ Sankaran},
       DOI = {10.1007/BF02775041},
       URL = {https://doi.org/10.1007/BF02775041},
}

@article {Mul5,
    AUTHOR = {Müller, P.},
     TITLE = {Finiteness results for {H}ilbert's irreducibility theorem},
   JOURNAL = {Ann. Inst. Fourier (Grenoble)},
  FJOURNAL = {Universit\'e{} de Grenoble. Annales de l'Institut Fourier},
    VOLUME = {52},
      YEAR = {2002},
    NUMBER = {4},
     PAGES = {983--1015},
      ISSN = {0373-0956,1777-5310},
   MRCLASS = {12E25 (12E30 14H25 20B25)},
MRREVIEWER = {Boris\ \`E.\ Kunyavski\u i},
       DOI = {10.5802/aif.1907},
       URL = {https://doi.org/10.5802/aif.1907},
}

@article {Pak4,
    AUTHOR = {Pakovich, F.},
     TITLE = {Tame rational functions: decompositions of iterates and orbit
              intersections},
   JOURNAL = {J. Eur. Math. Soc. (JEMS)},
  FJOURNAL = {Journal of the European Mathematical Society (JEMS)},
    VOLUME = {25},
      YEAR = {2023},
    NUMBER = {10},
     PAGES = {3953--3978},
      ISSN = {1435-9855,1435-9863},
   MRCLASS = {37F10 (14E05)},
MRREVIEWER = {Rin\ Gotou},
       DOI = {10.4171/jems/1277},
       URL = {https://doi.org/10.4171/jems/1277},
}

@article{Klein1879,
author = {Klein, F.},
journal = {Mathematische Annalen},
pages = {428-471},
title = {Ueber die {T}ransformation siebenter {O}rdnung der elliptischen {F}unctionen},
url = {http://eudml.org/doc/156835},
volume = {14},
year = {1879},
}

@article{Klein1879b,
author = {Klein, F.},
journal = {Mathematische Annalen},
pages = {533-555},
title = {Ueber die {T}ransformation elfter {O}rdnung der elliptischen {F}unctionen},
url = {http://eudml.org/doc/156871},
volume = {15},
year = {1879},
}

@article{GMS,
  title={The rational function analogue of a question of {S}chur and exceptionality of permutation representations},
  author={Guralnick, R. M. and Müller, P. and Saxl, J.},
  journal={Memoirs of the American Mathematical Society},
  year={2002},
  volume={162},
  pages={0-0},
}

@book{DM96,
  title={Permutation {G}roups},
  author={Dixon, J. D. and Mortimer, B.},
  series={Graduate Texts in Mathematics},
  year={1996},
  publisher={Springer New York}
}

@incollection{Schinzel1963,
  title={Some unsolved problems on polynomials},
  author={Schinzel, A.},
  booktitle={ Neki nereseni problemi u
matematici. Matematicka Biblioteka 25}, 
pages={63--70},
  year={1963}
}

@article{Schinzel1967,
 author = {Schinzel, A.},
 title = {Reducibility of polynominals of the form {{\(f(x)-g(y)\)}}},
 fjournal = {Colloquium Mathematicum},
 journal = {Colloq. Math.},
 issn = {0010-1354},
 volume = {18},
 pages = {213--218},
 year = {1967},
 doi = {10.4064/cm-18-1-213-218},
 zbMATH = {3246404},
 Zbl = {0153.37203}
}

@phdthesis{tv68,
  author       = {Tverberg, H.},
  title        = {A Study in Irreducibility of Polynomials},
  school       = {University of Bergen},
  year         = {1968},
  type         = {Ph.{D}. thesis},
  address      = {Bergen, Norway},
  note         = {Doctoral dissertation}
}

@phdthesis{PHDTali,
author = {Monderer, T.},
title = {Reducibility, Specialization
and Related Low Genus Phenomena},
school = {Technion, Israel Institute of Technology},
year = {2024},
type = {Ph.{D}. thesis},
address = {Haifa, Israel},
note = {Doctoral dissertation}
}

@article{Jones2012,
  title={Newly reducible iterates in families of quadratic polynomials},
  author={K. Chamberlin and E. Colbert and S.M. Frechette and P. Hefferman and R. Jones and S. Orchard},
  journal={Involve, A Journal of Mathematics},
  year={2012},
  volume={5},
  pages={481--495},
 }

@article{Jones2021,
author = {Illig, P. and Jones, R. and Orvis, E. and Segawa, Y. and Spinale, N.},
title = {Newly reducible polynomial iterates},
journal = {International Journal of Number Theory},
volume = {17},
number = {06},
pages = {1405--1427},
year = {2021},
}

@misc{NZ,
title={Monodromy groups of indecomposable coverings of bounded genus}, 
author={Neftin, D. and  Zieve, M.},
year = {2024}, 
eprint={2403.17167}, 
note={arXiv:2403.17167},
      url={https://arxiv.org/abs/2403.17167}, 
}

@article {MV,
    AUTHOR = {M\"uller, Peter and V\"olklein, Helmut},
     TITLE = {On a question of {D}avenport},
   JOURNAL = {J. Number Theory},
  FJOURNAL = {Journal of Number Theory},
    VOLUME = {58},
      YEAR = {1996},
    NUMBER = {1},
     PAGES = {46--54},
      ISSN = {0022-314X,1096-1658},
   MRCLASS = {12F10 (11R09 12E05)},
  MRNUMBER = {1387720},
MRREVIEWER = {Dan\ Haran},
       DOI = {10.1006/jnth.1996.0059},
       URL = {https://doi-org.technion.idm.oclc.org/10.1006/jnth.1996.0059},
}
\bibliographystyle{amsalpha}
\end{document}